\documentclass[12pt]{amsart} 
\usepackage[left=3cm,top=2.5cm,bottom=2.5cm,right=3cm]{geometry}
\usepackage[oldstyle]{libertine}
\usepackage{amssymb, amsmath, amsthm}
\usepackage{aliascnt}
\usepackage{mathtools}
\usepackage{graphicx,xspace}
\usepackage{epsfig}
\usepackage{enumitem}
\usepackage[usenames,dvipsnames]{xcolor}
\usepackage{tikz}
\usepackage[T1]{fontenc}
\usepackage[utf8]{inputenc}
\makeatletter
\renewcommand*\libertine@figurestyle{LF}
\makeatother
\usepackage[libertine,libaltvw,liby]{newtxmath}
\makeatletter
\renewcommand*\libertine@figurestyle{OsF}
\makeatother
\usepackage{subcaption}
\definecolor{green}{RGB}{0,127,0}
\definecolor{red}{RGB}{191,0,0}
\usepackage[colorlinks,cite color=red,link color=green,pagebackref=true]{hyperref}
\usepackage{todonotes}
\usepackage[capitalize]{cleveref}
\usepackage{seqsplit}
\usepackage{comment}
\usetikzlibrary{matrix, arrows, calc, decorations.markings, patterns, shapes}
\theoremstyle{plain}
\newtheorem{lemma}{Lemma}[section]
\newaliascnt{theorem}{lemma}
\newtheorem{theorem}[theorem]{Theorem}
\aliascntresetthe{theorem}

\newaliascnt{theoremdefinition}{lemma}
\newtheorem{theoremdefinition}[theoremdefinition]{Theorem-Definition}
\aliascntresetthe{theoremdefinition}

\newaliascnt{corollary}{lemma}
\newtheorem{corollary}[corollary]{Corollary}
\aliascntresetthe{corollary}

\newaliascnt{proposition}{lemma}
\newtheorem{proposition}[proposition]{Proposition}
\aliascntresetthe{proposition}

\newaliascnt{problem}{lemma}

\aliascntresetthe{problem}

\newaliascnt{question}{lemma}

\aliascntresetthe{question}

\newaliascnt{conjecture}{lemma}

\aliascntresetthe{conjecture}

\newaliascnt{definition}{lemma}
\newtheorem{definition}[definition]{Definition}
\aliascntresetthe{definition}

\newaliascnt{assumption}{lemma}

\aliascntresetthe{assumption}

\newaliascnt{claim}{lemma}
\newtheorem{claim}[claim]{Claim}
\aliascntresetthe{claim}

\theoremstyle{remark}
\newtheorem{remark}{Remark}
\newtheorem{example}{Example}

\crefname{theorem}{theorem}{theorems}
\Crefname{theorem}{Theorem}{Theorems}
\crefname{lemma}{lemma}{lemmas}
\Crefname{lemma}{Lemma}{Lemmas}
\crefname{corollary}{corollary}{corollaries}
\Crefname{corollary}{Corollary}{Corollaries}
\crefname{proposition}{proposition}{propositions}
\Crefname{proposition}{Proposition}{Propositions}
\crefname{problem}{problem}{problems}
\Crefname{problem}{Problem}{Problems}
\crefname{question}{question}{questions}
\Crefname{question}{Question}{Questions}
\crefname{conjecture}{conjecture}{conjectures}
\Crefname{conjecture}{Conjecture}{Conjectures}
\crefname{definition}{definition}{definitions}
\Crefname{definition}{Definition}{Definitions}
\crefname{assumption}{assumption}{assumptions}
\Crefname{assumption}{Assumption}{Assumptions}
\crefname{claim}{claim}{claims}
\Crefname{claim}{Claim}{Claims}
\crefname{theoremdefinition}{theorem-definition}{theorem-definitions}
\Crefname{theoremdefinition}{Theorem-Definition}{Theorem-Definitions}
\crefname{remark}{remark}{remarks}
\Crefname{remark}{Remark}{Remarks}
\crefname{example}{example}{examples}
\Crefname{example}{Example}{Examples}

\AtBeginDocument{\let\cref\Cref}

\DeclareMathOperator{\Sym}{Sym}
\DeclareMathOperator{\Alpha}{Alpha}
\DeclareMathOperator{\Beta}{Beta}
\DeclareMathOperator{\Planch}{Planch}
\DeclareMathOperator{\dd}{d}

\DeclareMathOperator{\wt}{wt}
\DeclareMathOperator{\topp}{top}
\DeclareMathOperator{\GL}{GL}
\DeclareMathOperator{\Or}{O}
\DeclareMathOperator{\U}{U}

\newcommand{\GT}{\mathbb{S}}

\newcommand{\C}{\mathbb{C}}
\newcommand{\R}{\mathbb{R}}
\newcommand{\Q}{\mathbb{Q}}
\newcommand{\Z}{\mathbb{Z}}
\newcommand{\N}{\mathbb{N}}
\newcommand{\HH}{\mathbb{H}}

\newcommand{\PP}{\mathbb{P}}

\newcommand{\E}{\mathbb{E}}
\newcommand{\Y}{\mathbb{Y}}
\newcommand{\LL}{\mathscr{L}}
\newcommand{\ev}{\text{ev}}
\newcommand{\xx}{\mathbf{x}}
\newcommand{\yy}{\mathbf{y}}
\newcommand{\Luk}{\mathbf{L}}
\newcommand{\la}{\lambda}
\newcommand{\Prob}{\mathbb P}
\newcommand{\Sy}[1]{\mathfrak{S}_{#1}}

\title[$N$-particle systems at high temperature]{Discrete $N$-particle systems at high temperature through Jack generating functions}

\author[C.~Cuenca]{Cesar Cuenca}
\address{
Department of Mathematics,
The Ohio State University,
231 West 18th Avenue,
Columbus, OH 43210, USA.
}
\email{cesar.a.cuenk@gmail.com}

\author[M.~Dołęga]{Maciej Dołęga}
\address{
Institute of Mathematics, 
Polish Academy of Sciences, 
ul. Śniadeckich 8, 
00-956 Warszawa, Poland.
}
\email{mdolega@impan.pl}

\thanks{ CC was partially supported by the NSF grant DMS-2348139 and by the Simons Foundation's Travel Support for Mathematicians grant MP-TSM-00006777. This research was funded in whole or in part by {\it Narodowe Centrum Nauki}, grant 2021/42/E/ST1/00162. For the purpose of Open Access, the author has applied a CC-BY public copyright licence to any Author Accepted Manuscript (AAM) version arising from this submission.}

\begin{document}
\emergencystretch 3em

\begin{abstract}
We find necessary and sufficient conditions for the Law of Large Numbers for random discrete $N$-particle systems with the deformation (inverse temperature) parameter $\theta$, as their size $N$ tends to infinity simultaneously with the inverse temperature going to zero.
Our conditions are expressed in terms of the Jack generating functions, and our analysis is based on the asymptotics of the action of Cherednik operators obtained via Hecke relations.
We apply the general framework to obtain the LLN for a large class of Markov chains of $N$ nonintersecting particles with interaction of log-gas type, and the LLN for the multiplication of Jack polynomials, as the inverse temperature tends to zero. We express the answer in terms of novel one-parameter deformations of cumulants and our description recovers previous work by Bufetov--Gorin~\cite{BufetovGorin2015b} on quantized free cumulants when $\theta=1$, and by Benaych-Georges--Cuenca--Gorin~\cite{Benaych-GeorgesCuencaGorin2022} after a deformation to continuous space of random matrix eigenvalues.
Our methods are robust enough to be applied to the fixed temperature regime, where we recover the LLN of Huang~\cite{Huang2021}.
\end{abstract}

\maketitle

\tableofcontents

\section{Introduction}

\subsection{Overview}

Motivated by classical problems in statistical physics and their deep connections with random matrix theory, understanding the large $N$ global asymptotics -- that is, the limiting behavior of empirical measures and linear statistics -- of various random $N$-particle systems $x_1 < \dots < x_N$ on the real line, at inverse temperature $\beta > 0$, has been a central question in probability theory for the last few decades. When the parameter $\beta$ is one of the values $1, 2, 4$, the random $N$-particle systems usually correspond to the spectra of self-adjoint random matrices. The first insight into their global asymptotics traces back to Wigner's discovery~\cite{Wigner1958} of the semicircle distribution as the limiting empirical measure for the Gaussian Unitary Ensemble, which corresponds to a specific model at $\beta=2$. In many cases, it has been shown that the global asymptotics of these random $N$-particle systems exhibit the following universality: they display the same behavior regardless of the actual value of $\beta$, as long as it remains fixed. Another important discovery, relevant to us, is due to Voiculescu~\cite{Voiculescu1986,Voiculescu1991}, who showed that the behavior of large independent random matrices can be captured by freely independent random variables. In particular, the spectrum of the sum of large independent uniformly-rotated matrices with fixed eigenvalues is described by the free convolution of probability measures.

In the limiting case $\beta=0$, the interaction between random particles disappears and they become independent; then the classical Law of Large Numbers describes the system when the number of particles $N$ tends to infinity.
Moreover, in the regime where $\beta$ tends to zero at the same time as $N$ grows, one can observe phenomena interpolating between classical and free probability.
This middle ground is known as the \emph{high temperature regime} and occurs when $\beta N$ tends to a positive constant $\gamma$.
The high temperature LLN for the Hermite and Laguerre $\beta$-ensembles, as well as the $\gamma$-deformation of the semicircle and Marchenko-Pastur distributions, were studied in~\cite{AllezBouchaudGuionnet2012,AllezBouchaudMajumdarVivo2013,DuyShirai2015,TrinhTrinh2021}. More general ensembles and $\gamma$-dependent limits, including $\beta$-sums of random matrices, were studied in~\cite{Benaych-GeorgesCuencaGorin2022} (see also~\cite{MergnyPotters2022}), where the authors provided necessary and sufficient conditions for the high temperature LLN. High temperature limits of $\beta$-ensembles have also appeared in a different but closely related context, initiated by Spohn's work on generalized Gibbs ensembles of the classical Toda chain~\cite{Spohn2020}, and have since motivated further asymptotic results for $\beta$-ensembles and integrable systems~\cite{HardyLambert2021,GuionnetMemin2022,MazzucaMemin2025}.

To the best of our knowledge, and in sharp contrast with the continuous case, the high temperature regime of the global asymptotics of discrete random $N$-particle systems at inverse temperature $\beta > 0$ has not been studied in the literature. Other discrete models of $\beta$-ensembles, however, have only recently been studied in the high and low-temperature regimes, see \cite{DolegaSniady2019, Moll2023,CuencaDolegaMoll2026}, but the case of discrete random $N$-particle systems was the remaining case to be understood. In this paper, we fill this gap by providing necessary and sufficient conditions for the LLN for a very general class of discrete random $N$-particle systems in the high temperature regime.

\subsection{Main results}\label{sec:main_results_intro}

Motivated by the previous work on discrete random $N$-particle systems at inverse temperature $\beta > 0$ (see~\cite{GorinShkolnikov2015,BorodinGorinGuionnet2017,Huang2021} and references therein), we consider the following model. Let $\{\Prob_N\}_{N\ge 1}$ be a sequence of finite signed measures on the set of $N$-tuples $(\LL_1>\dots>\LL_N)$ of the form $\LL_i := \la_i - (i-1)\theta$, for $i=1,2,\dots,N$, where
\begin{equation*}
\la = (\la_1,\dots,\la_N) \in \GT(N) := \{ \la = (\la_1\ge\cdots\ge\la_N)\in\Z^N \}
\end{equation*}
belongs to the set of highest weights of irreducible $U(N)$-representations, also known as $N$-signatures.
Clearly, $\Prob_N$ can also be interpreted as a finite signed measure on the set of $N$-signatures, and we will use this interpretation. At this point, the measure $\Prob_N$ is as general as it can be, and in particular it might\footnote{It is more interesting when it does depend on the parameter $\theta$, and in the examples we will consider later this is always the case. In the context of the discrete log-gas model, the parameter $\theta$ is related to the parameter $\beta$ by $2\theta = \beta$. Moreover, most of our applications concern probability measures, but our main result is more general and applies to finite signed measures.} or might not depend on the parameter $\theta$. We are interested in the behavior of the empirical measures
\begin{equation*}
\mu_N := \frac{1}{N}\sum_{i=1}^N\delta_{\LL_i},\text{ where }(\LL_1>\dots>\LL_N)\text{ is }\PP_N\text{--distributed}.
\end{equation*}

In this paper, we find explicit necessary and sufficient conditions for the empirical measures $\mu_N$ to converge (either weakly, or in the sense of moments) to some probability measure, as $N\to\infty$, $N\theta\to\gamma\in\R_{>0}$.
The conditions are expressed in terms of certain discrete Fourier-type transforms of $\Prob_N$, called the \emph{Jack generating functions}. 

To motivate the definition, let us mention at this point that in the special case $\theta = 1$, the underlying $N$-particles can be treated as random highest weights of irreducible $U(N)$-representations. The Fourier transform on the unitary group $U(N)$ that allows one to control natural representation-theoretic operations as $N \to \infty$ is given by the so-called~\emph{Schur generating function}, which turned out to be a very convenient tool to approach many interesting classes of interacting particles beyond representation theoretic models, for instance random tilings (see e.g.~\cite{BufetovGorin2015b, BufetovGorin2018, BufetovGorin2019}). The Fourier transform on the unitary group $U(N)$ can be naturally interpreted as the Fourier transform on the Gelfand pair $(\GL_N(\C),\U_N(\C))$, and by switching to classical (skew) fields like $\R$ or the quaternions $\HH$, one can study other $N$-particle systems of representation-theoretic origin corresponding to the Gelfand pairs $(\GL_N(\R),\Or_N(\R))$ and $(\GL_N(\HH),\U_N(\HH))$, through their Fourier transforms. The Jack generating function is equipped with a continuous parameter $\theta>0$, and interpolates between these three Fourier transforms at $\theta = 1,1/2,2$, respectively. Our main result explains that the germ at unity of the Jack generating function contains complete information about the Law of Large Numbers not only for a fixed parameter $\theta>0$, which was proven previously in~\cite{Huang2021}, but also in the high temperature regime.

Another reason for the significance of the Jack generating function is that it can be explicitly computed for a natural subclass of Jack measures (see \cref{def:general_jack_measure}).
The Jack measures are one-parameter generalizations of the Schur measures~\cite{Okounkov2001} and have been studied in connection to asymptotic representation theory~\cite{BorodinOlshanski2005,CuencaDolegaMoll2026}, and Markov processes with interaction of log-gas type~\cite{GorinShkolnikov2015,Huang2021}.
They generally depend on two specializations $\rho,\rho'$ of the algebra of symmetric functions $\Sym$, but if one of them is $\rho'=(1^N)$, then the corresponding Jack measures are supported on the set $\Y(N) = \{ \la=(\la_1\ge\dots\ge\la_N)\in(\Z_{\ge 0})^N \}$ of partitions with at most $N$ parts.
Natural choices for the other specialization $\rho$ are parametrized by points belonging to the Thoma cone (see \cref{thm:KOO} for the details), and we are able to prove the LLN at high temperature due to the fact that the Jack generating functions have explicit factorized forms and the conditions for LLN are easily verified. We will further motivate the definition of the Jack generating function in the following section, where we discuss applications of our criterion for the LLN.

\smallskip

Let $P_\la(x_1,\dots,x_N;\theta)$ be the Jack symmetric polynomial\footnote{See \cref{sec:jacks} for a review of the necessary background of the theory of Jack symmetric polynomials.} that depends on the Jack parameter $\theta>0$. We associate to any measure $\Prob_N$ on $\GT(N)$ the following formal multivariate symmetric Laurent series, to be called the \emph{Jack generating function of $\Prob_N$}:
\begin{equation}\label{eq:intro_JGF}
G_{\Prob_N,\theta}(x_1,\dots,x_N) := \sum_{\la\in\GT(N)}{ \Prob_N(\la)\frac{P_\la(x_1,\dots,x_N;\theta)}{P_\la(1^N;\theta)} }.
\end{equation}

Our main result (\cref{theo:main1} in the text) establishes necessary and sufficient conditions for the Law of Large Numbers of particles $(\LL_1>\dots>\LL_N)$ associated with a sequence of finite signed measures $\{\PP_N\}_{N\ge 1}$ on $N$-signatures in the high temperature regime:
\begin{equation*}
N\to\infty,\quad\theta\to 0,\quad N\theta\to\gamma\in (0,\infty).
\end{equation*}
For simplicity, we break down this result into the next two theorems in this introduction. Informally speaking, the first part states that all the information about the LLN of the empirical measures in the high temperature regime is encoded in the germ of the Jack generating function at the point $(1^N)$ in this regime. 

\begin{theorem}[First part of \cref{theo:main1} in the text]\label{thm:intro_1}
Let $\{\PP_N\}_{N\ge 1}$ be a sequence of finite signed measures on $N$-signatures (that satisfies a natural condition of \cref{def:small_tails} on the growth rate of $G_{\Prob_N,\theta}$) and let $\{\mu_N\}_{N\ge 1}$ be their empirical measures.
Then $\mu_N$ converge in the sense of moments, in the high temperature regime, i.e. 
there exist $m_1,m_2,\dots\in\R$ such that, for all $s\in\Z_{\ge 1}$ and $k_1,\dots,k_s\in\Z_{\ge 1}$, we have
\begin{equation*}
\lim_{\substack{N\to\infty\\ N\theta \to \gamma}}\frac{1}{N^s}\,\E_{\PP_N} \left[ \prod_{j=1}^s{ \sum_{i=1}^N{\LL_i^{k_j}} } \right]
= \prod_{j=1}^s{m_{k_j}},
\end{equation*}
if and only if the following two conditions are satisfied, for some constants $\kappa_1^\gamma,\kappa_2^\gamma,\cdots$:

\vspace{.1cm}

\begin{enumerate}
	\item $\displaystyle\lim_{\substack{N\to\infty\\N\theta\to\gamma}}\frac{1}{(n-1)!}\frac{\partial^n}{\partial x_1^n} \ln\big( G_{\Prob_N,\theta}\big) \Big|_{(x_1,\dots,x_N) = (1^N)} = \kappa_n^\gamma$ exists and is finite, for all $n \in \Z_{\geq 1}$,

	\item $\displaystyle\lim_{\substack{N\to\infty\\N\theta\to\gamma}} \frac{\partial^r}{\partial x_{i_1}\cdots \partial x_{i_r}} \ln\big( G_{\Prob_N,\theta}\big) \Big|_{(x_1,\dots,x_N) = (1^N)} = 0$, for all $r\ge 2$ and $i_1,\dots,i_r\in\Z_{\geq 1}$ with at least two distinct indices among $i_1,\dots,i_r$.
\end{enumerate}
Moreover, if $\PP_N$ are probability measures and $\sup_{n\ge 1}\big|\kappa_n^\gamma\big|^{1/n}<\infty$, then $\mu_N$ converge in the sense of moments, in probability, to a probability measure $\mu_\gamma$ uniquely determined by its moments.
\end{theorem}

The values $\kappa_1^\gamma,\kappa_2^\gamma,\dots$ from part~(1) of \cref{thm:intro_1} will be called the \emph{quantized $\gamma$-cumulants} associated with the limiting moment sequence. They uniquely determine and, at the same time, are determined by the limiting moments $m_1,m_2,\dots$. Whenever these moments determine a probability measure $\mu_\gamma$, we also call them the quantized $\gamma$-cumulants of $\mu_\gamma$.
The term quantized $\gamma$-cumulant is used because, upon certain normalization followed by the limit $\gamma\to\infty$, the quantities $\kappa_n^\gamma$ turn into the coefficients of the quantized $R$-transform from~\cite{BufetovGorin2015b}, as shown in \cref{sec:comparison}.

The second part of our main theorem provides a combinatorial transform $\mathcal{J}_\gamma^{\kappa\mapsto m}\colon\R^\infty\to\R^\infty$ for calculating the moments of the limiting measure in terms of the quantized $\gamma$-cumulants. It refines several classical and recently discovered moment-cumulant-type formulas and fits into the universal class of combinatorial formulas expressed as the weighted generating function of Łukasiewicz lattice paths (see~\cref{subsec:methods} for references on these other works). The exact relation is the following.

\begin{theorem}[Second part of \cref{theo:main1} in the text]\label{thm:intro_2}
Assume that we are in the setting of \cref{thm:intro_1}.
If $(m_\ell)_{\ell\ge 1}$ is the limiting moment sequence from \cref{thm:intro_1}, then
\begin{multline}\label{eq:intro_moms_cums}
m_\ell =  \sum_{\Gamma\in\Luk(\ell)}
\frac{\Delta_\gamma\left(x^{1+\text{\#\,horizontal steps at height $0$ in $\Gamma$}}\right)\big|_{x=\kappa_1^\gamma}}{1+\text{\#\,horizontal steps at height $0$ in $\Gamma$}}
\cdot\prod_{i\ge 1}\big(\kappa_1^\gamma + i\big)^{\text{\#\,horizontal steps at height $i$ in $\Gamma$}}\\
\cdot\prod_{j\ge 1}\big(\kappa_j^\gamma + \kappa_{j+1}^\gamma\big)^{\text{\#\,steps $(1,j)$ in $\Gamma$}}(j+\gamma)^{\text{\#\,down steps from height $j$ in $\Gamma$}},
\end{multline}
for all $\ell\in\Z_{\ge 1}$, where $\Luk(\ell)$ is the set of Łukasiewicz paths of length $\ell$ (certain lattice paths in $\Z^2$ whose definition is recalled in \cref{def:lukasiewicz}), and $\Delta_\gamma(f)(x) := \frac{1}{\gamma}(f(x)-f(x-\gamma))$.
\end{theorem}

Note that the sequence of the quantized $\gamma$-cumulants $\vec{\kappa^\gamma} := (\kappa_1^\gamma,\kappa_2^\gamma,\dots)$ in the formula~\eqref{eq:intro_moms_cums} can be replaced by an arbitrary sequence $\vec{\kappa} := (\kappa_1,\kappa_2,\dots)$ of real numbers, and~\eqref{eq:intro_moms_cums} will produce a sequence $\vec{m} := (m_1,m_2,\dots)$ of real numbers. Thus the combinatorial formula~\eqref{eq:intro_moms_cums} defines a map $\mathcal{J}_\gamma^{\kappa\mapsto m}\colon\R^\infty\to\R^\infty; \vec{\kappa}\mapsto\vec{m}$ that we will refer to as a \emph{combinatorial transform} (see~\cref{def:mk}).
In the following, we discuss various applications of our results.

\subsection{Applications}

\subsubsection{Quantized $\gamma$-convolution and the Jack Littlewood--Richardson coefficients}

Recall that classical cumulants linearize the classical convolution of two probability measures. Similarly, free cumulants linearize the free convolution of probability measures, which, as shown by Voiculescu in his seminal papers~\cite{Voiculescu1986,Voiculescu1991}, describes the spectrum of the sum of large independent uniformly-rotated matrices. An analogous result was shown by Bufetov and Gorin~\cite{BufetovGorin2015b}, who demonstrated that their quantized $R$-transform linearizes the operation of decomposing the tensor product of two random irreducible representations of the unitary group $U(N)$ into irreducible components, and they established the LLN for this operation as $N \to \infty$. We will show that the quantized $\gamma$-deformed cumulants linearize the Jack-deformed version of the aforementioned operation of tensoring representations, and we will prove the associated LLN in the high temperature regime. Interestingly, we will show that the existence of quantized $\gamma$-convolution for probability measures is closely related to the long-standing open problem posed by Stanley~\cite{Stanley1989} regarding the structure constants associated to the Jack symmetric functions.

Let $\{\Prob_N^{(1)}\}_{N \ge 1}$ and $\{\Prob_N^{(2)}\}_{N \ge 1}$ be sequences of probability measures on $N$-signatures that satisfy the growth-rate condition from \cref{def:small_tails}, and such that (see \cref{sec:app_1} for more details):
\begin{itemize}

  \item Their empirical measures $\big\{\mu_N^{(1)}\big\}_{N \ge 1}$ and $\big\{\mu_N^{(2)}\big\}_{N \ge 1}$ are supported on the same compact set in the high temperature regime.
  
  \item $\mu_N^{(1)} \to \mu^{(1)}$ and $\mu_N^{(2)} \to \mu^{(2)}$ as $N \to \infty$, weakly, in probability.

\end{itemize}
Let $\kappa_n^\gamma[\mu^{(1)}]$ and $\kappa_n^\gamma[\mu^{(2)}]$ be the quantized $\gamma$-cumulants of $\mu^{(1)}$ and $\mu^{(2)}$, respectively. We are interested in whether the sequence of sums
\begin{equation*}
\tilde{\kappa}_n^\gamma := \kappa_n^\gamma[\mu^{(1)}] + \kappa_n^\gamma[\mu^{(2)}], \quad n \in \Z_{\ge 1},
\end{equation*}
and the associated sequence $\tilde{m}_1, \tilde{m}_2, \dots$ of ``moments'' obtained from $\big(\tilde{\kappa}_n^\gamma\big)_{n \ge 1}$ by the combinatorial formula \eqref{eq:intro_moms_cums} correspond to some measure that arises as a limit of certain natural ``convolution'' in the high temperature regime.

Define the finite signed measures $\Prob_N^{(1)}\boxplus_\theta\Prob_N^{(2)}$ on $N$-signatures $\GT(N)$ by
\[
\big(\Prob_N^{(1)}\boxplus_\theta\Prob_N^{(2)}\big)(\la) := \sum_{\mu,\,\nu\in\Y(N)}c^\la_{\mu,\,\nu}(\theta)\Prob_N^{(1)}(\mu)\Prob_N^{(2)}(\nu),\qquad\la\in\GT(N),
\]
where $c^\la_{\mu,\,\nu}(\theta)$ are the Jack Littlewood--Richardson coefficients, defined by \cref{eq:multiplication_jacks}.
Equivalently, $\Prob_N^{(1)}\boxplus_\theta\Prob_N^{(2)}$ is the unique measure on $\GT(N)$ whose Jack generating function is the product of Jack generating functions of $\PP_N^{(1)}$ and $\PP_N^{(2)}$.
We prove in \cref{sec:app_1} that $\E_{\Prob_N^{(1)}\boxplus_\theta\Prob_N^{(2)}}\mu_N$ converges in the sense of moments, in the high temperature regime, and the limiting moment sequence is $\tilde{m}_1,\tilde{m}_2,\cdots$. Moreover, if we additionally knew that $\Prob_N^{(1)}\boxplus_\theta\Prob_N^{(2)}$ is a sequence of probability measures, then we would have a natural operation of \emph{quantized $\gamma$-convolution} that produces a probability measure $\mu^{(1)} \boxplus^{(\gamma)} \mu^{(2)}$ uniquely determined by the property that its quantized $\gamma$-cumulants are:
\begin{equation*}
\kappa_n^\gamma\big[ \mu^{(1)} \boxplus^{(\gamma)} \mu^{(2)} \big]
= \kappa_n^\gamma\big[ \mu^{(1)} \big] + \kappa_n^\gamma\big[ \mu^{(2)} \big],\quad\text{for all }n\in\Z_{\ge 1}.
\end{equation*}
In that case, the empirical measures $\mu_N$ of $\Prob_N^{(1)}\boxplus_\theta\Prob_N^{(2)}$ would converge in moments, in probability to the convolution $\mu^{(1)} \boxplus^{(\gamma)} \mu^{(2)}$.

One of the celebrated conjectures of Stanley~\cite{Stanley1989} states that the Jack Littlewood--Richardson coefficients are (up to a proper normalization) polynomials in $\theta$ with nonnegative integer coefficients. This conjecture is much stronger than the assumption that $\Prob_N^{(1)}\boxplus_\theta\Prob_N^{(2)}$ are probability measures. Even though Stanley's conjecture is still wide open, we will actually show the existence of the quantized $\gamma$-convolution for a large class of measures.
In fact, next we demonstrate that it naturally appears in the study of global asymptotics of the law of certain Markov chains on $N$ nonintersecting particles with log-gas type interaction in the high temperature regime.

\subsubsection{Nonintersecting particle systems in the high temperature regime through Jack measures}\label{sec:pure_jack_measures}

In this section, we assume that we have a sequence of Jack parameters $\theta_N>0$ such that $N\theta_N\to\gamma$, as $N\to\infty$, for some $\gamma>0$, as well as a sequence of specializations $\rho_N\colon\Sym\to\C$ such that $P_\la\big(\rho_N;\theta_N\big)\ge 0$, for all $\la\in\Y(N)$ and all $N\in\Z_{\ge 1}$.
If the condition given in~\cref{def:stable} is satisfied, then there exist unique probability measures with weights proportional to
\begin{equation*}
\Prob_N(\la)\propto Q_\la\big(\rho_N;\theta_N\big)P_\la\big(1^N;\theta_N\big),\quad\la\in\Y(N),
\end{equation*}
where $Q_\la$ is defined in~\eqref{j_lambda}.

These measures comprise a distinguished class of \emph{Jack measures}, see e.g.~\cite{BorodinOlshanski2005,CuencaDolegaMoll2026,GorinShkolnikov2015}. It is known that the specializations that are nonnegative on Jack symmetric functions are parametrized by points of the form $(\alpha,\beta,\delta)\in(\R_{\ge 0})^\infty\times(\R_{\ge 0})^\infty\times\R_{\ge 0}$ in the infinite-dimensional Thoma cone (see \cref{thm:KOO} for the precise statement).
When $\alpha_i=\beta_i=0$, for all $i$, and only $\delta$ is nonzero, the associated Jack measure (called the \emph{pure-Plancherel Jack measures} in \cref{def:example3}) was used by Gorin--Shkolnikov~\cite{GorinShkolnikov2015} to construct a continuous-time Markov process on a discrete space of $N$-tuples of particles, which has a diffusive limit to the $\theta$-Dyson Brownian motion, and which can be interpreted as a one-parameter deformation of a nonintersecting $N$-particle version of the Poisson random walk on $\Z$ (see~\cref{fig:Markov_chain} for a sample of this process; the marginals of this process at positive times $t$ are given by the pure-Plancherel Jack measures with the parameter $\delta = t$).
Based on~\cite{GorinShkolnikov2015}, Huang defined in~\cite{Huang2021} a large class of discrete Markov chains $\LL^{(0)}(N),\LL^{(1)}(N),\dots$, of $N$ nonintersecting discrete particles with log-gas type interaction, by replacing the pure Plancherel specialization with more general Jack-positive specializations (see~\cref{def:stable,def:Markov_chain}), and proved the LLN at a fixed parameter $\theta$ for the joint law at appropriately chosen sequences of ``times'' $\{t_N\in\Z_{\ge 1}\}_{N\ge 1}$.
We will generalize these results by providing an analogous result in the high temperature regime.
Moreover, we will show that the limiting measure can be naturally expressed as the quantized $\gamma$-convolution of two probability measures.

\begin{theorem}[\cref{theo:LLN_Markov} in the text]
Let $\rho=\{\rho_N\}_{N\ge 1}$ be a sequence of stable Jack-positive specializations. Under a natural assumption on the limiting behavior (see~\eqref{eqn:limit_trho}) of the sequence of ``times'' $\{t_N\in\Z_{\ge 1}\}_{N\ge 1}$, we have that:

(1) The sequence of Jack measures $\PP_{t_N\rho_N;N}^{(\theta_N)}$ satisfies the LLN. If, in addition, the limiting quantized $\gamma$-cumulants satisfy the exponential growth condition~\eqref{eq:rho_growth}, then the associated empirical measures converge weakly, in probability, to a unique probability measure $\mu_\rho$.

(2) Additionally, let $\mu^{(0)}_N$ be the empirical distributions of the law of the initial configuration $\LL^{(0)}(N)$ in the Markov chain $\LL^{(0)}(N),\LL^{(1)}(N),\LL^{(2)}(N),\dots$, with a compactly supported weak limit $\mu^{(0)}$ (in probability, as $N \to \infty$). Then the empirical measure of the law of $\LL^{(t_N)}(N)$ satisfies the LLN with quantized $\gamma$-cumulants equal to the sums of the quantized $\gamma$-cumulants of $\mu^{(0)}$ and the limiting quantized $\gamma$-cumulants from part~(1). If the growth condition~\eqref{eq:rho_growth} is satisfied, then these empirical measures converge weakly, in probability, to the quantized $\gamma$-convolution $\mu^{(0)} \boxplus^{(\gamma)}\mu_\rho$.
\end{theorem}

We specifically discuss the high temperature LLN above for the three kinds of ``pure specializations'', namely, the case when each $\rho_N$ depends only on finitely many parameters of the same kind ($\alpha_i$'s, or $\beta_i$'s, or $\delta$). The associated Markov chains can be interpreted as one-parameter deformations of nonintersecting $N$-particle versions of the geometric, Bernoulli, and Poisson random walks on $\Z$.
While the LLN in the generic, but fixed $\theta$ regime, proved by Huang~\cite{Huang2021}, did not depend on the parameter $\theta$ after the appropriate scaling, we will see that in the high temperature regime the limiting law depends on the parameter $\gamma=\lim_N{N\theta_N}$ describing the high temperature scaling (compare the left with the middle and the right pictures in~\cref{fig:Markov_chain}).

\begin{figure}[htbp]
  \centering
  \begin{subfigure}[b]{0.325\textwidth}
    \centering
    \includegraphics[width=\textwidth]{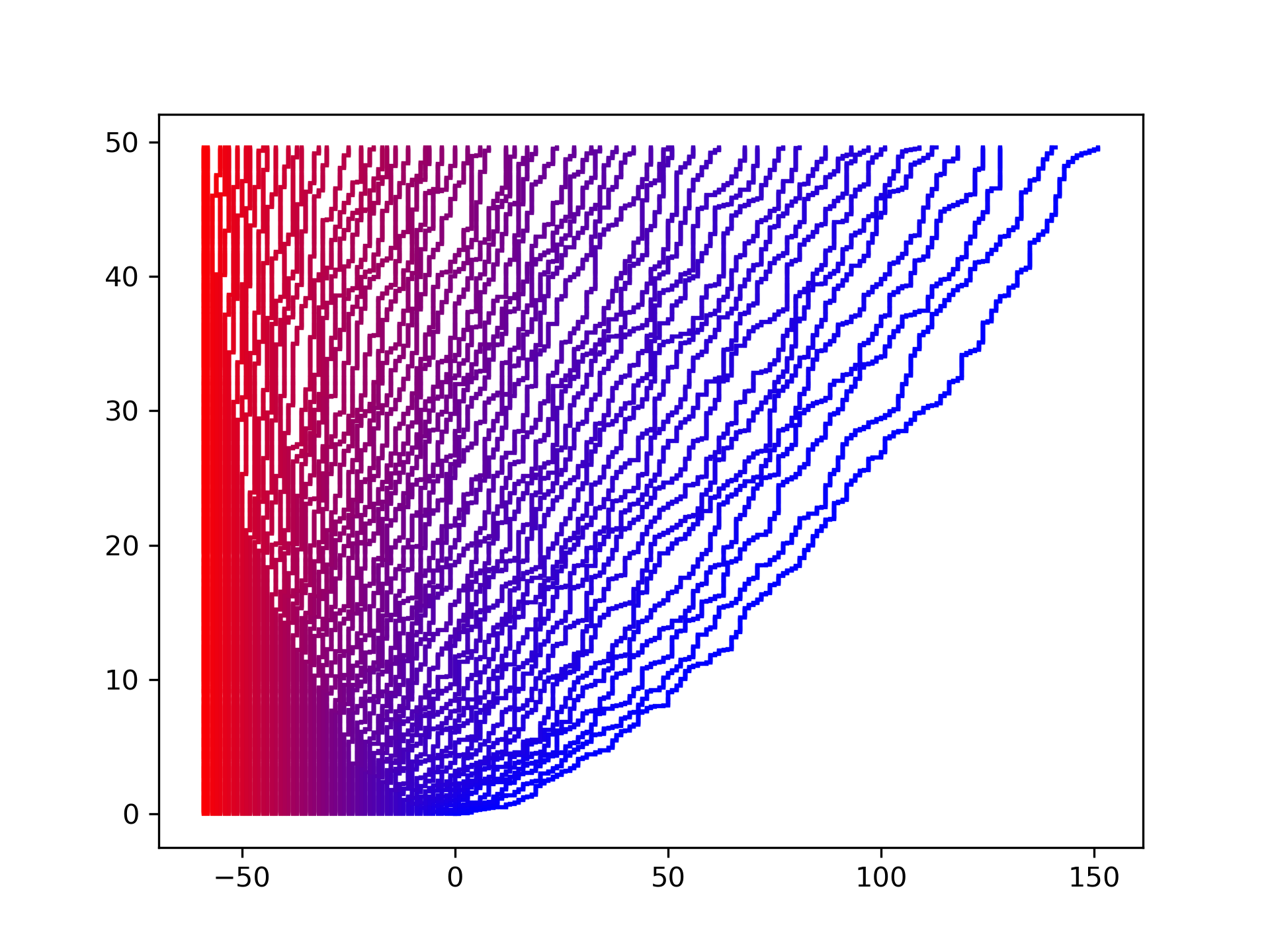}
    \caption{$\theta=1, N=60$}
    \label{fig:left}
  \end{subfigure}
  \hfill
  \begin{subfigure}[b]{0.325\textwidth}
    \centering
    \includegraphics[width=\textwidth]{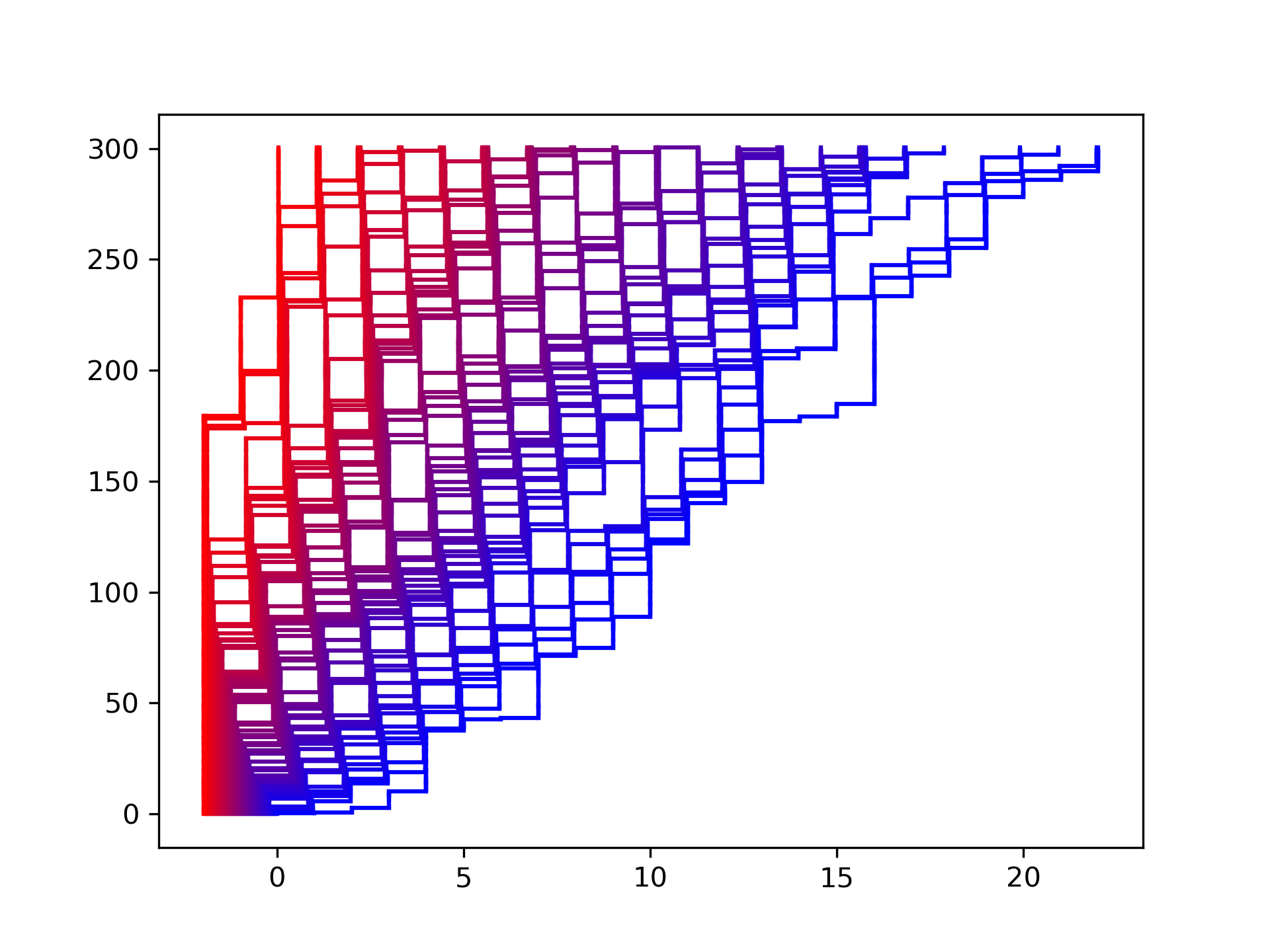}
    \caption{$\theta=\frac{2}{N}, N=60$}
    \label{fig:center}
  \end{subfigure}
  \hfill
  \begin{subfigure}[b]{0.325\textwidth}
    \centering
    \includegraphics[width=\textwidth]{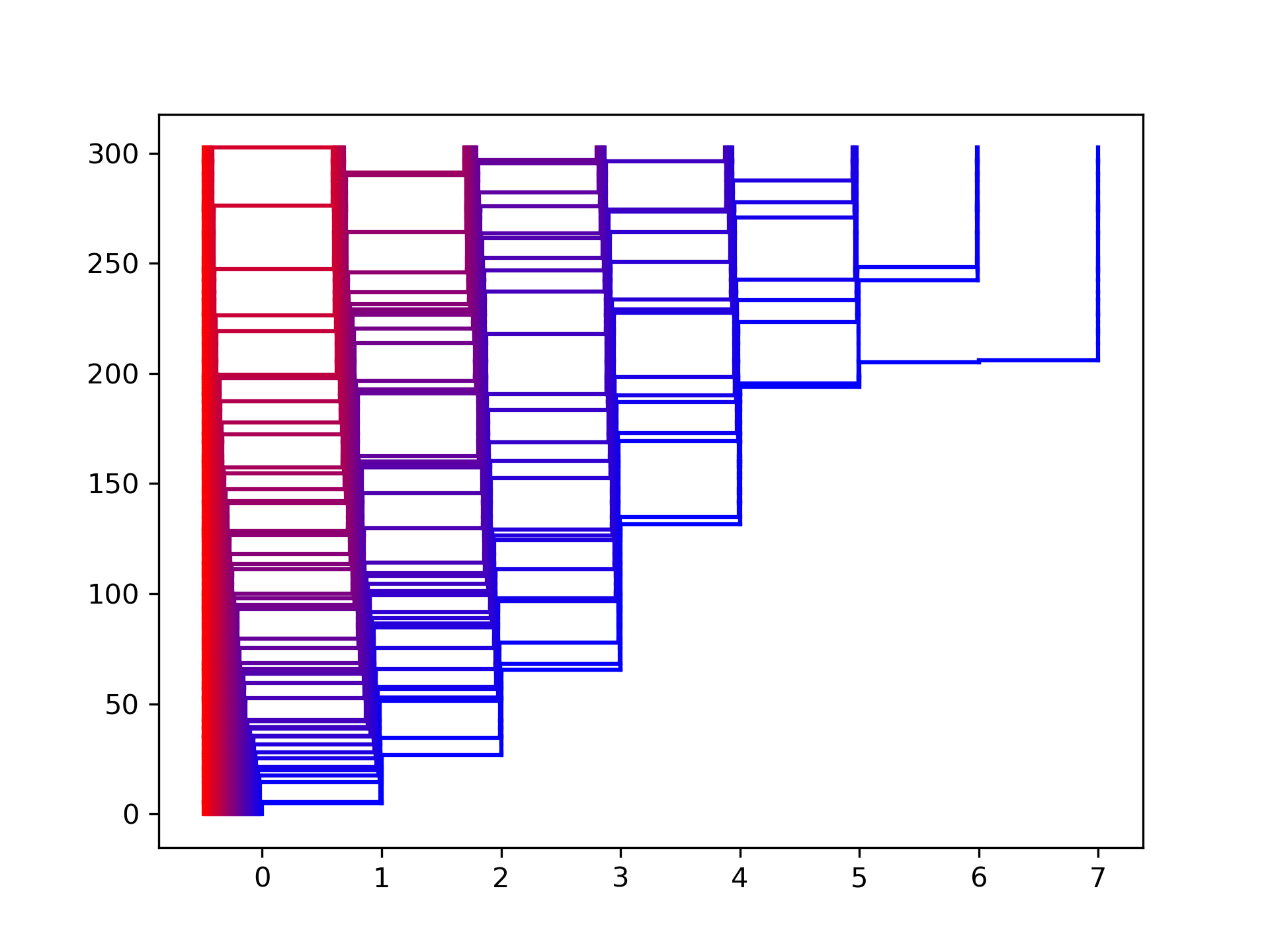}
    \caption{$\theta=\frac{1}{2N}, N=60$}
    \label{fig:right}
  \end{subfigure}
  \caption{Three simulations of the Gorin--Shkolnikov process with $N=60$ particles and initial configuration $\LL^{(0)}_i = \theta(1-i)$, for $1 \leq i \leq N$. The simulations show the asymptotic behavior in the fixed temperature regime on the left (with $\theta=1$), and in the high temperature regime with $\gamma = 2$ in the middle, and with $\gamma = \frac{1}{2}$ on the right. The $x$-axis represents the space and the $y$-axis represents the time; the marginals of this process at positive times $t$ are given by the pure-Plancherel Jack measures with the parameter $\delta = t$.}
  \label{fig:Markov_chain}
\end{figure}

Let us also mention that the pure-Plancherel Jack measure is a discrete $\beta$-ensemble, in the sense of~\cite{BorodinGorinGuionnet2017}.
That paper proves the LLN for fixed $\theta$, but does not consider the high temperature regime; moreover, it expresses the limit measures as unique solutions to variational problems, while our characterizations are as unique solutions to certain moment problems.
Lastly, we mention that the proof of \cref{thm:application3} shows that the limiting measure in the appropriate scaling for the law of the Gorin--Shkolnikov process starting from the empty partition has all vanishing quantized $\gamma$-cumulants, except for the first one. This should be compared with the usual cumulant characterizations of the Gaussian and semicircle laws: the Gaussian has only one nonzero classical cumulant, while the semicircle law has only one nonzero free cumulant. A similar high-temperature statement was obtained for Hermite $\beta$-ensembles, where the limiting measure has only one nonzero $\gamma$-cumulant~\cite{Benaych-GeorgesCuencaGorin2022,MergnyPotters2022}. In this sense, \cref{thm:application3} gives a discrete counterpart of these classical, free, and high-temperature random matrix limits.

This analogy can also be viewed dynamically. In free probability, convolution with the family of semicircle laws defines the free Brownian-motion semigroup, also known as the Wigner semicircle flow; analytically, this is the free counterpart of the classical heat flow by Gaussian convolution~\cite{Biane1997FreeBrownianMotion,Biane1997FreeConvolutionSemicircle}. Moreover, this flow is related to Dyson Brownian motion through the large $N$ limit of its empirical eigenvalue process at fixed temperature~\cite{Biane1997FreeBrownianMotion}. The quantized $\gamma$-convolution developed here suggests the corresponding semigroup framework in the high-temperature discrete regime, with the pure-Plancherel limiting measure playing a role analogous to the semicircle law.

\subsection{Previous related works}
\label{subsec:methods}

Motivated by asymptotic representation theory, Vershik--Kerov~\cite{VershikKerov1977} and Logan--Shepp~\cite{LoganShepp1977} initiated the study of random partitions by proving the limit shape phenomenon for Plancherel-distributed Young diagrams. The connection to random matrices was later hinted at by Kerov~\cite{Kerov1993,Kerov1993transition}, who linked this shape to Wigner's semicircle distribution. This parallelism between the Plancherel measure and Gaussian Unitary Ensembles has been shown to be much deeper in a series of groundbreaking works~\cite{BaikDeiftJohansson1999,Okounkov2000,BorodinOkounkovOlshanski2000,Johansson2001}, which catalyzed further investigations and insights into the parallelism between random matrices and partitions that continue to this day. In particular, motivated by the one-parameter $\beta$-extrapolations of random (continuous) eigenvalue ensembles, it is natural to ask: what are the parallel one-parameter random partition ensembles?

Many articles have addressed this question by studying several alternatives; see, e.g.,~\cite{Kerov2000,BorodinOlshanski2005,DolegaFeray2016,BorodinGorinGuionnet2017,DolegaSniady2019,
GuionnetHuang2019,Huang2021,Moll2023,CuencaDolegaMoll2026}. The common theme is that all these partition ensembles are based on the theory of Jack symmetric functions and depend on the Jack parameter $\theta>0$, which is related to the inverse temperature $\beta$ from the continuous world by $\theta=\beta/2$. It is worth noting that the random $\theta$-deformed partitions arise not only as natural discrete versions of $\beta$-eigenvalues, but they also have intrinsic interest in various areas, including algebraic geometry~\cite{Okounkov2003,LiQinWang2003,PohlYoung2025} and Seiberg--Witten theory~\cite{NekrasovOkounkov2006}.

Unlike the continuous case, there are many different models of random partitions, and they differ quite drastically, even regarding the state space: some ensembles are defined on the set of all partitions $\Y$, e.g., the Jack measures~\cite{GorinShkolnikov2015,Moll2023}; some on the set $\Y_d$ of partitions of size $d$, e.g., the Jack--Plancherel measure~\cite{Kerov2000,DolegaFeray2016}, and more general measures from Jack characters with the Approximate Factorization Property~\cite{DolegaSniady2019}; finally, some ensembles are defined on the set $\Y(N)$ of partitions with at most $N$ parts, e.g., discrete $\beta$-ensembles~\cite{BorodinGorinGuionnet2017}. In the case of fixed $\theta$, the LLN and CLT for the global asymptotics of the various models of random $\theta$-partitions have been shown to be independent of the value of $\theta$, as in the continuous case~\cite{DolegaFeray2016,BorodinGorinGuionnet2017,DolegaSniady2019,Huang2021,Moll2023,CuencaDolegaMoll2026}.

In the high and low temperature regimes, Dołęga--Śniady~\cite{DolegaSniady2019} proved LLN and CLT for the global asymptotics of a model of measures on partitions of a fixed size, while Moll~\cite{Moll2023} proved LLN and CLT for a subclass of Jack measures on the set of all partitions. Further, Cuenca--Dołęga--Moll~\cite{CuencaDolegaMoll2026} studied both of these models simultaneously and exposed a deep connection between them, which in particular led to a universal moment-cumulant-type formula deforming the celebrated free version of Speicher~\cite{Speicher1994}. This formula is naturally expressed as a weighted generating function of Łukasiewicz paths (see~\cref{def:lukasiewicz}), and it resembles the generalized framework of the combinatorial approach of Viennot~\cite{Viennot1985} to orthogonal polynomials.

The universal character of the combinatorial moment-cumulant-type formula from \cite{CuencaDolegaMoll2026} motivates the question of whether a similar universal moment-cumulant-type formula exists for the aforementioned models of random partitions with at most $N$ parts. In this paper, we provide such a formula. The connection between the global asymptotics of random partitions with at most $N$ parts and free probability was observed by Biane~\cite{Biane1995} and Collins--{\'S}niady~\cite{CollinsSniady2009} in the context of the tensor product of large unitary representations. However, the asymptotic behavior of their model was studied under a different scaling than the one studied in this paper. In the $\theta=1$ case of the model studied here, Bufetov and Gorin noticed that the free version of cumulants from the previous works of Biane and Collins--Śniady should be replaced by the so-called \emph{quantized} free cumulants~\cite{BufetovGorin2015b}. Further, the global asymptotics for a fixed, but generic $\theta$, studied in Huang~\cite{Huang2021}, turned out to be the same as for $\theta=1$ after the appropriate normalization. The aforementioned papers did not provide a combinatorial moment-cumulant formula and described the LLN in analytic terms. In~\cref{sec:comparison}, we show that the analytic description of LLN from~\cite{Huang2021}, in the regime when $\theta$ is fixed, can be recovered from our combinatorial moment-cumulant formula.
It is worth mentioning that a different limiting case of our combinatorial formula recovers another recent moment-cumulant formula from~\cite{Benaych-GeorgesCuencaGorin2022}.

\subsection{Methodology}

The main technique employed in this article is the method of moments.
More explicitly, we compute the moments of the prelimit empirical measures $\mu_N$ of $\Prob_N$ by extracting them from the Jack generating functions using certain differential-reflection operators of representation-theoretic origin and then study the asymptotics of the resulting expression.
This general approach was pioneered for random partition ensembles by Bufetov and Gorin~\cite{BufetovGorin2015b,BufetovGorin2018,BufetovGorin2019}, who considered the Schur generating functions and differential operators with the Schur polynomials as eigenfunctions. Some applications and extensions followed swiftly by Bufetov--Knizel~\cite{BufetovKnizel2018} and Gorin--Sun~\cite{GorinSun2022}.

In the continuous framework, Benaych-Georges, Cuenca, and Gorin~\cite{Benaych-GeorgesCuencaGorin2022} utilized the multivariate Bessel functions and Dunkl operators to study the large-scale asymptotics of one-parameter extrapolations of random matrix eigenvalues in the high temperature regime. These ideas have been extended to the BC-type by~\cite{Xu2026}, and they have already found several applications~\cite{Yao2021,KeatingXu2024,GorinXuZhang2024}. In the discrete case, Huang~\cite{Huang2021} explored similar ideas for ensembles of partitions of a given length, for a generic but fixed finite temperature. In the main part of his paper, he used the Nazarov--Sklyanin operators~\cite{NazarovSklyanin2013} together with a formal symmetric-function version of the Jack generating function, whose expansion in the power-sum basis was crucial, as the Nazarov--Sklyanin operators behave well with respect to this basis. An $N$-finite version of the Jack generating function, closer to the one used here, also appears in Appendix B of~\cite{Huang2021}.

The present research was mainly motivated by~\cite{Benaych-GeorgesCuencaGorin2022,Huang2021,BufetovGorin2015b} and, simultaneously, it generalizes them all.
On the one hand, the random partition ensembles that we study degenerate in a certain regime to the random eigenvalues considered in~\cite{Benaych-GeorgesCuencaGorin2022}; at the technical level, our Jack generating functions and Cherednik operators also degenerate to the Bessel generating functions and Dunkl operators used in that reference (see~\eqref{eq:Cherednik},~\eqref{eqn:dunkl} for the definitions of the Cherednik and Dunkl operators $\xi_i, \tilde{\xi}_i$). Moreover, we recover their main result from our proof of the LLN by picking up the highest degree terms in all polynomial equations.
On the other hand, though we focus on the high temperature regime, we can use the very same analysis to prove analogues of our results at fixed temperature, allowing us to recover the LLN from~\cite{Huang2021}. As we have mentioned before, although a version of the Jack generating function (JGF) also appears in that paper, there are at least two advantages of our version that employs polynomials and not symmetric functions:
(a) our JGF admits explicit product forms for a wide class of Jack measures, as shown in \cref{sec:jack_measures}; (b) the power-sum expansion of the JGF used in~\cite{Huang2021} was an artifact of the used tools, rather than a natural choice of the monomial basis directly related to the model of partitions with a fixed number of parts; the latter choice allows for a direct analysis, which was not possible in~\cite{Huang2021} and required many technicalities. 

A naive approach of mimicking the analysis of the action of Dunkl operators on the Bessel function performed in~\cite{Benaych-GeorgesCuencaGorin2022} and applied to the action of Cherednik operators on the JGF fails. One of the key ideas in~\cite{Benaych-GeorgesCuencaGorin2022} was a decomposition of Dunkl operators as a sum of operators, whose action on the Bessel functions always produces the same result after the specialization of all variables to zero. In fact, this kind of symmetry was always present explicitly or implicitly in all previous analyses of similar problems. Such symmetry does not hold in the case of Cherednik operators\footnote{Added after revision: at a similar time when we finished this paper, Dunkl and Gorin wrote a paper~\cite{DunklGorin2025} about yet another family of operators defined in terms of Dunkl operators, the so-called Heckman--Polychronakos operators, and they wrote: \emph{``Many questions posed
for matrix $\beta$-ensembles have discrete counterparts, and similar techniques may be applied
to analyze them. This, however, requires lifting the power sums of Dunkl operators to the
discrete setting - leading naturally to either Cherednik or Heckman--Polychronakos operators. It is not yet clear which of the two classes is better suited for asymptotic analysis''}. The lack of symmetry that we observe when working with the Cherednik operators does not occur in the ``building blocks'' of Heckman--Polychronakos operators, but the cost to pay is to perform an analysis that allows one to control all the mixed moments; in the case of the Cherednik operators these mixed moments are obtained directly by acting on Jack polynomials, which is not the case for the Heckman--Polychronakos operators.} and requires new ideas. Our new key insight is that the representation-theoretic Hecke relations allow us to reduce an analysis of generic terms to the analysis of some special terms, and this reduction is also the source of the appearance of some new ingredients in the moment-cumulant formula (see~\cref{sec:hecke_proof}, and \cref{sec:comparison} for the comparison with the formula from~\cite{Benaych-GeorgesCuencaGorin2022}). We hope that this idea will be exploited for further applications.



\subsection*{Acknowledgements}

CC is grateful to Vadim Gorin for many useful conversations over the years.
CC would like to thank the Institute of Mathematics of the Polish Academy of Sciences (IMPAN) for hosting him during his research visits, and also the Royal Swedish Academy of Sciences and Institut Mittag--Leffler, where part of this research was conducted during the special program on Random Matrices and Scaling Limits in the Fall of 2024. MD would like to thank Houcine Ben Dali for his comments.
The authors also thank the anonymous referees for useful comments.

\section{Jack polynomial theory and Jack generating functions}\label{sec:jacks}

\subsection{Partitions and signatures}

In this paper, $N$ denotes a positive integer that will later tend to infinity.

A partition is an infinite sequence $\la=(\la_1\ge\la_2\ge\cdots)\in(\Z_{\ge 0})^\infty$ of weakly decreasing nonnegative integers such that $\la_i=0$, for all large $i$. The size of $\la$ is defined to be $|\la|:=\sum_{i\ge 1}{\la_i}$, while its length is $\ell(\la):=\max\{ i\ge 1 \colon \la_i\ne 0 \}$. For the empty partition $\emptyset=(0,0,0,\cdots)$, the size and length are both zero, by convention. For a partition $\la$ with $\ell(\la)\le N$, we can write $\la=(\la_1,\dots,\la_N)$, with the understanding that infinitely many zeroes are being omitted.
We denote the set of all partitions by $\Y$.
For any $N\in\Z_{\ge 1}$, we also denote
\begin{equation*}
\Y(N) := \{ \la\in\Y \colon \ell(\la)\le N \}.
\end{equation*}

For any $N\in\Z_{\ge 1}$, we say that an $N$-tuple $\la=(\la_1\ge\cdots\ge\la_N)\in\Z^N$ of weakly decreasing integers is an \emph{$N$-signature}.
By analogy to partitions, we say that an $N$-signature $\la=(\la_1,\cdots,\la_N)$ has length $N$.
The set of all $N$-signatures is denoted by
\begin{equation*}
\GT(N) := \{ \la = (\la_1\ge\cdots\ge\la_N)\in\Z^N \}.
\end{equation*}
Observe that $\Y(N)\subseteq\GT(N)$.

\subsection{The Jack parameter}\label{sec:jack_parameter}

In this paper, there is a positive real $\theta>0$ that will later tend to zero simultaneously with $N$ tending to infinity.
The parameter $\theta$ is called the \emph{Jack parameter}, because of its relationship with Jack symmetric polynomials.
In the literature, e.g.~\cite{Stanley1989,Macdonald1995}, sometimes a different Jack parameter $\alpha>0$ is used; the relation between them is $\theta=\frac{1}{\alpha}$.

To any $N$-signature $\la=(\la_1\ge\cdots\ge\la_N)\in\GT(N)$, we associate the $N$-tuple of reals $(\LL_1>\cdots>\LL_N)$ defined by
\begin{equation}\label{eq:ell}
\LL_i := \la_i - (i-1)\theta,\text{ for }i=1,2,\dots,N.
\end{equation}
This gives the following bijection that will be used throughout the paper:
\begin{multline*}
\GT(N) \longleftrightarrow
\{(\LL_1,\dots,\LL_N)\in\R^N\colon \LL_{i}-\LL_{i+1} \geq \theta \text{ for }i=1,\dots,N-1,\\
\text{ and } \LL_i+(i-1)\theta\in\Z,\text{ for }i=1,\dots,N\}.
\end{multline*}

\subsection{Jack Laurent polynomials}

Given any $N\in\Z_{\ge 1}$, the \emph{Jack (symmetric) polynomials} $P_\la(x_1,\dots,x_N;\theta)$, for $\la\in\Y(N)$, are defined as the unique homogeneous symmetric polynomials of degree $|\la|$ of the form
\begin{multline}\label{eqn:expansion}
P_\la(x_1,\dots,x_N;\theta) = x_1^{\la_1}\cdots x_N^{\la_N} + \text{ monomials of the form }c_\mu x_1^{\mu_1}\cdots x_N^{\mu_N}\\
\text{ with }\mu\le\la\text{ in the lexicographic order},
\end{multline}
that are orthogonal with respect to a certain explicit weight function; we do not discuss the details further since they will not be needed, but see~\cite[Ch.~VI.10]{Macdonald1995}.
If $\la\in\Y$ has $\ell(\la)>N$, we use the convention $P_\la(x_1,\dots,x_N;\theta):=0$.
We have the index-stability property:
\begin{equation*}
P_{(\la_1+1,\dots,\la_N+1)}(x_1,\dots,x_N;\theta) = (x_1\cdots x_N)\cdot P_\la(x_1,\dots,x_N;\theta),\quad\la\in\Y(N).
\end{equation*}
This property allows us to define the \emph{Jack Laurent polynomials} (see~\cite{SergeevVeselov2015}), for any $\la\in\GT(N)$, by taking any $M>0$ with $\la_N+M>0$, so that $(\la_1+M,\dots,\la_N+M)\in\Y(N)$, and setting
\begin{equation}\label{eqn:jack_laurent}
P_\la(x_1,\dots,x_N;\theta) := (x_1\cdots x_N)^{-M}\cdot P_{(\la_1+M,\dots,\la_N+M)}(x_1,\dots,x_N;\theta).
\end{equation}
The definition~\eqref{eqn:jack_laurent} does not depend on the choice of $M>0$, because of the index-stability.

From~\eqref{eqn:expansion}, all monomials $x_1^{\mu_1}\cdots x_N^{\mu_N}$ in the expansion of the Jack Laurent polynomial $P_\la(x_1,\dots,x_N;\theta)$, $\la\in\GT(N)$, are such that $\la_1\ge\max_i{\mu_i}\ge\min_i{\mu_i}\ge\la_N$.
We will also need the fact that for any $\theta>0$ and $\la\in\GT(N)$, the coefficients of $P_\la(x_1,\dots,x_N;\theta)$ are nonnegative, see~\cite[Ch.~VI, (10.12)-(10.13)]{Macdonald1995}.

\subsection{Cherednik operators}

The Cherednik operators are the differential-reflection operators $\xi_1,\dots,\xi_N$ acting on functions with $N$ variables $x_1,\dots,x_N$, given by
\begin{equation}\label{eq:Cherednik}
\xi_i := x_i\partial_i-\theta(i-1)+\theta\sum_{j=1}^{i-1}\frac{x_i}{x_i-x_j}(1-s_{i,j}) +\theta\sum_{j=i+1}^{N}\frac{x_j}{x_i-x_j}(1-s_{i,j}),
\end{equation}
where $s_{i,j}f(x_1,\dots,x_i,\dots,x_j,\dots,x_N) :=  f(x_1,\dots,x_j,\dots,x_i,\dots,x_N)$ acts by transposing the variables $x_i$ and $x_j$, and we denote throughout $\partial_i := \frac{\partial}{\partial x_i}$, for simplicity. The operators $\xi_i$ were introduced by Cherednik in~\cite{Cherednik1991}.

Evidently, each $\xi_i$ preserves the space of $N$-variate polynomials of degree $d$, for any $d\in\Z_{\ge 0}$.
They also preserve the space of functions that are analytic in some neighborhood of any point of the form $(x_0^N):=(x_0,\dots,x_0)$, for example $(1^N)$.

The following theorem is due to Opdam~\cite[(2.7)]{Opdam1995}; see also \cite[Corollary 3.2]{KnopSahi1997}.

\begin{theorem}\label{thm:cherednik_ops}
The Cherednik operators $\xi_1,\dots,\xi_N$ satisfy the following Hecke relations:
\begin{equation}\label{eq:Hecke}
\begin{aligned}
    \xi_i s_{i} - s_{i}\xi_{i+1} &= \theta,\\
    \xi_{i+1} s_{i} - s_{i}\xi_{i} &= -\theta,\\
    \xi_{i} s_{j} - s_{j}\xi_{i} &= 0, \text{ for all }j \neq i-1,i,
\end{aligned}
\end{equation}
where $s_i$ denotes the transposition $s_i := s_{i,i+1}$ of variables $x_i$ and $x_{i+1}$.
Moreover, the operators $\xi_1,\dots,\xi_N$ pairwise commute and for any symmetric polynomial $f(x_1,\dots,x_N)$, and all $\la\in\GT(N)$, we have
\begin{equation}\label{eq:Cherednik_2}
f(\xi_1,\dots,\xi_N)\, P_\la(x_1,\dots,x_N;\theta) = f(\LL_1,\dots,\LL_N)\cdot P_\la(x_1,\dots,x_N;\theta),
\end{equation}
where $\LL_i:=\la_i-(i-1)\theta$.
\end{theorem}

We point out that~\eqref{eq:Cherednik_2} is stated for $\la\in\Y(N)$ in the references~\cite{Opdam1995,KnopSahi1997}, but this identity can be extended to all $\la\in\GT(N)$.

\subsection{Measures on $N$-signatures and Jack generating functions}

In this paper, we will generally consider sequences of finite signed measures $\PP_N$ on the sets $\GT(N)$ of $N$-signatures; moreover, we will always assume that $\PP_N$ have total mass equal to $1$.
Sometimes, to be concise, we may omit the qualifiers finite or signed or of total mass one, and simply say \emph{measures}.
In applications, however, we will usually restrict ourselves to the case when $\PP_N$ are probability measures, except in \cref{sec:app_1}.

Given a finite signed measure $\PP_N$ on $N$-signatures $\la=(\la_1\ge\dots\ge\la_N)\in\GT(N)$, with total mass equal to $1$, define its \emph{Jack generating function} as the following formal Laurent series on $N$ variables:
\begin{equation}\label{eq:JackGeneratFunction}
G_{\PP_N,\theta}(x_1,\dots,x_N) := \sum_{\la\in\GT(N)}\PP_N(\la)\frac{P_\la(x_1,\dots,x_N;\theta)}{P_\la(1^N;\theta)}.
\end{equation}
We implicitly assume that the measure $\PP_N$ might additionally depend on the Jack parameter $\theta$, so that the Jack generating function $G_{\PP_N,\theta}$ is defined with respect to the measure $\PP_N$ evaluated at $\theta > 0$, see~\cref{sec:applications} for examples.
Next, we identify a large family of measures $\PP_N$ for which $G_{\PP_N,\theta}$ is an analytic function in a neighborhood of $(1^N)$.

\begin{definition}\label{def:small_tails}
We say that the finite signed measure $\PP_N$ on $\GT(N)$ has \textbf{small tails} if the Laurent series
\begin{equation}\label{eqn:laurent_series}
\sum_{\la\in\GT(N)}{ \PP_N(\la)\frac{P_\la(z_1,\dots,z_N;\theta)}{P_\la(1^N;\theta)} }
\end{equation}
converges absolutely on an $N$-dimensional annulus of the form
\begin{equation}\label{eqn:annulus}
A_{N;R} := \{(z_1,\dots,z_N)\in\C^N\colon R^{-1}<|z_i|<R,\text{ for all }i=1,\dots,N\},
\end{equation}
for some $R>1$.
\end{definition}

\begin{example}
If $\PP_N$ is supported on a finite subset of $\GT(N)$, then $\PP_N$ has small tails.
\end{example}

If $\PP_N$ has small tails, then evidently its Jack generating function $G_{\PP_N,\theta}$ is an analytic function in the annulus \eqref{eqn:annulus}, and moreover
\begin{equation*}
G_{\PP_N,\theta}(1^N) = 1,
\end{equation*}
since $\PP_N$ has total mass equal to $1$.

\begin{remark}
If $\PP_N$ is a probability measure with small tails and is supported on the set $\Y(N)\subseteq\GT(N)$ of partitions of length $\le N$, then the series~\eqref{eqn:laurent_series} is a power series (not a Laurent series) and it converges absolutely on an $N$-dimensional disk
\begin{equation*}
D_{N;R} := \{(z_1,\dots,z_N)\in\C^N\colon |z_i|<R,\text{ for all }i=1,\dots,N\},
\end{equation*}
for some $R>1$.
This follows from the fact that the coefficients of any Jack symmetric polynomial $P_\la(z_1,\dots,z_N;\theta)$ are nonnegative.
\end{remark}

For our examples in \cref{sec:applications}, all the probability measures $\PP_N$ are supported on the set $\Y(N)\subseteq\GT(N)$ as in the previous remark, and~\cref{def:small_tails} is easy to check.

\subsection{Jack symmetric functions}

We recall some material from \cite[Ch.~VI.10]{Macdonald1995}, \cite{Stanley1989}.
Let $\Sym$ be the real algebra of symmetric functions, that is, the projective limit of the chain
\begin{equation*}
\cdots \xrightarrow{\ x_{N+2}=0\ } \R[x_1,\dots,x_N,x_{N+1}]^{\Sy{N+1}} \xrightarrow{\ x_{N+1}=0\ } \R[x_1,\dots,x_N]^{\Sy{N}} \xrightarrow{\ x_{N}=0\ } \cdots
\end{equation*}
of symmetric polynomials in finitely many variables.
Let $\Y:=\bigcup_{N\ge 1}{\Y(N)}$ be the set of all partitions.
For any $\la\in\Y$, we have the stability property of Jack polynomials
\begin{equation*}
P_\la(x_1,\dots,x_N,x_{N+1};\theta) \big|_{x_{N+1}=0} = P_\la(x_1,\dots,x_N;\theta),
\end{equation*}
which shows the existence of the projective limits $P_\la(\xx;\theta)=\varprojlim{P_\la(x_1,\dots,x_N;\theta)}$, called the \emph{Jack symmetric functions}.
Sometimes we denote the Jack symmetric functions by $P_\la(;\theta)$, without specifying the alphabet $\xx=(x_1,x_2,\dots)$.
It is known that $\{P_\la(;\theta)\}_{\la\in\Y}$ is a basis of $\Sym$.
We will also need the following \emph{Cauchy identity} for Jack symmetric functions:
\begin{equation}\label{eqn:cauchy}
\sum_{\la\in\Y}P_\la(\xx;\theta)Q_\la(\yy;\theta) = \prod_{i,j \geq 1}\left( 1-x_i y_j\right)^{-\theta} = \exp\Bigg\{ \theta\sum_{k\ge 1}{\frac{p_k(\xx)p_k(\yy)}{k}} \Bigg\} =: H_\theta(\xx;\yy),
\end{equation}
where $\xx=(x_1,x_2,\dots)$, $\yy=(y_1,y_2,\dots)$, and
\begin{equation}\label{j_lambda}
Q_\la(;\theta) := \prod_{(i,j)\in\la}\frac{(\la_i-j)+\theta(\la_j'-i)+\theta}{(\la_i-j)+\theta(\la_j'-i)+1}\cdot P_\la(;\theta)
\end{equation}
is the $Q$-normalization of Jack polynomials/symmetric functions.
In~\cref{j_lambda}, we denoted $\la'$ the conjugate partition to $\la$, so that $\la_j'$ is the length of the $j$-th largest column of the Young diagram of $\la$.
We will also need the following dual Cauchy identity:
\begin{equation}\label{eqn:cauchy_dual}
\sum_{\la\in\Y}P_\la(\xx;\theta)P_{\la'}(\yy;\theta^{-1}) = \prod_{i,j \geq 1}\left( 1+x_i y_j\right) = \exp\Bigg\{ \sum_{k\ge 1}{\frac{(-1)^{k-1}p_k(\xx)p_k(\yy)}{k}} \Bigg\}.
\end{equation}

For any partitions $\la,\mu\in\Y$, we also define the \emph{skew-Jack symmetric function} $Q_{\la/\mu}(\xx;\theta)$ (and analogously, the $P$-normalized version $P_{\la/\mu}(\xx;\theta)$) by the following identity:
\begin{equation}\label{Q_skew}
Q_\la(\xx,\yy;\theta) := \sum_{\mu \in \Y}Q_\mu(\yy;\theta)Q_{\la/\mu}(\xx;\theta).
\end{equation}
We remark that the sum in the right hand side of~\eqref{Q_skew} is actually finite, since it is known that the skew-Jack symmetric functions $Q_{\la/\mu}(\xx;\theta)$ vanish unless $\mu\subseteq\la$.
Moreover, $Q_{\la/\mu}(\xx;\theta)$ coincides with $Q_{\la}(\xx;\theta)$ when $\mu=\emptyset$.
We will also need the following \emph{skew-Cauchy identities}:
\begin{equation}\label{eqn:cauchy_skew}
\sum_{\la\in\Y}P_{\la/\mu}(\xx;\theta)Q_{\la/\rho}(\yy;\theta) = H_\theta(\xx;\yy)\sum_{\la\in\Y}Q_{\mu/\la}(\yy;\theta)P_{\rho/\la}(\xx;\theta).
\end{equation}

\subsection{Jack-positive specializations}\label{sec:jack_specs}

\begin{definition}
A \textbf{specialization} of $\Sym$ is a unital algebra homomorphism $\rho\colon\Sym\to\R$.
For any $f\in\Sym$, we denote the image of $f$ under $\rho$ by $f(\rho)$.
The specialization $\rho\colon\Sym\to\R$ is said to be \textbf{Jack-positive} if
\begin{equation*}
P_\la(\rho;\theta) \ge 0,\quad\text{for all }\la\in\Y.
\end{equation*}
Note that this definition depends on the value of $\theta$, which is usually understood from the context.
\end{definition}

\begin{remark}
Technically speaking, it would be more appropriate to call this specialization \emph{Jack-non-negative}, as the definition allows $P_\la(\rho;\theta) = 0$. However, the above terminology is much more common in the literature and became standard in the field, so we will stick to it, and ask the reader to keep in mind that the inequality is weak.
\end{remark}

The Cauchy identity~\eqref{eqn:cauchy} turns into a numerical equality if $\xx,\yy$ are replaced by two specializations $\rho_1,\rho_2$, as long as the resulting LHS in~\eqref{eqn:cauchy} converges absolutely.
Furthermore, if $\rho_1,\rho_2$ are Jack-positive, then each term $P_\la(\rho_1;\theta)Q_\la(\rho_2;\theta)$ is nonnegative, hence~\eqref{eqn:cauchy} is a series of nonnegative real numbers that converges to $H_\theta(\rho_1;\rho_2)$ and this value is strictly positive, as $H_\theta(\rho_1;\rho_2)\ge P_\emptyset(\rho_1;\theta)Q_\emptyset(\rho_2;\theta)=1>0$.

\begin{theoremdefinition}[\cite{KerovOkounkovOlshanski1998}]\label{thm:KOO}
Given $\theta>0$ fixed, the set of Jack-positive specializations is in bijection with points of the following \textbf{Thoma cone}:
\begin{multline*}
\Omega := \Big\{ (\alpha, \beta, \delta)\in(\R_{\ge 0})^\infty\times(\R_{\ge 0})^\infty\times\R_{\ge 0} \ \mid 
\ \alpha=(\alpha_1,\alpha_2,\dots),\ \beta=(\beta_1,\beta_2,\dots),\\
\alpha_1\ge\alpha_2\ge\dots\ge 0,\quad \beta_1\ge\beta_2\ge\dots\ge 0,\quad \sum_{i=1}^\infty(\alpha_i+\beta_i) \le \delta \Big\}.
\end{multline*}
For any $\omega=(\alpha,\beta,\delta)\in\Omega$, the corresponding Jack-positive specialization $\rho_\omega\colon\Sym\to\R$ is given by
\begin{equation*}
p_1(\rho_\omega)=\delta,\qquad p_k(\rho_\omega)=\sum_{i=1}^\infty\Big( \alpha_i^k + (-\theta)^{k-1}\beta_i^k \Big),\quad k=2,3,\cdots.
\end{equation*}
If $\omega=(\alpha,\beta,\delta)\in\Omega$, we say that $\alpha_i, \beta_i,\delta$ are the \textbf{Thoma parameters} of $\rho_\omega$.
\end{theoremdefinition}

\begin{example}
\label{exam:pure_alpha}
If all Thoma parameters of $\omega=(\alpha,\beta,\delta)\in\Omega$ vanish, except possibly $\alpha_1,\dots,\alpha_n$ and $\delta = \sum_{i=1}^n\alpha_i$, then we say that $\rho_\omega$ is a \emph{pure alpha specialization} and we denote it by $\Alpha(\alpha_1,\dots,\alpha_n)$.
For any $f=f(x_1,x_2,\dots)\in\Sym$, we have
\begin{equation*}
f\big(\Alpha(\alpha_1,\dots,\alpha_n)\big) = f(\alpha_1,\dots,\alpha_n),
\end{equation*}
i.e. $\Alpha(\alpha_1,\dots,\alpha_n)$ specializes $n$ variables $x_i\mapsto\alpha_i$ and the rest of $x_i$'s are set to zero. Note that $P_{\lambda}\big(\Alpha(\alpha_1,\dots,\alpha_n); \theta\big) = 0$ for all partitions $\lambda$ that have at least $n+1$ positive parts, i.e. $\lambda_{n+1} > 0$.
\end{example}

\begin{example}\label{exam:pure_beta}
If the only possibly nonzero Thoma parameters of $\omega=(\alpha,\beta,\delta)\in\Omega$ are $\beta_1,\dots,\beta_n$ and $\delta = \sum_{i=1}^n\beta_i$, then we say that $\rho_\omega$ is a \emph{pure beta specialization} and we denote it by $\Beta(\beta_1,\dots,\beta_n)$.
For any $f\in\Sym$, we have
\begin{equation*}
f\big(\Beta(\beta_1,\dots,\beta_n)\big) = (\omega_{\theta^{-1}}f)(\theta\beta_1,\dots,\theta\beta_n),
\end{equation*}
where $\omega_{\theta^{-1}}$ is the automorphism of $\Sym$ defined by $\omega_{\theta^{-1}}(p_r):=(-1)^{r-1}\theta^{-1}p_r$, for all $r\ge 1$.
This means that $\Beta(\beta_1,\dots,\beta_n)$ is the composition of $\omega_{\theta^{-1}}$ and the pure alpha specialization $\Alpha(\theta\beta_1,\dots,\theta\beta_n)$. Therefore, since $\omega_{\theta^{-1}}(Q_\la(;\theta)) = P_{\la'}(;\theta^{-1})$ (see \cite[Ch.~VI.10, Eqn.~(10.17)]{Macdonald1995}), then
\[
Q_{\lambda'}\big(\Beta(\beta_1,\dots,\beta_n);\theta\big)
=
P_\lambda(\theta\beta_1,\dots,\theta\beta_n;\theta^{-1})=0
\]
for all partitions $\lambda$ that have at least $n+1$ positive parts.
\end{example}

\begin{example}\label{exam:plancherel}
If all Thoma parameters of $\omega=(\alpha,\beta,\delta)\in\Omega$ vanish, except possibly $\delta$, then we say that $\rho_\omega$ is a \emph{Plancherel specialization} and we denote it by $\Planch(\delta)$.
Equivalently, it is uniquely defined by $p_k(\Planch(\delta)) = \delta\cdot\mathbf{1}_{k=1}$. It turns out that this specialization is always strictly Jack-positive (see~\eqref{eq:planch_1}).
\end{example}

For any specializations $\rho_1,\rho_2$ of $\Sym$, we define the \emph{union specialization}, to be denoted simply $(\rho_1,\rho_2)$, by the formula
\[ p_k(\rho_1,\rho_2) := p_k(\rho_1) + p_k(\rho_2),\quad\text{for all }k\ge 1.\]
If $\rho_1,\rho_2$ are Jack-positive specializations, then so is their union $(\rho_1,\rho_2)$, by virtue of \cref{thm:KOO}.
Furthermore, for any specialization $\rho$, define the $n$-fold union $n\rho := (\underbrace{\rho,\dots,\rho}_{n})$; again, if $\rho$ is Jack-positive, then so is $n\rho$.

\section{Statement of the main theorem: Law of Large Numbers at high temperature}\label{sec:main_thm_statement}

We are going to study sequences of finite signed measures $\{\PP_N\}_{N \geq 1}$ on $N$-signatures in the \emph{high temperature regime}
\begin{equation}\label{eq:HTRegime}
N\to\infty,\quad \theta\to 0,\quad N\theta\to\gamma,
\end{equation}
where $\gamma\in(0,\infty)$ is a fixed constant.

Let us assume for the rest of this paper that all measures $\PP_N$ have small tails in the sense of \cref{def:small_tails}.
Then the Jack generating functions, denoted by
\begin{equation}\label{eqn:jack_G}
G_{N,\theta}(x_1,\dots,x_N) := \sum_{\la\in\GT(N)}\PP_N(\la)\frac{P_\la(x_1,\dots,x_N;\theta)}{P_\la(1^N;\theta)},
\end{equation}
are analytic in a neighborhood of $(1^N)$ and $G_{N,\theta}(1^N)=1$.
Consequently, the logarithm
\begin{equation}\label{eqn:jack_F}
F_{N,\theta}(x_1,\dots,x_N) := \ln(G_{N,\theta}(x_1,\dots,x_N))
\end{equation}
also defines an analytic function in a neighborhood of $(1^N)$ and $F_{N,\theta}(1^N)=0$.

\begin{definition}[HT-appropriateness]\label{def:appropriate}
Let $\{\PP_N\}_{N \geq 1}$ be a sequence of finite signed measures on $N$-signatures with small tails and let $F_{N,\theta}=F_{N,\theta}(x_1,\dots,x_N)$ be the logarithms of the Jack generating functions of $\PP_N$, i.e.~the functions $F_{N,\theta}(x_1,\dots,x_N)$ defined by \eqref{eqn:jack_G}--\eqref{eqn:jack_F}.

We say that the sequence $\{\PP_N\}_{N \geq 1}$ is \text{HT-appropriate}\footnote{HT stands for high temperature.} if there exists a sequence of real constants $\kappa_1^{\gamma}, \kappa_2^{\gamma},\dots$ such that

\begin{enumerate}
	\item $\displaystyle\lim_{\substack{N\to\infty\\N\theta\to\gamma}} \frac{\partial_1^n}{(n-1)!}F_{N,\theta}\big|_{(x_1,\dots,x_N) = (1^N)} = \kappa_n^\gamma$, for all $n \in \Z_{\geq 1}$,

	\item $\displaystyle\lim_{\substack{N\to\infty\\N\theta\to\gamma}} \partial_{i_1}\cdots \partial_{i_r}F_{N,\theta}\big|_{(x_1,\dots,x_N) = (1^N)} = 0$, for all $r\ge 2$ and $i_1,\dots,i_r\in\Z_{\geq 1}$ with at least two distinct indices among $i_1,\dots,i_r$.
\end{enumerate}
We also say that the sequence $\{F_{N,\theta}\}_{N \geq 1}$ is HT-appropriate if these conditions hold.
\end{definition}

The measure $\PP_N$ on $N$-signatures $\la=(\la_1\ge\dots\ge\la_N)\in\GT(N)$ can also be regarded as a measure on real $N$-tuples $(\LL_1>\dots>\LL_N)$, where $\LL_i+(i-1)\theta\in\Z$, for all $i=1,\dots,N$, by using the shifted coordinates $\LL_i = \la_i - (i-1)\theta$, as explained in \cref{sec:jack_parameter}.

\begin{definition}[LLN-satisfaction]\label{def:LLN}
Let $\{\PP_N\}_{N \geq 1}$ be a sequence of finite signed measures on $N$-signatures with small tails.
We say that $\{\PP_N\}_{N \geq 1}$ \textbf{satisfies the LLN} (Law of Large Numbers) if  there exist $m_1,m_2,\dots\in\R$ such that, for all $s\in\Z_{\ge 1}$ and $k_1,\dots,k_s\ge 1$, we have
\begin{equation}\label{eq:DefConvMoments}
\lim_{\substack{N\to\infty,\\ N\theta \to \gamma}}\frac{1}{N^s}\,\E_{\PP_N} \left[ \prod_{j=1}^s{ \sum_{i=1}^N{\LL_i^{k_j}} } \right]
= \prod_{j=1}^s{m_{k_j}},
\end{equation}
where, in the LHS, the $N$-tuples $(\LL_1>\dots>\LL_N)$ are $\PP_N$-distributed.\footnote{We use here the probabilistic terminology and notation, as if $\PP_N$ were a probability measure, but in general, for a function $f\colon\GT(N)\to\R$, the expression $\E_{\PP_N}[f(\la)]$ means $\sum_{\la\in\GT(N)}{f(\la)\PP_N(\la)}$.}
\end{definition}

If all $\PP_N$ are probability measures, then the LHS of~\eqref{eq:DefConvMoments} can be interpreted in terms of the \emph{empirical measures} $\mu_N$ of $\PP_N$, which by definition are the (random) probability measures on $\R$, given by
\begin{equation}\label{eqn:empirical_measures}
\mu_N := \frac{1}{N}\sum_{i=1}^N\delta_{\LL_i},\text{ where }(\LL_1>\dots>\LL_N)\text{ is }\PP_N\text{--distributed}.
\end{equation}
Indeed, the moments of the empirical measures are
\begin{equation*}
m_k(\mu_N) := \int_{\R}x^k\mu_N(\text{d}x) = \frac{1}{N}\sum_{i=1}^N{\LL_i^k},
\end{equation*}
then~\eqref{eq:DefConvMoments} is equivalent to:
\begin{equation*}
\lim_{\substack{N\to\infty\\N\theta\to\gamma}} \E_{\PP_N} \left[ \prod_{j=1}^s{ m_{k_j}(\mu_N) } \right]
= \prod_{j=1}^s{m_{k_j}}.
\end{equation*}
As a result, LLN-satisfaction is equivalent to the convergence $\mu_N\to\mu$ in the sense of moments, in probability, to a probability measure $\mu$ with moments $m_1,m_2,\cdots$.
In the case that $m_1,m_2,\dots$ are the moments of a unique probability measure, then LLN-satisfaction implies $\mu_N\to\mu$ weakly, in probability.

Our main result below shows that a sequence $\{\PP_N\}_{N\ge 1}$ is HT-appropriate if and only if it satisfies the LLN. However, this result would not be complete if we did not show how the sequences $(\kappa^\gamma_n)_{n\ge 1}$ and $(m_n)_{n\ge 1}$ in these definitions are related to each other.
This relation depends on the combinatorial gadgets called Łukasiewicz paths.

\begin{definition}\label{def:lukasiewicz}
A \textbf{Łukasiewicz path $\Gamma$ of length $\ell$} is a lattice path on $\Z^2$ that starts at the origin $(0,0)$, ends at $(\ell,0)$, stays above the $x$-axis, and which has steps $(1,0)$ (horizontal steps), $(1,-1)$ (down steps) and steps of the form $(1, j)$, for some $j\in\Z_{\ge 1}$ (up steps).
We denote by $\Luk(\ell)$ the set of all Łukasiewicz paths of length $\ell$.
\end{definition}

It is known that the cardinality of $\Luk(\ell)$ is the Catalan number $C_\ell = \frac{1}{\ell+1}{2\ell\choose\ell}$.
For example, the five Łukasiewicz paths of length $3$ are depicted in \cref{fig:PathsEx}.
It turns out that the transformation $(\kappa^\gamma_n)_{n\ge 1}\mapsto (m_n)_{n\ge 1}$ can be described in terms of weighted Łukasiewicz paths. Since this transformation is well defined for any sequence of real numbers, not necessarily for the sequence $(\kappa^\gamma_n)_{n\ge 1}$ that comes from an HT-appropriate sequence of measures and explicitly depends on $\gamma$, we will define it in a more general setting, and we will drop the superscript $\gamma$ from the notation of $\kappa^\gamma_n$.

\begin{definition}\label{def:mk}
Given a constant $\gamma\in (0,\infty)$ and $\vec{\kappa}=(\kappa_n)_{n\ge 1} \in \R^\infty$, define $\vec{m}=(m_n)_{n\ge 1} \in \R^\infty$ by
\begin{multline}\label{eq:Moments}
m_\ell := \sum_{\Gamma\in\Luk(\ell)}
\frac{\Delta_\gamma\left(x^{1+\text{\#\,horizontal steps at height $0$ in $\Gamma$}}\right)(\kappa_1)}{1+\text{\#\,horizontal steps at height $0$ in $\Gamma$}}
\cdot\prod_{i\ge 1}(\kappa_1+i)^{\text{\#\,horizontal steps at height $i$ in $\Gamma$}}\\
\cdot\prod_{j\ge 1}(\kappa_j+\kappa_{j+1})^{\text{\#\,steps $(1,j)$ in $\Gamma$}}(j+\gamma)^{\text{\#\,down steps from height $j$ in $\Gamma$}},\quad\text{for all $\ell\in\Z_{\ge 1}$},
\end{multline}
where the term $\Delta_\gamma$ is the divided difference operator, defined by
\begin{equation}\label{eq:DivDiff}
\Delta_\gamma(f)(x) := \frac{f(x)-f(x-\gamma)}{\gamma},
\end{equation}
and then evaluated at $x = \kappa_1$.
We denote the map $\vec{\kappa}\mapsto\vec{m}$ by $\mathcal{J}_\gamma^{\kappa\mapsto m}\colon\R^\infty\to\R^\infty$ and the compositional inverse $\vec{m}\mapsto\vec{\kappa}$ by $\mathcal{J}_\gamma^{m\mapsto\kappa} := (\mathcal{J}_\gamma^{\kappa\mapsto m})^{(-1)}$.
\end{definition}

Note that $\Gamma\in\Luk(\ell)$ cannot have steps $(1,j)$ with $j\ge\ell$ because $\Gamma$ should begin at $(0,0)$, end at $(\ell,0)$, and the only down steps are $(1,-1)$.
Also, there is exactly one Łukasiewicz path in $\Luk(\ell)$ with at least one step $(1,\ell-1)$, namely the one that first takes the step $(1,\ell-1)$ and then takes $\ell-1$ down steps $(1,-1)$.
Both observations imply that only $\kappa_1,\cdots,\kappa_{\ell-1},\kappa_\ell$ appear in the right hand side of~\eqref{eq:Moments} and $\kappa_\ell$ appears only from the unique Łukasiewicz path with a step $(1,\ell-1)$ at the beginning, followed by $\ell-1$ down steps.
As a result, we deduce
\begin{equation}\label{eq:top_moments}
m_\ell = (\gamma+1)^{\uparrow(\ell-1)}\cdot\kappa_\ell\,+ \text{ polynomial over } \Q[\gamma] \text{ in }\kappa_1,\cdots,\kappa_{\ell-1},
\end{equation}
where
\begin{equation}\label{eq:rising_factorial}
g^{\uparrow 0} := 1,\qquad g^{\uparrow n} := \prod_{i=1}^{n}{(g+i-1)},\quad n\in\Z_{\ge 1},
\end{equation}
is the \emph{rising factorial}. Therefore the system of equations~\eqref{eq:Moments} is invertible,
\begin{equation}\label{eq:kappa-top}
\kappa_\ell = \frac{1}{(\gamma+1)^{\uparrow(\ell-1)}}\cdot m_\ell\,+ \text{ polynomial over } \Q(\gamma) \text{ in }m_1,\cdots,m_{\ell-1},
\end{equation}
which indeed guarantees the existence of the compositional inverse map $\mathcal{J}_\gamma^{m\mapsto\kappa}$.

Certain limit of formula~\eqref{eq:Moments} when $\gamma=\infty$, as well as the corresponding maps $\mathcal{J}_\infty^{\kappa\mapsto m}$ and $\mathcal{J}_\infty^{m\mapsto\kappa}$, are discussed in \cref{sec:comparison}; see \cref{def:J_infty}.

\begin{example}
  \label{ex:mk}
The first few entries of the sequence $\vec{m}=\mathcal{J}_\gamma^{\kappa\mapsto m}(\vec{\kappa})$ are:
\begin{align*}
m_1 &= \kappa_1 - \frac{\gamma}{2},\\
m_2 &= (\gamma+1)\kappa_2 + \kappa_1^2 + \kappa_1 + \frac{\gamma^2}{3},\\
m_3 &= (\gamma+1)(\gamma+2)\kappa_3 + 3(\gamma+1)\kappa_2\kappa_1+\kappa_1^3+3(\gamma+1)\kappa_2+\frac{3}{2}(\gamma+2)\kappa_1^2+\kappa_1-\frac{\gamma^3}{4}.
\end{align*}
See \cref{fig:PathsEx} for the calculation of $m_3$.
From these equalities, we can deduce the first few entries of the inverse map $\vec{\kappa}=\mathcal{J}_\gamma^{m\mapsto\kappa}(\vec{m})$:
\begin{align*}
\kappa_1 &= m_1 + \frac{\gamma}{2},\\
\kappa_2 &= \frac{m_2-m_1^2-(\gamma+1)m_1-\frac{\gamma(7\gamma+6)}{12}}{\gamma+1},\\
\kappa_3 &= \frac{m_3-3m_2m_1+2m_1^{3}-\frac{3}{2}(\gamma+2)m_2+\frac{3}{2}(\gamma+2)m_1^2+(\gamma+1)(\gamma+2)m_1+\frac{\gamma(5\gamma+4)(\gamma+2)}{8}}{(\gamma+1)(\gamma+2)}.
\end{align*}
\end{example}

\begin{figure}
	\centering
	\includegraphics[width=\textwidth]{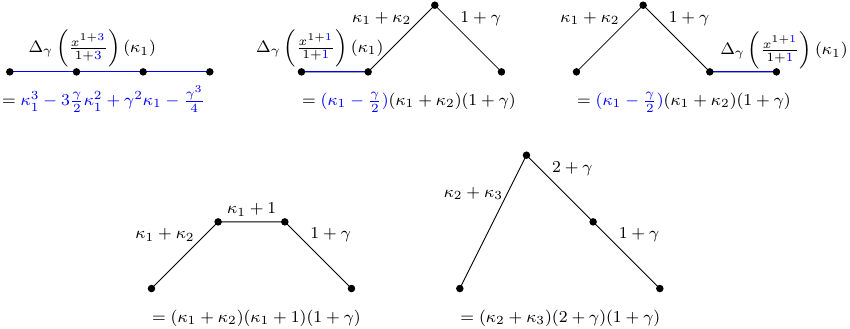}
	\caption{There are five Łukasiewicz paths of length $3$.
The horizontal steps at height~$0$ contribute to the weight of a Łukasiewicz path with a factor involving the divided difference operator $\Delta_\gamma$. The horizontal steps at height~$0$ and the corresponding factors are marked in blue.
The contributions of all other steps in the Łukasiewicz paths are also shown in the figure.}
	\label{fig:PathsEx}
\end{figure} 

Note that
\begin{equation*}
\Delta_\gamma\left( \frac{x^{1+n}}{1+n} \right)(\gamma) = \frac{\gamma^n}{1+n},\quad\text{for all }n\in\Z_{\ge 0},
\end{equation*}
and consequently~\eqref{eq:Moments} leads to the special case stated in the following lemma.

\begin{lemma}\label{lem:special_case}
If $\vec{m} = \mathcal{J}^{\kappa\mapsto m}_\gamma(\vec{\kappa})$, in the sense of \cref{def:mk}, and $\kappa_1=\gamma$, then
\begin{multline*}
m_\ell = \sum_{\Gamma\in\Luk(\ell)}
\frac{1}{1+\text{\#\,horizontal steps at height $0$ in $\Gamma$}}\cdot
\prod_{i\ge 0}(i+\gamma)^{\text{\#\,horizontal steps at height $i$ in $\Gamma$}}\\
\cdot\prod_{j\ge 1} (\kappa_j+\kappa_{j+1})^{\text{\#\,up steps $(1,j)$ in $\Gamma$}}(j+\gamma)^{\text{\#\,down steps from height $j$ in $\Gamma$}},\quad\text{for all $\ell\in\Z_{\ge 1}$}.
\end{multline*}
\end{lemma}

We remark that the following computation
\begin{equation}\label{eqn:A_action}
\Delta_\gamma\left( \frac{x^{1+n}}{1+n} \right)(\kappa_1) = \frac{1}{1+n}\sum_{i=0}^n{ \kappa_1^i(\kappa_1-\gamma)^{n-i} },\quad\text{for all }n\in\Z_{\ge 0},
\end{equation}
allows us to express $m_\ell$, from \cref{def:mk}, in an alternative way as a $\kappa$-weighted sum of Łukasiewicz paths of \emph{length at most $\ell$}, but the new weights do not contain $\Delta_\gamma$.
In other words, \cref{eqn:A_action} allows us to replace the divided difference contribution in \cref{eq:Moments} by more local weights at the cost of introducing paths of lengths smaller than $\ell$.
In the special case when $\kappa_1=\gamma$, \cref{lem:special_case} shows that said formula for $m_\ell$ (without the divided difference contributions) reduces to a $\kappa$-weighted sum of Łukasiewicz paths of length exactly $\ell$.

\begin{theorem}[The main theorem:~HT-appropriateness is equivalent to LLN-satisfaction]\label{theo:main1}
Let $\{\PP_N\}_{N \ge 1}$ be a sequence of finite signed measures with total mass $1$ on $N$-signatures and with small tails.
Then $\{\PP_N\}_{N\ge 1}$ is HT-appropriate (\cref{def:appropriate}) if and only if it satisfies the LLN (\cref{def:LLN}).
If these equivalent conditions hold, then the sequences $\vec{\kappa}=(\kappa_n)_{n\ge 1}$ and $\vec{m}=(m_n)_{n\ge 1}$ are related to each other by $\vec{m}=\mathcal{J}_\gamma^{\kappa\mapsto m}(\vec{\kappa})$, given by \cref{def:mk}.
Furthermore, if all $\PP_N$ are probability measures and moreover there exists a constant $C>0$ such that $|\kappa_n|\le C^n$, for all $n\geq 1$, then there exists a unique probability measure $\mu^\gamma$ on $\R$ with moments $m_\ell$:
\begin{equation*}
m_\ell = \int_{\R}x^\ell \mu^\gamma(\text{d}x),\quad\text{for all }\ell\in\Z_{\ge 1}.
\end{equation*}
\end{theorem}

In the case where all $\PP_N$ are probability measures and the growth condition in the last part of the previous theorem is satisfied, the empirical measures $\mu_N$ of $\PP_N$, defined by~\eqref{eqn:empirical_measures}, converge to $\mu^\gamma$ weakly and in the sense of moments, in probability, as discussed earlier.

\begin{remark}
Note that our combinatorial formula~\eqref{eq:Moments} for the moments is expressed as a sum over weighted lattice paths, which is reminiscent of Viennot's combinatorial construction~\cite{Viennot1985} of moments as sums over lattice paths weighted by the Jacobi parameters of the associated orthogonal polynomials. However, our formula is very different, and it is not clear at all how to ``read'' the associated orthogonal polynomials from our combinatorial expression, which we leave as an open problem here. In Viennot's construction the sum runs only over Motzkin paths, and the Jacobi weights are associated to the horizontal and down steps, while in our case we sum over a larger set of Łukasiewicz paths, the weights are coming from the quantized $\gamma$-cumulants, and additionally we have a non-trivial weight coming from the appearance of the divided difference operator.
\end{remark}

\section{Proof of the LLN at high temperature}\label{sec:main_thm_proof}

\subsection{The action of Cherednik operators at high temperature}

Recall that $\xi_i$ denote the Cherednik operators from equation \eqref{eq:Cherednik}.
Moreover, it will be convenient to define
\begin{equation*}
\mathcal{P}_\ell := \sum_{i=1}^N{\xi_i^\ell},\quad\ell\in\Z_{\ge 1}.
\end{equation*}

Let $\{F_N(x_1,\dots,x_N)\}_{N\ge 1}$ be a sequence of smooth functions in some neighborhoods of $(1^N)=(1,\dots,1)$ such that $F_N(1^N)=0$, for all $N\ge 1$.
Let $k\in\Z_{\ge 1}$ be arbitrary and let the $k$-th order Taylor expansion of $F_N$ at $(1^N)$ be
\begin{equation}\label{eqn:taylor_expansion}
F_N(x_1,\dots,x_N)
= \sum_{\la\in\Y(N)\colon|\la|\le k}{c_{F_N}^\la \, \mathfrak{m}_\la(x_1-1,\dots,x_N-1)} + O(\|\xx-\mathbf{1}\|^{k+1}),
\end{equation}
where $\xx-\mathbf{1}:=(x_1-1,\dots,x_N-1)$ and $\mathfrak{m}_\la$ are the monomial symmetric polynomials.
Note that the coefficients corresponding to the empty partition vanish: $c_{F_N}^\emptyset=0$, for all $N\ge 1$.

The proof of~\cref{theo:main1} employs the next three helper theorems.
The first~\cref{theo:helper1} is a refinement of \cite[Thm.~5.1]{Benaych-GeorgesCuencaGorin2022}, but the argument is similar, so we will keep the proof here fairly brief.
In all these three theorems, we regard $\big\{c_{F_N}^\la\big\}_{\la\in\Y(N)}$ as a sequence of variables of degrees $\deg\big(c_{F_N}^\la\big) := |\la|$.

\begin{theorem}\label{theo:helper1}
Let $k\in\Z_{\ge 1}$, let $\{F_N=F_N(x_1,\dots,x_N)\}_{N\ge 1}$ be a sequence of smooth functions near $(1^N)$ such that $F_N(1^N)=0$, with $k$-th order Taylor expansions of $F_N$ given by~\eqref{eqn:taylor_expansion}, and let $\la=(\la_1\ge\dots\ge\la_s>0)$ be a partition with $|\la|=k$ and $\ell(\la)=s$. For $N\ge s$, we can write
\begin{multline}\label{eqn:cherednik_asymptotics}
N^{-s}\left[ \prod_{j=1}^s{ \mathcal{P}_{\la_j} } \right] e^{F_N} \bigg|_{(x_1,\cdots,x_N)=(1^N)}
= b^\la_\la\,c^\la_{F_N} + \sum_{\mu\colon|\mu|=k,\,\ell(\mu)>s}{ b^\la_\mu\,c^\mu_{F_N} }\\
+ \prod_{j=1}^s{ \frac{1}{N}\cdot\mathcal{P}_{\la_j} e^{F_N} \bigg|_{(x_1,\cdots,x_N)=(1^N)} } 
- \mathbf{1}_{\{s=1\}}\cdot k(\gamma+1)^{\uparrow(k-1)}\cdot c^{(k)}_{F_N}\\
+ R_1\left( c^\nu_{F_N} \colon |\nu|<k \right)
+ R_2\left( c^\nu_{F_N} \colon |\nu|\le k \right),
\end{multline}
where $b^\la_\mu$ are uniformly bounded coefficients in the high temperature regime~\eqref{eq:HTRegime}. In particular,
\begin{equation}\label{eqn:limit_b}
\lim_{\substack{N\to\infty\\N\theta\to\gamma}}{b^\la_\la} = \prod_{j=1}^s{\la_j(\gamma+1)^{\uparrow(\la_j-1)}},
\end{equation}
where we used the notation of rising factorial~\eqref{eq:rising_factorial}.
Moreover, $R_1(c^\nu_{F_N}\colon |\nu|<k)$ is a polynomial in the variables $c^\nu_{F_N}$, $|\nu|<k$, of degree at most $k$ with uniformly bounded coefficients in the regime~\eqref{eq:HTRegime} and such that each monomial has at least one factor $c^\nu_{F_N}$ with $\ell(\nu)>1$.
Finally, $R_2(c^\nu_{F_N} \colon |\nu|\le k)$ is a polynomial in $c^\nu_{F_N}$, $|\nu|\le k$, of degree at most $k$ with coefficients of order $o(1)$ in the regime~\eqref{eq:HTRegime}.
\end{theorem}

\begin{theorem}\label{theo:helper2}
Let $\ell\in\Z_{\ge 1}$ and let $\{F_N=F_N(x_1,\dots,x_N)\}_{N\ge 1}$ be a sequence of smooth functions such that $F_N(1^N)=0$ with $\ell$-th order Taylor expansions at $(1^N)$:
\begin{equation}\label{eq:taylor_expansions_helper}
F_N(x_1,\dots,x_N) = \sum_{\la\in\Y(N)\colon|\la|\leq\ell}{c_{F_N}^\la\,\mathfrak{m}_\la(x_1-1,\dots,x_N-1)} + O(\|\xx-\mathbf{1}\|^{\ell+1}).
\end{equation}
Let $(m_1,m_2,\dots):=\mathcal{J}_\gamma^{\kappa\mapsto m}\Big(c^{(1)}_{F_N},\, 2c^{(2)}_{F_N},\, 3c^{(3)}_{F_N},\cdots\Big)$ be the sequence obtained by \cref{def:mk}.
In particular, $m_\ell=m_\ell\Big(c^{(1)}_{F_N},\cdots,\ell c^{(\ell)}_{F_N}\Big)$ is a polynomial in the variables $n c^{(n)}_{F_N}$, $n=1,\dots,\ell$. Then
\begin{equation}\label{eqn:cherednik_asymptotics_2}
\frac{1}{N}\cdot\mathcal{P}_\ell e^{F_N}\big|_{(x_1,\dots,x_N)=(1^N)}
= m_\ell\left( c^{(1)}_{F_N},\cdots,\ell c^{(\ell)}_{F_N} \right)
+ S_1\left( c^\nu_{F_N} \colon |\nu|\le\ell \right)
+ S_2\left( c^\nu_{F_N} \colon |\nu|\le\ell \right),
\end{equation}
where $S_1\big(c^\nu_{F_N} \colon |\nu|\le\ell\big)$ is a polynomial in $c^\nu_{F_N}$, $|\nu|\le\ell$, of degree at most $\ell$ with uniformly bounded coefficients in the regime~\eqref{eq:HTRegime}, and each monomial has at least one factor $c^\nu_{F_N}$ with $\ell(\nu)>1$.
Similarly, $S_2\big(c^\nu_{F_N} \colon |\nu|\le\ell\big)$ is a polynomial in $c^\nu_{F_N}$, $|\nu|\le\ell$, of degree at most $\ell$ with coefficients of order $o(1)$ in the regime~\eqref{eq:HTRegime}.
\end{theorem}

An immediate corollary of the previous theorem is the following.

\begin{corollary}\label{cor:helper2}
Assume that the sequence of smooth functions $F_N=F_N(x_1,\dots,x_N)$ in \cref{theo:helper2} satisfies

\begin{enumerate}[label=(\alph*),itemsep=0.1cm]

	\item $\displaystyle\lim_{N\to\infty} \frac{\partial_1^n}{(n-1)!}F_N\,\Big|_{(x_1,\dots,x_N) = (1^N)} =: \kappa_n$ exists and is finite, for all $n \in \Z_{\geq 1}$,

	\item $\displaystyle\lim_{N\to\infty} \partial_{i_1}\cdots \partial_{i_r}F_N\,\Big|_{(x_1,\dots,x_N) = (1^N)} = 0$, for any $r\ge 2$ and any $i_1,\dots,i_r\in\Z_{\geq 1}$ with at least two distinct indices among $i_1,\dots,i_r$.

\end{enumerate}
If we let $(m_1,m_2,\dots):=\mathcal{J}_\gamma^{\kappa\mapsto m}(\kappa_1,\kappa_2,\dots)$, then
\begin{equation}\label{eq:limit_ell}
\lim_{\substack{N\to\infty\\N\theta\to\gamma}} { \frac{1}{N}\cdot\mathcal{P}_\ell e^{F_N}\big|_{(x_1,\dots,x_N)=(1^N)} }
= m_\ell,\quad\text{for all }\ell\in\Z_{\ge 1}.
\end{equation}
\end{corollary}

\begin{remark}\label{rem:conditions}
The assumptions (a)-(b) in \cref{cor:helper2} are the same as the ones in \cref{def:appropriate} for $F_N := F_{N,\theta_N}$, where $N\theta_N \to \gamma$, as $N \to \infty$. In terms of the Taylor coefficients $c^\nu_{F_N}$ of $F_N$, these conditions can be expressed as:
\begin{enumerate}[label=(\alph*'),itemsep=0.1cm]

	\item $\displaystyle\lim_{N\to\infty}{\big(n\cdot c^{(n)}_{F_N}\big)} =: \kappa_n$ exists and is finite, for all $n \in \Z_{\geq 1}$,

	\item $\displaystyle\lim_{N\to\infty}{ c^\nu_{F_N} } = 0$, for all $\nu\in\Y(N)$ with $\ell(\nu)>1$.

\end{enumerate}
This reformulation makes the fact that \cref{theo:helper2} implies \cref{cor:helper2} more transparent.
\end{remark}

\begin{theorem}\label{theo:helper3}
Assume that $\{\PP_N\}_{N\ge 1}$ is a sequence of probability measures that is HT-appropriate in the sense of \cref{def:appropriate} and satisfies the LLN in the sense of \cref{def:LLN}.
Let $(\kappa_n)_{n\ge 1}$ and $(m_n)_{n\ge 1}$ be the sequences that arise from these definitions.
Moreover, assume that there exists a constant $C>0$ such that $|\kappa_n|\leq C^n$, for all $n \geq 1$.
Then $(m_n)_{n\ge 1}$ is the sequence of moments of a unique probability measure on $\R$.
\end{theorem}

\subsection{Proof of the main Theorem~\ref{theo:main1}}

Assuming the validity of Theorems~\ref{theo:helper1}, \ref{theo:helper2} and \ref{theo:helper3}, we prove \cref{theo:main1} here.
The argument is similar to the proof of~\cite[Thm.~5.1]{Benaych-GeorgesCuencaGorin2022}.
The three helper theorems will be proved in later subsections.

\begin{lemma}\label{lem:action_cherednik}
Let $\PP_N$ be any finite signed measure on $N$-signatures with small tails and with Jack generating function $G_{N,\theta}(x_1,\dots,x_N)$.
For any $s\in\Z_{\ge 1}$ and $k_1,\dots,k_s\ge 1$, we have
\begin{equation}\label{eq:momentsInCherednik}
\left[ \prod_{j=1}^s{\mathcal{P}_{k_j}} \right] G_{N,\theta}\,\bigg|_{(x_1,\dots,x_N)=(1^N)}
= \E_{\PP_N} \left[ \prod_{j=1}^s{ \sum_{i=1}^N{\LL_i^{k_j}} } \right].
\end{equation}
\end{lemma}


\begin{proof}
The desired equation~\eqref{eq:momentsInCherednik} follows from writing $G_{N,\theta}$ as the Laurent series \eqref{eq:JackGeneratFunction}, then exchanging the order between the operator $\prod_{j=1}^s{\mathcal{P}_{k_j}}$ and the infinite sum defining $G_{N,\theta}$, and finally applying \cref{thm:cherednik_ops}, which shows that $\prod_{j=1}^s{\mathcal{P}_{k_j}}$ acts on each $P_\la(x_1,\dots,x_N;\theta)$ as the eigenvalue $\prod_{j=1}^s{ \sum_{i=1}^N{\LL_i^{k_j}} }$.
The condition of $\PP_N$ having small tails is needed to argue that the series \eqref{eq:JackGeneratFunction} converges and that we can exchange the order between the operators $\mathcal{P}_{k_j}$ and the sum.
Note also that the operators $\mathcal{P}_{k_j}$ are not differential operators because of the transpositions $s_{i,j}$ from formula~\eqref{eq:Cherednik}, so it is not immediately obvious that we can exchange the order between them and the summation.
However, a result due to Heckman \cite[Theorem 1.7]{Heckman_1991} shows that each $\mathcal{P}_{k_j}$ is equivalent to a differential operator when restricted to the ring of symmetric functions in the variables $x_1,\dots,x_N$; this justifies our argument.
\end{proof}

\textbf{HT-appropriateness implies LLN-satisfaction.}
We prove that if $\{\PP_N\}_{N\ge 1}$ is HT-appropriate with sequence $\vec{\kappa}=(\kappa_n)_{n\ge 1}$, then it satisfies the LLN with $\vec{m}=(m_n)_{n\ge 1}:=\mathcal{J}_\gamma^{\kappa\mapsto m}(\vec{\kappa})$, i.e.~we verify the condition from \cref{def:LLN}, for any $k_1,\dots,k_s\ge 1$.

Assume w.l.o.g. that $k_1\ge\dots\ge k_s$.
By \cref{lem:action_cherednik} and \cref{theo:helper1}, $N^{-s}\cdot\E_{\PP_N} \left[ \prod_{j=1}^s{ \sum_{i=1}^N{\LL_i^{k_j}} } \right]$ is equal to the RHS of~\eqref{eqn:cherednik_asymptotics} with $\la=(k_1,\dots,k_s)$ and $F_N$ replaced by $F_{N,\theta}=\ln(G_{N,\theta})$ (we will drop the dependence of $\theta$ on $N$ in the notation for brevity, but we should keep in mind that we choose any sequence $\theta_N$ such that $N\theta_N\to\gamma$, as $N \to \infty$, as in \cref{rem:conditions}).
It remains to prove that this converges to $\prod_{j=1}^s{m_{k_j}}$ in the regime~\eqref{eq:HTRegime}.
Indeed, $R_2\to 0$ because of the hypothesis on its coefficients, whereas $R_1\to 0$ because each monomial contains a factor
$c^\nu_{F_{N,\theta}}$ with $\ell(\nu)>1$ and this converges to zero by HT-appropriateness.
The terms $b^\la_\mu\,c^\mu_{F_{N,\theta}}$ in the sum have $\ell(\mu)>1$, so they converge to zero.
The term $b^\la_\la\,c^\la_{F_{N,\theta}}$ converges to zero, except when $\ell(\la)=1$, in which case $\la=(k_1)$ and so $b^\la_\la\,c^\la_{F_{N,\theta}}\to(\gamma+1)^{\uparrow(k_1-1)}\cdot\kappa_{k_1}$ by \eqref{eqn:limit_b} and the assumption $k_1\cdot c^{(k_1)}_{F_{N,\theta}}\to\kappa_{k_1}$.
Finally, by~\cref{cor:helper2}, the second line in the RHS of~\eqref{eqn:cherednik_asymptotics} converges to $\prod_{j=1}^s{m_{k_j}}-\mathbf{1}_{\{\ell(\la)=1\}}\cdot(\gamma+1)^{\uparrow(k_1-1)}\cdot\kappa_{k_1}$.
By putting everything together, the RHS of ~\eqref{eqn:cherednik_asymptotics} converges to $\prod_{j=1}^s{m_{k_j}}$, as desired.

\smallskip

\textbf{LLN-satisfaction implies HT-appropriateness.}
Assume that $\{\PP_N\}_{N\ge 1}$ satisfies the LLN with $\vec{m}=(m_n)_{n\ge 1}$; let us prove HT-appropriateness with $\vec{\kappa}=(\kappa_n)_{n\ge 1}:=\mathcal{J}_\gamma^{m\mapsto\kappa}(\vec{m})$.
By \cref{rem:conditions}, we need to prove the following limits:

\begin{enumerate}[label=(\alph*'),itemsep=0.1cm]

	\item $\lim_{\substack{N\to\infty\\N\theta\to\gamma}}{c^{(n)}_{F_{N,\theta}}} = \frac{\kappa_n}{n}$,

	\item $\lim_{\substack{N\to\infty\\N\theta\to\gamma}}{ c^\nu_{F_{N,\theta}} } = 0$, for all $\nu\in\Y(N)$ with $\ell(\nu)>1$ and $|\nu|=n$.

\end{enumerate}

\noindent The proof is by induction on $n$. Let us start with the base case $n=1$.
We only need to verify (a') for $n=1$, as condition (b') is non-existent for $n=1$.
By the assumption of LLN-satisfaction and \cref{lem:action_cherednik}, we have that
\begin{equation}\label{base_case_1}
\lim_{\substack{N\to\infty\\N\theta\to\gamma}}{ \frac{1}{N}\,\mathcal{P}_1e^{F_{N,\theta}}\big|_{(x_1,\dots,x_N)=(1^N)} } = m_1.
\end{equation}
By definition, $\mathcal{P}_1=\xi_1+\dots+\xi_N$, where $\xi_i$'s are the Cherednik operators~\eqref{eq:Cherednik}.
Since $F_{N,\theta}$ is symmetric in its arguments, it is killed by operators $1-s_{i,j}$, therefore $\mathcal{P}_1e^{F_{N,\theta}} = \sum_{i=1}^N(x_i\partial_i - \theta(i-1))e^{F_{N,\theta}}$.
The symmetry also shows that $x_1\partial_1F_{N,\theta}\big|_{(x_1,\dots,x_N)=(1^N)} = x_i\partial_iF_{N,\theta}\big|_{(x_1,\dots,x_N)=(1^N)}$, for all $i=1,\dots,N$.
Together with the fact that $F_{N,\theta}(1^N)=0$, we deduce
\begin{equation}\label{base_case_2}
\begin{aligned}
\frac{1}{N}\,\mathcal{P}_1e^{F_{N,\theta}}\big|_{(x_1,\dots,x_N)=(1^N)}
&= \frac{1}{N} \Big( Nx_1\partial_1  - \theta\,\frac{N(N-1)}{2} \Big)e^{F_{N,\theta}} \Big|_{(x_1,\dots,x_N)=(1^N)}\\
&= \partial_1F_{N,\theta}\big|_{(x_1,\dots,x_N)=(1^N)} - \frac{\theta(N-1)}{2}.
\end{aligned}
\end{equation}
From \eqref{base_case_1} and \eqref{base_case_2}, we obtain
\begin{equation}\label{base_case_3}
\lim_{\substack{N\to\infty\\N\theta\to\gamma}}\partial_1F_{N,\theta}\big|_{(x_1,\dots,x_N)=(1^N)} = m_1+\frac{\gamma}{2}.
\end{equation}
Since $\kappa_1=m_1+\gamma/2$ (see~\cref{ex:mk}), and the Taylor coefficient $c^{(1)}_{F_{N,\theta}}$ can be expressed as $c^{(1)}_{F_{N,\theta}}=\partial_1F_{N,\theta}\big|_{(x_1,\dots,x_N)=(1^N)}$ (see Equation~\eqref{eqn:taylor_expansion}), it follows that \eqref{base_case_3} is precisely (a') for $n=1$.

\medskip

For the inductive step, take any $k\ge 2$ and assume that (a')-(b') above are true for all $n\le k-1$; to complete the argument, it remains to prove these conditions for $n=k$.

We first prove (b') for all partitions of size $k$ with more than one part.
Take any $\la=(\la_1\ge\dots\ge\la_s)$ with $\ell(\la)=s>1$, $|\la|=k$, and write Equation~\eqref{eqn:cherednik_asymptotics} for this partition.
The LHS of this equation converges to $\prod_{j=1}^s m_{\la_j}$, by \cref{lem:action_cherednik} and LLN-satisfaction.
Since $s>1$, we have $\la_j<k$ for every $j=1,\dots,s$.
Therefore, by the induction hypothesis and \cref{eqn:cherednik_asymptotics_2}, applied with $\ell=\la_j$, we have
\[
\lim_{\substack{N\to\infty\\N\theta\to\gamma}}\frac{1}{N}\mathcal{P}_{\la_j}e^{F_{N,\theta}}\big|_{(x_1,\dots,x_N)=(1^N)}
= m_{\la_j},\qquad j=1,\dots,s.
\]
Indeed, every monomial of the corresponding $S_1$ contains a factor $c^\nu_{F_{N,\theta}}$ with $\ell(\nu)>1$ and $|\nu|\le\la_j<k$, so it tends to zero by the induction hypothesis; the $S_2$ term tends to zero by definition.

Multiplying these equations and comparing with \eqref{eqn:cherednik_asymptotics}, we obtain, for every partition $\la\vdash k$ with $\ell(\la)>1$,
\[
b^\la_\la c^\la_{F_{N,\theta}}
+\sum_{\substack{\mu\vdash k\\ \ell(\mu)>\ell(\la)}}b^\la_\mu c^\mu_{F_{N,\theta}}
=o(1).
\]
Let $p(k)$ be the number of partitions of size $k$.
By letting $\la$ range over all partitions of size $k$ with more than one part, we obtain a system of $p(k)-1$ equations in the $p(k)-1$ variables
\[
c^\mu_{F_{N,\theta}},\qquad |\mu|=k,\quad \ell(\mu)>1.
\]
Equation~\eqref{eqn:limit_b} implies that its coefficient matrix is asymptotically triangular with nonzero diagonal entries, therefore
\[
\lim_{\substack{N\to\infty\\N\theta\to\gamma}} c^\mu_{F_{N,\theta}} = 0,\qquad |\mu|=k,\quad \ell(\mu)>1.
\]

It remains to prove (a') for $n=k$.
Take \cref{eqn:cherednik_asymptotics_2} with $\ell=k$.
Its LHS converges to $m_k$. Every monomial of $S_1$ on the RHS contains at least one coefficient with $\ell(\nu)>1$.
Such coefficients tend to zero by the induction hypothesis when $|\nu|<k$, and by the preceding paragraph when $|\nu|=k$.
Therefore $S_1\to0$, while $S_2\to0$ by definition.

Finally, since $kc^{(k)}_{F_{N,\theta}}$ appears in
$m_k(c^{(1)}_{F_{N,\theta}},\dots,kc^{(k)}_{F_{N,\theta}})$ only through the linear monomial
$(\gamma+1)^{\uparrow(k-1)}kc^{(k)}_{F_{N,\theta}}$, by~\eqref{eq:top_moments}, the induction hypothesis gives
\[
m_k(c^{(1)}_{F_{N,\theta}},\dots,kc^{(k)}_{F_{N,\theta}})
=
m_k+(\gamma+1)^{\uparrow(k-1)}
\big(kc^{(k)}_{F_{N,\theta}}-\kappa_k\big)+o(1).
\]
Comparison with the limiting LHS of \eqref{eqn:cherednik_asymptotics_2} yields
\[
\lim_{\substack{N\to\infty\\N\theta\to\gamma}} kc^{(k)}_{F_{N,\theta}} = \kappa_k.
\]
This proves (a')-(b') for $n=k$ and completes the induction.

\smallskip

\textbf{The moment problem.}
The fact that if $|\kappa_n|\le C^n$, for some constant $C>0$, then $m_1,m_2,\dots$ are the moments of a unique probability measure on $\R$ is exactly the content of \cref{theo:helper3}. This ends the proof of \cref{theo:main1}.

\subsection{Method of moments:~Proof of Theorem~\ref{theo:helper1}}

Consider the $k$-th order truncations
\begin{equation}\label{eq:Trunc}
\tilde{F}_N(x_1,\dots,x_N) := \sum_{\la\colon |\la|\le k,\,\ell(\la)\leq N}c_{F_N}^\la \, \mathfrak{m}_\la(x_1-1,\dots,x_N-1).
\end{equation}

The operators $x_i\partial_i$ and $\frac{x_\ell}{x_i-x_j}(1-s_{i,j})$ for some
$i,j,\ell$, are constituents of the Cherednik operators~\eqref{eq:Cherednik}. Their actions on \emph{shifted monomials}, that is, monomials in the variables $x_1-1,\dots,x_N-1$, are the following:
\begin{multline}\label{eq:actionOnMon}
x_i\partial_i \left[ (x_1-1)^{\alpha_1}\cdots(x_N-1)^{\alpha_N} \right] =
\alpha_i\bigg((x_1-1)^{\alpha_1}\cdots(x_N-1)^{\alpha_N} +\\
\hfill+ (x_1-1)^{\alpha_1}\cdots(x_i-1)^{\alpha_i-1}\cdots(x_N-1)^{\alpha_N}\bigg),\\
\frac{x_\ell}{x_i-x_j}(1-s_{i,j}) \left[ (x_1-1)^{\alpha_1}\cdots(x_N-1)^{\alpha_N} \right]
= \pm x_\ell
(x_1-1)^{\alpha_1}\cdots(x_i-1)^{a}\cdots(x_j-1)^{a}\cdots(x_N-1)^{\alpha_N}\\
\times\sum_{c=0}^{b-a-1}{(x_i-1)^c (x_j-1)^{b-a-1-c}},
\end{multline}
where $a := \min(\alpha_i,\alpha_j)$, $b := \max(\alpha_i,\alpha_j)$, and the sign is positive (negative, respectively) if $a = \alpha_j$ ($a = \alpha_i$, respectively). This leads to the following lemma.

\begin{lemma}\label{lem:Polynomiality}
Let $1\leq i_1,\dots,i_k\leq N$ be arbitrary integers, not necessarily distinct. Then
\begin{equation}\label{eq:MainQuant}
\left(\prod_{j=1}^k \xi_{i_j}\right) \exp(F_N)\Big|_{(x_1,\dots,x_N)=(1^N)}
= \left(\prod_{j=1}^k \xi_{i_j}\right) \exp(\tilde{F}_N)\Big|_{(x_1,\dots,x_N)=(1^N)}.
\end{equation}
Moreover, this is a polynomial in the variables $c_{F_N}^\la$, $|\la|\le k$, of degree at most $k$, and with coefficients that are uniformly bounded over all $i_1,\dots,i_k$ in the regime~\eqref{eq:HTRegime}.
\end{lemma}

\begin{proof}
Note that
\begin{equation*}
\exp(F_N(x_1,\dots,x_N)) = \exp(\tilde{F}_N(x_1,\dots,x_N)) + R(x_1,\dots,x_N),
\end{equation*}
where $R(x_1,\dots,x_N) = O(\|\xx-\mathbf{1}\|^{k+1})$.
This means that $R(x_1,\dots,x_N)$ vanishes up to degree $(k+1)$ at $(1^N)$.
After applying any sequence of $k$ operators $\theta(i-1)$, $x_i\partial_i$ or $\frac{x_\ell}{x_i-x_j}(1-s_{i,j})$, we obtain a continuously differentiable function that still vanishes at $(1^N)$. Since $\prod_{j=1}^k \xi_{i_j}$ is a linear combination of sequences of $k$ operators as just described, we obtain the desired~\eqref{eq:MainQuant}.

We now prove the polynomiality part. Consider any product of $k$ operators, each of the form
$\theta(i-1)$, $x_i\partial_i$ or $\frac{x_\ell}{x_i-x_j}(1-s_{i,j})$, for some $i,j,\ell$. We want to apply them inductively to
$\exp(\tilde{F}_N)$ using the following rules:
\begin{equation}\label{eq:DerRules}
\begin{gathered}
x_i\partial_i \left(H(\xx)\cdot \exp(\tilde{F}_N(\xx))\right) =
\left(x_i\big(\partial_iH(\xx)+H(\xx)\cdot\partial_i\tilde{F}_N(\xx)\big) \right)\cdot\exp(\tilde{F}_N(\xx)),\\
\frac{x_\ell}{x_i-x_j}(1-s_{i,j})\left(H(\xx)\cdot\exp(\tilde{F}_N(\xx))\right) =
\left(\frac{x_\ell}{x_i-x_j}(1-s_{i,j})H(\xx) \right)\cdot\exp(\tilde{F}_N(\xx)).
\end{gathered}
\end{equation}
Taking into account that $\tilde{F}_N(1^N) = 0$, the result of having $k$ such operators act on $\exp(\tilde{F}_N(\xx))$
and then setting all variables equal to $1$ is a finite linear combination of products of actions of $\theta(i-1)$, $x_i\partial_i$ and $\frac{x_\ell}{x_i-x_j}(1-s_{i,j})$ on $\tilde{F}_N(\xx)$, followed by evaluation at $\xx=(1^N)$.
Since $\tilde{F}_N(\xx)$ is a polynomial of the shifted variables $x_1-1,\dots,x_N-1$ with coefficients $c^\la_{F_N}$, and the actions of $x_i\partial_i$, $\frac{x_\ell}{x_i-x_j}(1-s_{i,j})$ on shifted monomials are given by~\eqref{eq:actionOnMon}, we conclude the polynomiality. It remains to verify the statements about the degree and uniform boundedness.

The previous argument shows in fact that the rules~\eqref{eq:actionOnMon} and~\eqref{eq:DerRules} imply that the action of $\xi_{i_q}\cdots\xi_{i_1}$ on $\exp(\tilde{F}_N)$ gives an expression of the form
\begin{equation*}
\xi_{i_q}\cdots\xi_{i_1}\exp(\tilde{F}_N) = \left( \sum_{r_1,\dots,r_N\ge 0}{ h^{(q)}_{r_1,\dots,r_N}(x_1-1)^{r_1}\cdots (x_N-1)^{r_N} } \right)\cdot\exp(\tilde{F}_N),
\end{equation*}
for all $q=1,\dots,k$, and some coefficients $h^{(q)}_{r_1,\dots,r_N}$. At the end, since $\tilde{F}_N(1^N)=0$, we have
\begin{equation*}
\xi_{i_k}\cdots\xi_{i_1}\exp(\tilde{F}_N)\big|_{(x_1,\dots,x_N)=(1^N)} = h^{(k)}_{0,\dots,0}.
\end{equation*}
We claim that each $h^{(q)}_{r_1,\dots,r_N}$, with $1\le q\le k$ and $r_1,\dots,r_N\ge 0$, is a polynomial in the $c^\la_{F_N}$, $|\la|\le k$, of degree at most $r_1+\dots+r_N+q$, and with coefficients that are uniformly bounded in the regime~\eqref{eq:HTRegime}. In particular, this will be true of $h^{(k)}_{0,\dots,0}$ and will end the proof.

The previous claim follows from the explicit expression for the Cherednik operators~\eqref{eq:Cherednik}.
Indeed, the action of Cherednik operators on a homogeneous polynomial of degree $r$ in the shifted variables $x_1-1,\dots,x_N-1$ can be expressed as an inhomogeneous polynomial in these shifted variables, with nonzero components only in degrees $r-1$ and $r$; this is a consequence of the calculations~\eqref{eq:actionOnMon}.
This observation and a simple induction on $q$ proves the claim about the degree of $h^{(q)}_{r_1,\dots,r_N}$.

For the claim about uniform boundedness, note that the rules~\eqref{eq:DerRules} allow one to express each $h^{(q+1)}_{r_1,\dots,r_N}$ as a linear combination of $h^{(q)}_{s_1,\dots,s_N}$, over $s_1,\dots,s_N\ge 0$.
In this linear combination, the number of coefficients coming from the action of constituents $x_{i_{q+1}}\partial_{i_{q+1}}$ and $-\theta(i_{q+1}-1)$ is $O(1)$, whereas the number of coefficients coming from the action of constituents of the form $\theta\frac{x_\ell}{x_i-x_j}(1-s_{i,j})$, with $i$ or $j$ being $i_{q+1}$, is $O(N)$.
Since $|-\theta(i_{q+1}-1)|\le\theta N$ and $\theta N\sim\gamma$, it follows by induction that for $q=1,\dots,k$, the coefficients of $h^{(q)}_{s_1,\dots,s_N}$, as a polynomial in the variables $c^\lambda_{F_N}$, can be bounded by a constant that only depends on $\gamma$, as desired.
\end{proof}

By virtue of \cref{lem:Polynomiality} we can, and will, assume without loss of generality that $F_N(x_1,\dots,x_N)$ is the polynomial given by the RHS of \eqref{eq:Trunc}:
\begin{equation}\label{eq:truncated}
F_N(x_1,\dots,x_N) = \sum_{\la\colon |\la|\le k,\,\ell(\la)\leq N}c_{F_N}^\la\,\mathfrak{m}_\la(x_1-1,\dots,x_N-1).
\end{equation}
In particular, $F_N$ is symmetric in the variables $x_1,\dots,x_N$.

\begin{lemma}\label{lem:TopDegree_2}
Let $s\in\Z_{\ge 1}$ be arbitrary and let $\alpha_1,\dots,\alpha_s\in\Z_{\geq 1}$ be such that $\alpha_1+\cdots+\alpha_s\leq k$. Further, let $1\le i_1,\dots,i_s\le N$ be any distinct integers. Then
\begin{equation*}
\left(\prod_{j=1}^s \xi_{i_j}^{\alpha_j}\right) \exp(F_N)\bigg|_{(x_1,\dots,x_N)=(1^N)}
= \prod_{j=1}^s \left(\xi_{i_j}^{\alpha_j}\exp(F_N) \right) \Big|_{(x_1,\dots,x_N)=(1^N)} + R + O\left( \frac{1}{N} \right),
\end{equation*}
where $R$ is a polynomial in the $c_{F_N}^\la$ such that $\deg(R)\le k$, it has coefficients of order $O(1)$ in the regime~\eqref{eq:HTRegime}, uniformly over all $1\le i_1,\dots,i_s\le N$, and such that each monomial in it contains at least a factor $c_{F_N}^\nu$ with $\ell(\nu)>1$.
Moreover, $O(\frac{1}{N})$ is a polynomial in the $c_{F_N}^\la$ of degree $\le k$, whose coefficients are $O(\frac{1}{N})$ in the regime~\eqref{eq:HTRegime}, uniformly over all $1\le i_1,\dots,i_s\le N$.
\end{lemma}
\begin{proof}
When computing the left hand side of the equality, the coefficients $c_{F_N}^\nu$ with $\ell(\nu)>1$ that appear will be part of $R$ in the right hand side, and the uniformity of bounds on the coefficients of $R$ is already proved in \cref{lem:Polynomiality}, for this lemma we can replace $F_N(\xx)$ in~\eqref{eq:truncated} by
\begin{equation}\label{eq:H_series}
H_N(\xx) := \sum_{n=0}^k{c_{F_N}^{(n)}\mathfrak{m}_{(n)}(x_1-1,\dots,x_N-1)}
= \sum_{l=1}^N\sum_{n=0}^k{c_{F_N}^{(n)}(x_l-1)^n}.
\end{equation}
Then the desired statement reduces to proving:
\begin{equation}\label{eqn:factorization}
\xi_{i_s}^{\alpha_s}\cdots\xi_{i_1}^{\alpha_1} \exp(H_N) \big|_{(x_1,\dots,x_N)=(1^N)}\\
= \prod_{j=1}^s \left[ \xi_{i_j}^{\alpha_j} \exp(H_N) \right] \bigg|_{(x_1,\dots,x_N)=(1^N)} + O\left( \frac{1}{N} \right).
\end{equation}
Note that $\xi_{i_s}^{\alpha_s}\cdots\xi_{i_1}^{\alpha_1}$ will be a
sum of $(N+1)^{\sum_{j=1}^s{\alpha_j}}$ words $w_1^s\cdots
w_{\alpha_s}^s\cdots w_1^1\cdots w_{\alpha_1}^1$, where each $w_b^a$
is of the form $x_{i_a}\partial_{i_a}$ or $-\theta(i_a-1)$ or
$\frac{\theta x_\ell}{x_{i_a}-x_j}(1-s_{i_a,j})$, for some $1 \leq j < i_a$
and $\ell = i_a$, or $i_a < j \leq N$
and $\ell = j$.
If a word $w_1^s\cdots w_{\alpha_s}^s\cdots w_1^1\cdots w_{\alpha_1}^1$ satisfies:

\smallskip

$\bullet$ some $w^a_b=\frac{\theta x_\ell}{x_{i_a}-x_j}(1-s_{i_a,j})$ and $j
\in \{i_1,\dots,i_s\}$, or

$\bullet$ some $w^a_b=\frac{\theta x_\ell}{x_{i_a}-x_j}(1-s_{i_a,j}),\
w^{a'}_{b'}=\frac{\theta x_{\ell'}}{x_{i_{a'}}-x_{j'}}(1-s_{i_{a'},j'})$, such
that $a \neq a'$, and $j=j'$,

\smallskip

\noindent then we call it \emph{singular} and if it does not satisfy any of the two conditions above, we call it \emph{generic}.

There are at least $(N-s)(N-s-1)\cdots (N-s-\sum_{j=1}^s{\alpha_j}+1)$ generic words, where all letters $w^a_b$ are divided difference operators $\frac{\theta x_\ell}{x_i-x_j}(1-s_{i,j})$, as can be seen by picking the letters $w_1^s,w_2^s,\cdots,w^{1}_{\alpha_1}$ in that order, preventing at each step both conditions for singularity.
Moreover, the total number of words with only divided difference operators as letters is at most $N^{\sum_{j=1}^s{\alpha_j}}$.
Then the number of singular words with only divided difference operators as letters is at most the difference $N^{\sum_{j=1}^s{\alpha_j}}-(N-s)(N-s-1)\cdots (N-s-\sum_{j=1}^s{\alpha_j}+1)$, which is independent of $\theta$ and of order $N^{\sum_{j=1}^s{\alpha_j}-1}$.
Each of those singular words contains a factor of $\theta^{\sum_{j=1}^s{\alpha_j}}$.
As a result, the contribution from singular terms with only divided difference operators as letters is $O\big(N^{\sum_{j=1}^s{\alpha_j}-1}\cdot\theta^{\sum_{j=1}^s{\alpha_j}}\big) = O(\frac{1}{N})$.

More generally, note that any singular word has at least one divided difference operator $\frac{\theta x_\ell}{x_i-x_j}(1-s_{i,j})$ as letter, and at most $\sum_{j=1}^s{\alpha_j}$ of them.
Let $1\le m\le\sum_{j=1}^s{\alpha_j}$ be fixed.
The same argument as above shows that the number of singular words with exactly $m$ divided difference operators as letters is independent of $\theta$ and of order $N^{m-1}$.
Moreover, each of them comes with a factor of $\theta^m$.
Consequently, the contribution from singular terms with exactly $m$ divided difference operators as letters is also $O\big(N^{m-1}\cdot\theta^m\big) = O(\frac{1}{N})$.

This argument implies that, in order to prove~\eqref{eqn:factorization}, it suffices to consider the generic terms $w_1^s\cdots w_{\alpha_s}^s\cdots w_1^1\cdots w_{\alpha_1}^1 \exp(H_N) \big|_{(x_1,\dots,x_N)=(1^N)}$ on the LHS, because the contribution coming from all singular ones will be part of the $O(\frac{1}{N})$ in~\eqref{eqn:factorization}.

By the rules~\eqref{eq:actionOnMon}, \eqref{eq:DerRules}, we deduce
\begin{equation}\label{eqn:factorization_2}
w_1^1\cdots w_{\alpha_1}^1 \exp(H_N) = h_1\cdot\exp(H_N),
\end{equation}
where $h_1=h_1(\xx)$ is a polynomial that involves only variables $x_{i_1}$ and $x_j$, for those $j$ that are part of some operator $w^1_a = \frac{\theta x_\ell}{x_{i_1}-x_j}(1-s_{i_1,j})$; this part of the argument is where the special form of $H_N(\xx)$ in~\eqref{eq:H_series} is used.
Then the next step is to apply $w_1^2\cdots w_{\alpha_2}^2$ to~\eqref{eqn:factorization_2}.
By the condition of being generic and the rules~\eqref{eq:DerRules}, the operators $w^2_a$ do not act on $h_1$, but only on $\exp(H_N)$. Hence
\[
w_1^2\cdots w_{\alpha_2}^2 w_1^1\cdots w_{\alpha_1}^1 \exp(H_N)
=h_1\cdot w_1^2\cdots w_{\alpha_2}^2\exp(H_N).
\]
Since $H_N(1^N)=0$, evaluating this identity at $(1^N)$ gives
\begin{multline*}
\left.w_1^2\cdots w_{\alpha_2}^2 w_1^1\cdots w_{\alpha_1}^1 \exp(H_N)\right|_{(x_1,\dots,x_N)=(1^N)}\\
=
\left.w_1^1\cdots w_{\alpha_1}^1\exp(H_N)\right|_{(x_1,\dots,x_N)=(1^N)}\cdot
\left.w_1^2\cdots w_{\alpha_2}^2\exp(H_N)\right|_{(x_1,\dots,x_N)=(1^N)}.
\end{multline*}
By repeating this argument $s$ times, we deduce that for every generic word,
\begin{equation*}
\left.w_1^s\cdots w_{\alpha_s}^s\cdots w_1^1\cdots w_{\alpha_1}^1 \exp(H_N)\right|_{(x_1,\dots,x_N)=(1^N)}
= \prod_{j=1}^s \left(\left.w_1^j\cdots w_{\alpha_j}^j \exp(H_N)\right|_{(x_1,\dots,x_N)=(1^N)}\right).
\end{equation*}
Adding these identities over all generic words and noting that the singular terms contribute an additive uniform error of $O(\frac{1}{N})$, as discussed earlier, we conclude the desired~\eqref{eqn:factorization}.
\end{proof}

\begin{lemma}\label{lem:TopDegree}
Let $s\in\Z_{\ge 1}$ be arbitrary and let $1 \leq i_1,\dots,i_s \leq N$ be distinct integers.
Further, let $\la=(\alpha_1\ge\dots\ge\alpha_s>0)$ be a partition with $\ell(\la)=s$ and $|\la|=\alpha_1+\cdots+\alpha_s=k$.
Then
\begin{equation}\label{eq:MainQuant1}
\left(\prod_{j=1}^s \xi_{i_j}^{\alpha_j}\right) \exp(F_N)\big|_{(x_1,\dots,x_N)=(1^N)}
= b^\la_\la c^\la_{F_N} + \sum_{\mu\colon |\mu|=k,\,\ell(\mu)>s}{b^\la_\mu c^\mu_{F_N}} + R + O\left( \frac{1}{N} \right),
\end{equation}
where the coefficients $b^\la_\mu$ are uniformly bounded in the regime~\eqref{eq:HTRegime}; in particular,
\begin{equation}\label{limit_b_diagonal}
\lim_{\substack{N\to\infty\\N\theta\to\gamma}}{b^\la_\la} = \prod_{j=1}^s{ \alpha_j(\gamma+1)^{\uparrow(\alpha_j-1)} }.
\end{equation}
Moreover, $R$ is a polynomial in the $c^\la_{F_N}$, $|\la|<k$ (i.e.~it does not involve $c^\nu_{F_N}$ with $|\nu|=k$) with $\deg(R)\le k$, and $O(\frac{1}{N})$ is a polynomial in the $c^\la_{F_N}$ of degree $\le k$, whose coefficients are $O(\frac{1}{N})$ in the regime~\eqref{eq:HTRegime}.
\end{lemma}

\begin{proof}
According to \cref{lem:Polynomiality}, the LHS
of~\eqref{eq:MainQuant1} is a polynomial of degree at most $k$ in the coefficients $c^\la_{F_N}$, so its linear degree $k$ component is of the form
\begin{equation}\label{linear_component}
\sum_{\mu\colon|\mu|=k}{b^\la_\mu c^\mu_{F_N}}.
\end{equation}
It remains to prove that $b^\la_\mu=O(\frac{1}{N})$, unless $\ell(\mu)>s$ or $\mu=\la$, as well as the limit~\eqref{limit_b_diagonal}.

Consider the polynomial
\begin{equation}\label{eqn:F_hat}
\hat{F}_N(x_1,\dots,x_N) := \sum_{\la\colon |\la|\le k,\,\ell(\la)\leq N}{c_{F_N}^\la\, \mathfrak{m}_\la(x_1,\dots,x_N)},
\end{equation}
which is related to $F_N$ by $F_N(x_1,\dots,x_N)=\hat{F}_N(x_1-1,\dots,x_N-1)$.
We also consider the following rational \emph{Dunkl operators}~\cite{Dunkl1989},
\begin{equation}\label{eqn:dunkl}
\tilde{\xi}_i := \partial_i+\theta\sum_{j\neq i}\frac{1}{x_i-x_j}(1-s_{i,j}),\quad i=1,2,\dots,N.
\end{equation}
Note that the Dunkl operators decrease the degree of polynomials by $1$, whereas the Cherednik operators~\eqref{eq:Cherednik} preserve the degree of polynomials.
Moreover, by looking at the rules~\eqref{eq:DerRules} of how the Cherednik operators act, and at the analogous relations for Dunkl operators~\cite[Eqns.~(5.8)-(5.9)]{Benaych-GeorgesCuencaGorin2022}, we deduce
\begin{equation*}
\left(\prod_{j=1}^s \xi_{i_j}^{\alpha_j}\right) \exp(F_N)
= \left(\prod_{j=1}^s (\xi_{i_j}+\tilde{\xi}_{i_j})^{\alpha_j}\right) \exp(\hat{F}_N)\bigg|_{x_i\mapsto x_i-1,\ \forall i=1,\dots,N}
\end{equation*}
and therefore
\begin{equation}\label{cherednik_dunkl}
\left(\prod_{j=1}^s \xi_{i_j}^{\alpha_j}\right) \exp(F_N)\bigg|_{(x_1,\dots,x_N)=(1^N)}
= \left(\prod_{j=1}^s (\xi_{i_j}+\tilde{\xi}_{i_j})^{\alpha_j}\right) \exp(\hat{F}_N)\bigg|_{(x_1,\dots,x_N)=(0^N)}.
\end{equation}
To calculate the RHS of~\eqref{cherednik_dunkl}, write $\prod_{j=1}^s (\xi_{i_j}+\tilde{\xi}_{i_j})^{\alpha_j}$ as a sum of $2^k$ sequences of $k$ operators, each being either $\xi_{i_j}$ or $\tilde{\xi}_{i_j}$, then apply each sequence of operators to $\exp(\hat{F}_N)$ in order to obtain an expression of the form $h\cdot\exp(\hat{F}_N)$, and finally pick up the constant term of $h=h(\xx)$.
Since we are only interested in the degree $k$ linear component~\eqref{linear_component} and the Cherednik operators preserve the degree of polynomials, any sequence involving some Cherednik operator $\xi_{i_j}$ will not contribute to~\eqref{linear_component}.
As a result, the part of the RHS of~\eqref{cherednik_dunkl} that contributes to~\eqref{linear_component} comes from
\begin{equation}\label{eqn:dunkl_part}
\left(\prod_{j=1}^s \tilde{\xi}_{i_j}^{\,\alpha_j}\right) \exp(\hat{F}_N)\bigg|_{(x_1,\dots,x_N)=(0^N)}.
\end{equation}
It was precisely proved in \cite[Prop.~5.5]{Benaych-GeorgesCuencaGorin2022} that the coefficients of the linear component $\sum_{\mu\colon|\mu|=k}{b^\la_\mu c^\mu_{F_N}}$ of \eqref{eqn:dunkl_part} are of order $b^{\la}_\mu=O(\frac{1}{N})$, unless $\ell(\mu)>\ell(\la)$ or $\mu=\la$; the limit~\eqref{limit_b_diagonal} was also shown there. Hence, the proof is now finished.
\end{proof}

\smallskip

\begin{proof}[Proof of~\cref{theo:helper1}]
The theorem is an easy consequence of~\cref{lem:Polynomiality,lem:TopDegree_2,lem:TopDegree}; since the argument is
very similar to the proof of
\cite[Thm.~5.1]{Benaych-GeorgesCuencaGorin2022}, we omit the details. Note however that
\cite[Thm.~5.1]{Benaych-GeorgesCuencaGorin2022} has a polynomial $R_2$
with coefficients of order $O(\frac{1}{N})$, while in our theorem
the polynomial $R_2$ has coefficients of order $o(1)$.\footnote{It seems that \cite[Thm.~5.1]{Benaych-GeorgesCuencaGorin2022} has a mistake, namely, the difference between the two sides of \cite[Eqn.~(5.3)]{Benaych-GeorgesCuencaGorin2022} is not of order
$O(\frac{1}{N})$, but of order $o(1)$. However, the applications of \cite[Thm.~5.1]{Benaych-GeorgesCuencaGorin2022} only
use that the contribution of $R_2$ vanishes in the regime~\eqref{eq:HTRegime}, so their mistake does not affect their other results.}
\end{proof}

\subsection{Limiting moments via Łukasiewicz paths:~Proof of Theorem~\ref{theo:helper2}}\label{sec:hecke_proof}

By \cref{lem:Polynomiality}, with $k=\ell$, all quantities entering the proof of \cref{theo:helper2} at order $\ell$ depend only on the $\ell$-th order Taylor polynomial of $F_N$ at $(1^N)$.
Therefore, without loss of generality, throughout this subsection we assume that $F_N$ is given exactly by~\eqref{eq:truncated}, with $k=\ell$.
For the proof, it is convenient to introduce some notation.
Let $\mathcal{F}:=\{F_N=F_N(x_1,\dots,x_N)\}_{N\ge 1}$ be the resulting sequence of symmetric polynomials, with $F_N(1^N)=0$.
The Taylor coefficients $c^\nu_{F_N}$, $\nu\in\Y$, will be treated as variables of degrees $\deg(c^\nu_{F_N})=|\nu|$, as before.

For any $\ell\in\Z_{\ge 1}$, define $I_1(\mathcal{F}, \ell)$ as the set of polynomials $R_1=R_1\big( c^\nu_{F_N} \colon \nu\in\Y \big)$ in the variables $c^\nu_{F_N}$ with real coefficients, possibly depending on $N$ and $\theta$, and such that
\begin{itemize}
	\item $\deg(R_1)\le\ell$;
	\item $R_1$ has uniformly bounded coefficients in the regime~\eqref{eq:HTRegime};
	\item each monomial of $R_1$ has at least one factor $c^\nu_{F_N}$ with $\ell(\nu)>1$.
\end{itemize}

Similarly, define $I_2(\mathcal{F}, \ell)$ as the set of polynomials $R_2=R_2\big( c^\nu_{F_N} \colon \nu\in\Y \big)$ such that
\begin{itemize}
	\item $\deg(R_2)\le\ell$;
	\item $R_2$ has coefficients of order $o(1)$ in the regime~\eqref{eq:HTRegime}.
\end{itemize}

Finally we will need the following formal power series, for all $i=1,2,\dots,N$:
\begin{equation}\label{eq:definitionXi}
\Xi_i(z) := 1 + \sum_{\ell\ge 1}\xi_i^\ell\exp(F_N)\big|_{(x_1,\dots,x_N) = (1^N)}\cdot z^\ell.
\end{equation}

\begin{lemma}\label{lem:Moment}
Assume that we are in the setting of \cref{theo:helper2} and the functions $F_N$ are given by~\eqref{eq:truncated}.
In particular, $\ell\in\Z_{\ge 1}$ and $\mathcal{F}:=\{F_N=F_N(x_1,\dots,x_N)\}_{N\ge 1}$ is a sequence of smooth symmetric functions with Taylor expansions~\eqref{eq:truncated}, such that $F_N(1^N)=0$.
Also, let $i=1,2,\dots,N$. Then
\begin{equation}\label{eq:MainQuant2}
\xi_i^{\ell}\exp(F_{N})\big|_{(x_1,\dots,x_N) = (1^N)} =
[z^\ell]\Xi_1(z)\sum_{k \geq 0}\big( (1-i)\theta z\cdot\Xi_1(z) \big)^k + R_1 + R_2,
\end{equation}
for some $R_1\in I_1(\mathcal{F},\ell)$ and $R_2\in I_2(\mathcal{F},\ell)$.
Moreover, the bounds on the coefficients of $R_1,R_2$ are uniform over $i=1,2,\dots,N$.
\end{lemma}

\begin{proof}
The statement is obvious when $i=1$, so assume that $2\leq i\leq N$ and let $\ev$ denote the operator of evaluation of a function in variables $x_1,\dots,x_N$ at $(1^N)$. Since $\ev = \ev \circ s_{i-1}$, by using the Hecke relation \eqref{eq:Hecke} repeatedly, we get that the LHS of \eqref{eq:MainQuant2} is equal to 
\begin{multline*}
\ev\xi_{i}^{\ell}\exp(F_{N})
= \ev s_{i-1}\xi_{i}^{\ell}\exp(F_{N})
= \ev (\xi_{i-1}s_{i-1} -\theta)\xi_{i}^{\ell-1}\exp(F_{N})=\\
\ev \bigg(\xi_{i-1}(\xi_{i-1} s_{i-1} -\theta) \xi_{i}^{\ell-2} - \theta \xi_{i}^{\ell-1}\bigg)\exp(F_{N})
=\cdots
= \ev\xi_{i-1}^\ell\exp(F_{N}) - \theta\sum_{\substack{a,b\ge 0\\a+b=\ell-1}}\ev\xi_{i-1}^a \xi_{i}^b\exp(F_{N}).
\end{multline*}

Applying \cref{lem:TopDegree_2} to each summand $\ev\xi_{i-1}^a \xi_{i}^b\exp(F_{N})$ with $a,b\ge 1$, we have
\begin{equation}\label{eq:lemma_Rs}
\ev\xi_{i-1}^a \xi_{i}^b\exp(F_{N}) = \ev\xi_{i-1}^a\exp(F_{N}) \cdot \ev\xi_{i}^b\exp(F_{N}) - R_{i;a,b} - \frac{1}{N}R'_{i;a,b},
\end{equation}
for $R_{i;a,b}$, $R'_{i;a,b}$ as described in that lemma.
If $a=0$ or $b=0$ (or both), then~\eqref{eq:lemma_Rs} is also true if we set both $R_{i;a,b}$ and $R'_{i;a,b}$ to be the zero polynomials.
This leads to the power series identity
\begin{equation}\label{eqn:Xi_formula}
\Xi_i(z) = \left(\Xi_{i-1}(z) + \theta z R_i(z)\right)\left(1+\theta z\cdot\Xi_{i-1}(z)\right)^{-1},
\end{equation}
where
\begin{equation}\label{eq:R_i}
R_i(z) := \sum_{a,b\ge 1}\left(R_{i;a,b}+ \frac{1}{N}R'_{i;a,b}\right)z^{a+b}.
\end{equation}

Taking into account that~\eqref{eqn:Xi_formula} holds true for any $2 \leq i\leq N$, 
we obtain by iterating this equation that for any $1 \leq k\leq i-1$, one has
\begin{equation}
  \label{eq:ComRec}
  \Xi_i(z) = \left(\Xi_{k}(z)+\sum_{m = 1}^{i-k}(\theta
    z)^mf^{(k)}_{m}\right)\left(1+(i-k)\theta z\cdot\Xi_{k}(z) + \sum_{m =
      2}^{i-k}(\theta z)^mg^{(k)}_{m}\right)^{-1},
  \end{equation}
where
\begin{align*}
f^{(k)}_{2j} &= \Xi_k(z)\cdot\sum_{k+1\le i_1<\cdots<i_j \le i} (i_1-k-1) \prod_{s=1}^{j-1}(i_{s+1}-i_s-1)\cdot R_{i_1}(z)\cdots R_{i_j}(z),\\
f^{(k)}_{2j-1} &= \sum_{k+1 \leq i_1<\cdots < i_j \leq i} \,\prod_{s=1}^{j-1}(i_{s+1}-i_s-1) \cdot R_{i_1}(z)\cdots R_{i_j}(z),\\
g^{(k)}_{2j} &= \sum_{k+1 \leq i_1<\cdots < i_j \leq i}(i-i_j)\prod_{s=1}^{j-1}(i_{s+1}-i_s-1)\cdot R_{i_1}(z)\cdots R_{i_j}(z),\\
g^{(k)}_{2j+1} &= \Xi_k(z)\cdot\sum_{k+1\le i_1<\cdots<i_j\le i}(i-i_j)(i_1-k-1) 
\prod_{s=1}^{j-1}(i_{s+1}-i_s-1)\cdot R_{i_1}(z)\cdots R_{i_j}(z).
\end{align*}
By taking $k=1$ in~\eqref{eq:ComRec} we obtain that
\begin{equation}\label{eq:pomocnicze}
\Xi_i(z) = \left(\Xi_{1}(z)+z\tilde{R}(z)\right)\left(1+(i-1)\theta z \cdot\Xi_{1}(z) + z\tilde{R}'(z)\right)^{-1},
\end{equation}
where
\begin{equation}\label{eq:tildeR}
\tilde{R}(z) := \theta\sum_{m=1}^{i-1}(\theta z)^{m-1}f^{(1)}_m,\quad
\tilde{R}'(z) := \theta\sum_{m=2}^{i-1}(\theta z)^{m-1}g^{(1)}_m.
\end{equation}

\begin{claim}\label{claim:small_coeffs}
The constant coefficients $[z^0]\tilde{R}(z)$, $[z^0]\tilde{R}'(z)$ are equal to zero.
Moreover, for any $p\in\Z_{\ge 1}$, both $[z^p]\tilde{R}(z)$ and $[z^p]\tilde{R}'(z)$ are of the form $S_1+S_2$, where $S_1\in I_1(\mathcal{F},p)$ and $S_2\in I_2(\mathcal{F},p)$.
\end{claim}

Let us defer the proof of the claim until later.

Assuming \Cref{claim:small_coeffs}, \Cref{lem:Moment} is now a consequence of the following chain of equalities

\begin{align}
\xi_i^{\ell}\exp(F_N)\big|_{(x_1,\dots,x_N) = (1^N)} &= [z^\ell]\Xi_i(z)\nonumber\\
&= [z^\ell]\left(\Xi_1(z)+z\tilde{R}(z)\right)\sum_{k\ge 0}\left((1-i)\theta z \cdot\Xi_{1}(z) - z\tilde{R}'(z)\right)^k\label{chain_line2}\\
&= [z^\ell]\,\Xi_1(z)\sum_{k\ge 0}\big( (1-i) \theta z\cdot\Xi_1(z) \big)^k + R_1 + R_2,\label{chain_line3}
\end{align}
for some $R_1\in I_1(\mathcal{F},\ell)$ and $R_2\in I_2(\mathcal{F},\ell)$.
Indeed, the first equality above is due to the definition~\eqref{eq:definitionXi} of $\Xi_i(z)$, whereas the second one is due to Equation~\eqref{eq:pomocnicze}, so only the last equality remains to be justified.
For this, let the difference between the coefficient of $z^\ell$ in line~\eqref{chain_line2} and the coefficient of $z^\ell$ in line~\eqref{chain_line3} be denoted by $[z^\ell]A(z)$; the task is to verify that $[z^\ell]A(z)$ has the form $R_1+R_2$, for some $R_1\in I_1(\mathcal{F},\ell)$ and $R_2\in I_2(\mathcal{F},\ell)$.
Evidently, $A(z)$ is a finite sum of products, each of which possibly contains the factors $\Xi_1(z)$ and $\big((1-i)\theta\big)^j\cdot \big(z\cdot\Xi_1(z)\big)^j$, for some $1\le j\le\ell$, and it always contains a factor of the form $z\tilde{R}(z)$ or $\big(z\tilde{R}'(z)\big)^j$, for some $1\le j\le\ell$ (or both of them).
By \cref{lem:Polynomiality}, applied to $\Xi_1(z)=\xi_1^{\ell}\exp(F_N)\big|_{(x_1,\dots,x_N) = (1^N)}$, and \cref{claim:small_coeffs}, it follows that $[z^\ell]A(z)$ is a polynomial of the $c^\la_{F_N}$ of degree at most $\ell$.
Then we can gather those monomials that have some factor $c^\nu_{F_N}$ with $\ell(\nu)>1$ into $R_1$ and those that do not contain such factors into $R_2$.
The fact that the coefficients of $R_1$ and $R_2$ are of orders $O(1)$ and $o(1)$, respectively, in the regime~\eqref{eq:HTRegime} also follow from \cref{lem:Polynomiality} and \cref{claim:small_coeffs}; finally, the uniformity claim stated in the lemma follows because the factors $\big((1-i)\theta\big)^j$, for some $1\le j\le\ell$, that could appear are bounded by
\[
\big| ((1-i)\theta)^j \big| \le (N\theta)^j \sim \gamma^j\le\max\{1,\gamma^\ell\}
\]
and this is independent of $i=1,2,\dots,N$. It only remains to show \cref{claim:small_coeffs}.

\smallskip

\emph{Proof of \cref{claim:small_coeffs}.}
From~\eqref{eq:tildeR}, we have
\begin{equation*}
[z^p]\tilde{R}(z) = \sum_{m=1}^{p+1}\theta^m[z^{p+1-m}]f_m^{(1)},\qquad
[z^p]\tilde{R}'(z) = \sum_{m=2}^{p+1}\theta^m[z^{p+1-m}]g_m^{(1)}.
\end{equation*}
All terms inside the sums defining $f_m^{(1)}$, $m\ge 1$, and $g_m^{(1)}$, $m\ge 2$, contain at least some factor $R_i(z)$.
Since $[z^0]R_i(z)=0$, by~\eqref{eq:R_i}, it follows that $[z^0]\tilde{R}(z)=[z^0]\tilde{R}'(z)=0$.

Next, assume $p\ge 1$.
Since $\theta\sim\frac{\gamma}{N}$ in the regime~\eqref{eq:HTRegime}, it is enough to check that the coefficients $[z^{p+1-m}]f_m^{(1)}$, $[z^{p+1-m}]g_m^{(1)}$ are polynomials in the $c_{F_N}^\la$ of degree at most $p$ such that the coefficients of the monomials containing at least one $c_{F_N}^\nu$ with $\ell(\nu)>1$ are $O(N^m)$ in the regime~\eqref{eq:HTRegime} and all the other coefficients are $o(N^m)$ and, in fact, we check that these coefficients are $O(N^{m-1})$.
We will only verify these bounds for $f^{(1)}_m$ indexed by an even integer $m$, as the proofs of the remaining cases are fully analogous.

Let $m=2j$, for some $j\in\Z_{\ge 1}$. For any $k\in\Z_{\ge 0}$, the coefficient $[z^k]\Xi_1(z)$ is a polynomial in the $c_{F_N}^\la$ of degree at most $k$
with uniformly bounded coefficients in the regime~\eqref{eq:HTRegime}
by \cref{lem:Polynomiality}, and for any $1 \leq i_s < N$, the coefficient $[z^k]R_{i_s}(z)$ is a polynomial in the $c_{F_N}^\la$ with degree at
most $k$ such that the coefficients of the monomials containing at least one $c_{F_N}^\nu$ with $\ell(\nu)>1$ are $O(1)$ and all the other coefficients are $O(\frac{1}{N})$ in the regime~\eqref{eq:HTRegime}.
Therefore, for any sequence $2\le i_1<\cdots < i_j \le i$, we have that
$[z^{p+1-2j}]\big(\Xi_1(z)\cdot R_{i_1}(z)\cdots R_{i_j}(z)\big)$
is a polynomial in the $c_{F_N}^\la$ with degree at most~$p+1-2j<p$ and such that the coefficients of the monomials containing at least one $c_{F_N}^\nu$ with $\ell(\nu)>1$ are $O(1)$ and all the other coefficients are
$O(\frac{1}{N})$ in the regime~\eqref{eq:HTRegime}.
For each such sequence, we have that $(i_1-2) \prod_{s=1}^{j-1}(i_{s+1} - i_s - 1) = O(N^j)$ and the number of such sequences is also $O(N^j)$, therefore $[z^{p+1-2j}]f_{2j}^{(1)}$ is a polynomial in the $c_{F_N}^\la$ with degree at most $p$ and such that the coefficients of the monomials containing at least one $c_{F_N}^\nu$ with $\ell(\nu)>1$ are $O(N^{2j})$ and all the other coefficients are $O(N^{2j-1})$ in the regime~\eqref{eq:HTRegime}, as desired.

The previous argument ends the proof of \cref{claim:small_coeffs} and, henceforth, we are done.
\end{proof}

\begin{lemma}\label{lem:Xi1}
Assume that we are in the setting of \cref{theo:helper2} and the $F_N$ have Taylor expansions~\eqref{eq:truncated}. Then
\begin{multline}\label{eq:LimitXi1}
\xi_1^{\ell}\exp(F_{N})\big|_{(x_1,\dots,x_N)=(1^N)} = \sum_{\Gamma\in\Luk(\ell)}
\bigg(\prod_{i\ge 0} \big( c_{F_N}^{(1)}+i \big)^{\text{\# horizontal steps at height $i$ in $\Gamma$}}\\
\cdot \prod_{j\ge 1} \big( jc_{F_N}^{(j)}+(j+1)c_{F_N}^{(j+1)} \big)^{\text{\# steps $(1,j)$ in $\Gamma$}}
(j+\gamma)^{\text{\# down steps from height $j$ in $\Gamma$}}\bigg)+R_1'+R_2',
\end{multline}
for some $R_1'\in I_1(\mathcal{F},\ell)$ and $R_2'\in I_2(\mathcal{F}, \ell)$.
\end{lemma}

\begin{proof}
\cref{lem:Polynomiality} implies that the LHS of \eqref{eq:LimitXi1}
is a polynomial in $c_{F_N}^\la$ of degree at most $\ell$ with
uniformly bounded coefficients in the regime~\eqref{eq:HTRegime}. Therefore, similarly as in the proof of \cref{lem:TopDegree_2}, it suffices to prove~\eqref{eq:LimitXi1} with $F_N(\xx)$ replaced by
\[ 
H_N(\xx) = \sum_{i=1}^N{ f_N(x_i-1) },\quad \text{ where }\quad
f_N(x) := \sum_{k\geq 1}c_{F_N}^{(k)}x^k.
\]
With this reduction, all terms involving a coefficient $c^\nu_{F_N}$ with $\ell(\nu)>1$ can be collected into a remainder from $I_1(\mathcal{F},\ell)$.
Thus it remains to identify, up to a remainder from $I_2(\mathcal{F},\ell)$, the part of the left hand side that depends only on $c_{F_N}^{(1)},\dots,c_{F_N}^{(\ell)}$.
Since both sides are polynomials of degree at most $\ell$ in these variables with uniformly bounded coefficients in the regime~\eqref{eq:HTRegime}, it is enough to verify the following limiting identity for arbitrary real numbers $\kappa_1,\dots,\kappa_\ell$,
\begin{multline}\label{eq:LimitXi1Prim}
\lim_{\substack{N\to\infty\\N\theta\to\gamma}} \xi_{1}^{\ell}\exp(H_N)\big|_{(x_1,\dots,x_N) = (1^N)}
= \sum_{\Gamma\in\Luk(\ell)}\bigg(\prod_{i\ge 0}(\kappa_1+i)^{\text{\# horizontal steps at height $i$ in $\Gamma$}}\\
\cdot \prod_{j\ge 1} (\kappa_j+\kappa_{j+1})^{\text{\# steps $(1,j)$ in $\Gamma$}}
(j+\gamma)^{\text{\# down steps from height $j$ in $\Gamma$}}\bigg),
\end{multline}
provided that $\lim_{N\to\infty}{\big(n\cdot c_{F_N}^{(n)}\big)}=\kappa_n$, for $n=1,\dots,\ell$.
Indeed, because $\kappa_1,\dots,\kappa_\ell$ are arbitrary, this limiting identity implies coefficientwise convergence of the corresponding polynomials, so their difference belongs to $I_2(\mathcal{F},\ell)$.
Let us define the formal power series
\begin{equation}\label{eq:kappa}
\kappa(x) := \kappa_1 + \sum_{n \geq 1}(\kappa_n+\kappa_{n+1})x^n.
\end{equation}
We claim first that \eqref{eq:LimitXi1Prim} follows from the following formula, which will be proved later:
\begin{equation}\label{eq:LimitXi1'}
\lim_{\substack{N\to\infty\\N\theta\to\gamma}} \xi_{1}^{\ell}\exp(H_N)\big|_{(x_1,\dots,x_N) = (1^N)}
 = (x\partial_x+\partial_x+\gamma\!\dd+\kappa)^{\ell-1}\kappa(x)\big|_{x=0},
\end{equation}
where $\kappa$ is the operator of multiplication by $\kappa(x)$, while $\partial_x=\frac{d}{dx}$ is the derivative, and $\dd$ is the lowering operator:
\begin{equation*}
\dd\left(\sum_{i \geq 0}f_i x^i\right) := \sum_{i\geq 0}{f_{i+1}x^i}.
\end{equation*}
Indeed, note that the RHS of \eqref{eq:LimitXi1'} can be interpreted as the constant
term of a formal power series obtained as a sum of terms obtained by $\ell-1$ applications of
operators $x\partial_x, \partial_x, \gamma\!\dd, \kappa$ to
$\kappa(x)$. Let $w' = w_{\ell}\cdots w_{2}$ be a word of length $\ell-1$, where each $w_i$ is one of the
operators $\partial_x+\gamma\!\dd, \kappa_1+x\partial_x$, or $(\kappa_j+\kappa_{j+1})x^j$, for some $j\geq 1$.
Note that these operators are homogeneous of degrees $-1,0,j$,
respectively, therefore $w' x^n\big|_{x=0}$ vanishes unless

\smallskip

\begin{enumerate}[label=(\Alph*),itemsep=0.1cm]
	
	\item for each $2\leq j\leq\ell$, one has $n+\deg(w_2)+\cdots+\deg(w_j) \geq 0$,

	\item $\deg(w') := \sum_{j=2}^{\ell}\deg(w_j) = -n$.

\end{enumerate}

\noindent Define $w_1 := (\delta_{n>0}\kappa_n+\kappa_{n+1})x^n$ and $w:=w_\ell\cdots w_1$, so that $\deg(w)=0$.
For each $w$, associate a lattice path $\Gamma$ in $\Z^2$ starting at $(0,0)$ with the $i$-th step being $(1,\deg(w_i))$; then conditions (A)-(B) are equivalent to $\Gamma$ being a Łukasiewicz path. Note that the height of $\Gamma$ after the $i$-th step is equal to $j$ if and only if $\deg(w_i\cdots w_1) = j$. Together with the following observations:
\begin{equation*}
(\partial_x+\gamma\!\dd)x^j = (j+\gamma)x^{j-1},\quad j\ge 1,\qquad\quad
(\kappa_1+x\partial_x)x^i = (\kappa_1+i)x^i,\quad i\ge 0,
\end{equation*}
this gives that $w'\kappa(x)\big|_{x=0}=w(1)\big|_{x=0}$ is equal to
\begin{multline*}
\prod_{i\ge 0}{ (\kappa_1+i)^{\text{\# horizontal steps at height $i$ in $\Gamma$}} }\\
\times\prod_{j\ge 1}{ (\kappa_j+\kappa_{j+1})^{\text{\# steps $(1,j)$ in $\Gamma$}}(j+\gamma)^{\text{\# down steps from height $j$ in $\Gamma$}} }.
\end{multline*}
Together with~\eqref{eq:LimitXi1'}, this proves the desired \eqref{eq:LimitXi1Prim}, however it remains to prove~\eqref{eq:LimitXi1'}.

\medskip

\emph{Proof of Equation~\eqref{eq:LimitXi1'}.}
Recall that the Dunkl operators $\tilde{\xi}_i $ were defined in the proof of \cref{lem:TopDegree}, namely in equation~\eqref{eqn:dunkl}.
Moreover, we need the following special case of~\eqref{cherednik_dunkl}:
\begin{equation}\label{eq:ChangeOfVar}
\xi_{1}^{\ell}\exp(H_{N}(\xx)) \Big|_{(x_1,\dots,x_N)=(1^N)} 
= \Big( \xi_1 + \tilde{\xi}_1 \Big)^{\ell}\exp\big(H_N\big(\xx+1^N\big)\big) \Big|_{(x_1,\dots,x_N)=(0^N)}.
\end{equation}
Let us define
\begin{multline}\label{eqn:kappa_N}
\kappa_N(x) := (x_1\partial_1+\partial_1) H_N(\xx+1^N) \big|_{x_1\mapsto x} = (x+1)\cdot f_N'(x) \\
= c_{F_N}^{(1)}+\sum_{n\geq 1}\big( n\cdot c_{F_N}^{(n)} + (n+1)\cdot c_{F_N}^{(n+1)} \big) x^n,
\end{multline}
so that, by comparing with \eqref{eq:kappa}, we have $\lim_{N\to\infty}{\kappa_N(x)}=\kappa(x)$ coefficient-wise.
Since $H_N$ is symmetric in its variables, then all operators $1-s_{i,j}$ kill it.
It then follows from the definition~\eqref{eqn:kappa_N} of $\kappa_N(x)$ and the definitions~\eqref{eq:Cherednik}, \eqref{eqn:dunkl} of $\xi_1, \tilde{\xi}_1$, that
\begin{multline}\label{eq:Limit1}
\Big(\xi_1 + \tilde{\xi}_1\Big)^{\ell}\exp(H_{N}(\xx+1^N))
= \Big(\xi_1 + \tilde{\xi}_1\Big)^{\ell-1}(x_1\partial_1+\partial_1)\exp(H_{N}(\xx+1^N))\\
= \Big( \xi_1 + \tilde{\xi}_1\Big)^{\ell-1}\left[\kappa_N(x_1)\exp(H_{N}(\xx+1^N))\right].
\end{multline}
By~\eqref{eq:ChangeOfVar}, we are interested in the limit of~\eqref{eq:Limit1} in the regime~\eqref{eq:HTRegime}.
We note that $\big(\xi_{1}+\tilde{\xi}_{1}\big)^{\ell-1}$ is a sum of products
of $\ell-1$ operators, each of the form $(x_1+1)\partial_1$, $\frac{\theta}{x_1-x_j}(1-s_{1,j})$, or
$\frac{\theta x_j}{x_1-x_j}(1-s_{1,j})$, for some $1<j\le N$.
A generic term involves products of $\ell-1$ operators, where all the indices $j$ are distinct.
However, \eqref{eq:DerRules} implies that all the terms that involve some $\frac{\theta x_j}{x_1-x_j}(1-s_{1,j})$
will not contribute to the limit~\eqref{eq:Limit1}, because the factor $x_j$ in it 
will survive after the action of the operators $(x_1+1)\partial_1$, $\frac{\theta}{x_1-x_{j'}}(1-s_{1,j'})$, and
$\frac{\theta x_{j'}}{x_1-x_{j'}}(1-s_{1,j'})$, where $j'\neq j$, and
therefore the whole contribution will vanish after the substitution
$\xx = (0^N)$. Finally, \eqref{eq:DerRules} implies that
$\frac{\theta}{x_1-x_{j}}(1-s_{1,j}) = \theta\!\dd_1 + O(x_j)$, where
\begin{equation*}
\dd_1(x_1^{\alpha_1}\cdots x_N^{\alpha_N}) := \mathbf{1}_{\alpha_1>0}\cdot x_1^{\alpha_1-1}\cdots x_N^{\alpha_N},
\end{equation*}
therefore
\begin{multline}\label{eq:Limit2}
\lim_{\substack{N\to\infty\\N\theta\to\gamma}}\left(\xi_{1}+\tilde{\xi}_{1}\right)^{\ell-1}
\left[\kappa_N(x_1)\exp(H_{N}(\xx+1^N))\right]\Big|_{(x_1,\dots,x_N) = (0^N)}\\
=\lim_{\substack{N\to\infty\\N\theta\to\gamma}}\big( (x_1+1)\partial_1+(N-1)\theta\!\dd_1+\kappa_N(x_1) \big)^{\ell-1}
\left[\kappa_N(x_1)\right]\Big|_{(x_1,\dots,x_N) = (0^N)}\\
= \big( (x+1)\partial_x + \gamma\!\dd + \kappa \big)^{\ell-1}\left[\kappa(x)\right]\big|_{x=0}.
  \end{multline}
The desired \eqref{eq:LimitXi1'} follows by combining \eqref{eq:ChangeOfVar}, \eqref{eq:Limit1} and \eqref{eq:Limit2}.
The proof is now complete.
\end{proof}

\begin{proof}[Proof of~\cref{theo:helper2}]
Equation~\eqref{eq:MainQuant2} implies that
\begin{equation}\label{eq:ellthMoment}
\frac{1}{N}\cdot\mathcal{P}_\ell e^{F_N}\big|_{(x_1,\dots,x_N)=(1^N)}
= [z^\ell] \frac{1}{N}\Xi_1(z) \sum_{i=1}^N\sum_{k \geq 0}\left( (1-i)\theta z\cdot\Xi_1(z)\right)^k + R_1'' + R_2'',
\end{equation}
where $R_1''\in I_1(\mathcal{F},\ell)$ and $R_2''\in I_2(\mathcal{F},\ell)$.
Next, define the formal power series
\begin{multline}\label{eq:defX}
X(z) := 1+\sum_{\ell \geq 1}\sum_{\Gamma\in\Luk(\ell)}\bigg(\prod_{i\ge 0}
\big( c_{F_N}^{(1)}+i \big)^{\text{\# horizontal steps at height $i$ in $\Gamma$}}\\
\cdot \prod_{j\ge 1} \big( jc_{F_N}^{(j)}+(j+1)c_{F_N}^{(j+1)} \big)^{\text{\# steps $(1,j)$ in $\Gamma$}}
(j+\gamma)^{\text{\# down steps from height $j$ in $\Gamma$}}\bigg)\cdot z^\ell,
\end{multline}
so that \eqref{eq:LimitXi1} can be rewritten as
\begin{equation}\label{eq:LimitXi1'''}
[z^\ell]\Xi_1(z) = [z^\ell]X(z)+R_1'+R_2',
\end{equation}
where $R_1'\in I_1(\mathcal{F},\ell)$ and $R_2'\in I_2(\mathcal{F},\ell)$.
Moreover, we have
\begin{equation}\label{eqn:sum_powers}
\frac{1}{N}\sum_{i=1}^N\big((1-i)\theta\big)^k = \left(-N\theta\right)^{k}\frac{1}{N}\sum_{i=1}^N\left(\frac{i-1}{N}\right)^{k}= (-\gamma)^k \int_0^1x^{k}dx +o(1) = \frac{(-\gamma)^{k}}{1+k}+o(1)
\end{equation}
in the regime \eqref{eq:HTRegime}. Then, by substituting \eqref{eq:LimitXi1'''} and \eqref{eqn:sum_powers} into \eqref{eq:ellthMoment}, we obtain
\begin{equation}\label{eq:before_claim}
\frac{1}{N}\cdot\mathcal{P}_\ell e^{F_N}\big|_{(x_1,\dots,x_N)=(1^N)} 
= [z^\ell] \left(X(z)\sum_{k \geq 0}\frac{\left( -\gamma z\cdot X(z)\right)^k}{1+k}\right) + S_1 + S_2,
\end{equation}
where $S_1\in I_1(\mathcal{F},\ell)$ and $S_2\in I_2(\mathcal{F},\ell)$.
We claim that
\begin{equation}\label{eq:desired_moment_lth}
  [z^\ell] \left(X(z)\sum_{k \geq 0}\frac{\left( -\gamma z\cdot X(z)\right)^k}{1+k}\right) =m_\ell\left( c^{(1)}_{F_N},\cdots,\ell c^{(\ell)}_{F_N} \right).
\end{equation}
Evidently, plugging \eqref{eq:desired_moment_lth} into \eqref{eq:before_claim} would finish the proof, so it only remains to prove the claimed equality~\eqref{eq:desired_moment_lth} and we do this next.

For any $\ell\in\Z_{\ge 1}$ and $\Gamma\in\Luk(\ell)$, consider the $x$-dependent weight of $\Gamma$:
\begin{multline}\label{eq:gamma_prime_weight}
\wt_{x}(\Gamma) := \big( c_{F_N}^{(1)}-x \big)^{\text{\# horizontal steps at height $0$ in $\Gamma$}}\cdot\prod_{i\ge 1}\big( c_{F_N}^{(1)}+i \big)^{\text{\#\,horizontal steps at height $i$ in $\Gamma$}}\\
\cdot \prod_{j\ge 1} \big( jc_{F_N}^{(j)}+(j+1)c_{F_N}^{(j+1)} \big)^{\text{\# steps $(1,j)$ in $\Gamma$}}
(j+\gamma)^{\text{\# down steps from height $j$ in $\Gamma$}}
\end{multline}
and let $\wt_0(\Gamma)$ be the same weight~\eqref{eq:gamma_prime_weight}, but with $x=0$.
Also, if $\Gamma$ is the only Łukasiewicz path of length $0$, we set $\wt_x(\Gamma)=\wt_0(\Gamma)=1$, by convention.
For any Łukasiewicz path $\Gamma$, we denote by $\ell(\Gamma)$ the length of the Łukasiewicz path $\Gamma$.
Then, by the definition~\eqref{eq:defX}, we have
\begin{equation}\label{eq:alternative_X}
X(z) = \sum_{\ell\ge 0}\sum_{\Gamma\in\Luk(\ell)}{ \wt_0(\Gamma)z^{\ell(\Gamma)} }.
\end{equation}

On the other hand, if we let $s\in\Z_{\ge 0}$ be the number of horizontal steps at height $0$ of $\Gamma$, and let $1\le i_1 < \dots < i_s\le\ell$ be the integers such that $(i_j-1,0)\to (i_j,0)$ are the horizontal steps of $\Gamma$, then by expanding the first factor $( c_{F_N}^{(1)}-x)^s$ in~\eqref{eq:gamma_prime_weight}, we deduce
\begin{multline}\label{eq:weights_comb_1}
\wt_{x}(\Gamma) z^{\ell(\Gamma)} = \sum_{A=\{j_1<\cdots<j_m\}\subseteq\{1,\dots,s\}}
\wt_0\big(\Gamma_{0,\, i_{j_1}-1}\big)z^{\ell\left(\Gamma_{0,\, i_{j_1}-1}\right)}\\
\cdot\prod_{a=1}^m{\left(-x z\cdot\wt_0\big(\Gamma_{i_{j_a},\, i_{j_{a+1}}-1}\big)z^{\ell\left(\Gamma_{i_{j_a},\, i_{j_{a+1}}-1}\right)}\right)},
\end{multline}
where for any subset $A = \{j_1<\cdots<j_m\}\subset\{1,\dots,s\}$ with cardinality $|A|=m$, we denoted $\Gamma_{i_{j_a},\, i_{j_{a+1}}-1}$ the Łukasiewicz path of length $i_{j_{a+1}}-i_{j_a}-1$, which is the subpath of $\Gamma$ starting at $(i_{j_a},0)$ and finishing at $(i_{j_{a+1}}-1,0)$, and we set $i_{j_{m+1}} := \ell+1$, so that $\Gamma_{i_{j_m},\, i_{j_{m+1}}-1}$ is the subpath of $\Gamma$ starting at $(i_{j_m},0)$ and finishing at $(\ell,0)=(\ell(\Gamma),0)$.
From~\eqref{eq:weights_comb_1} and \eqref{eq:alternative_X}, one readily deduces the following identity of formal power series
\begin{equation}\label{eq:complicated_series}
\sum_{\ell\ge 0}\sum_{\Gamma\in\Luk(\ell)}\wt_{x}(\Gamma) z^{\ell(\Gamma)}
= X(z)\sum_{m\ge 0}\big( -x z\cdot X(z)\big)^m.
\end{equation}
On the other hand, note that
\begin{equation}\label{eq:helper_integral}
\Delta_\gamma\left( \frac{x^{1+k}}{1+k}\right)\big(c^{(1)}_{F_N}\big) = \frac{1}{\gamma}\int_{0}^\gamma \left( c^{(1)}_{F_N}-x\right)^k dx.
\end{equation}
Finally, we obtain
\begin{equation*}
m_\ell\left( c^{(1)}_{F_N},\cdots,\ell c^{(\ell)}_{F_N} \right) = [z^\ell] \sum_{\ell\ge 0}\sum_{\Gamma\in\Luk(\ell)}z^{\ell(\Gamma)}\frac{1}{\gamma}\int_{0}^\gamma\wt_x(\Gamma)dx  = [z^\ell] \left(X(z)\sum_{m\ge 0}\frac{\big( -\gamma z\cdot X(z)\big)^m}{1+m}\right),
\end{equation*}
where the first equality follows from \cref{eq:Moments} of \cref{def:mk}, the definition~\eqref{eq:gamma_prime_weight} of $\wt_x(\Gamma)$ and \cref{eq:helper_integral}, while the second equality follows from \cref{eq:complicated_series}.
This concludes the proof of the desired~\eqref{eq:desired_moment_lth} and hence finishes the proof of \cref{theo:helper2}.
\end{proof}

\begin{remark}
Note that the horizontal steps at height $0$ of Łukasiewicz paths do not play a distinguished role in the right hand side of~\eqref{eq:LimitXi1} from \cref{lem:Xi1}, but they do in the right hand side of~\eqref{eq:Moments} from \cref{def:mk}.
We point out that it is equation~\eqref{eq:desired_moment_lth} in the previous proof that explains the discrepancy; indeed, the left side of this equation does not distinguish steps at height $0$, whereas the right side does.
And ultimately, the proof of~\eqref{eq:desired_moment_lth} boils down to the equation~\eqref{eq:weights_comb_1} that explains that the horizontal steps at height $0$ of a Łukasiewicz path naturally divide it into smaller ones, and this is special about the height $0$.
\end{remark}

\subsection{Carleman's condition:~Proof of Theorem~\ref{theo:helper3}}

As $\vec{m}=(m_n)_{n\ge 1}$ comes from the definition of LLN-satisfaction and we are assuming that the $\PP_N$ are probability measures in this case, the $m_\ell$'s are the moments of some probability measure. In order to prove the uniqueness, we verify Carleman's condition, i.e. we prove that $\sum_{\ell\geq 1}m_{2\ell}^{-\frac{1}{2\ell}}$ diverges.
We will use the condition $\vec{m}=\mathcal{J}_\gamma^{\kappa\mapsto m}(\vec{\kappa})$ that expresses each moment as a $\kappa$-weighted sum of Łukasiewicz paths, in order to upper bound each $m_{2\ell}$.

\begin{proof}[Proof of~\cref{theo:helper3}]
We use Equation~\eqref{eq:Moments} that expresses each $m_{2\ell}$ as a weighted sum over Łukasiewicz paths of length $2\ell$.
Take any $\Gamma\in\Luk(2\ell)$ and let us upper bound the corresponding term in the sum in the RHS of~\eqref{eq:Moments}.

First, from \cref{eqn:A_action} for the divided difference operator $\Delta_\gamma$, we deduce that for any $\gamma>0$, $n\in\Z_{\ge 0}$, we have
\begin{equation*}
\bigg| \frac{\Delta_\gamma\left(x^{1+\text{\#\,horizontal steps at height $0$}}\right)(\kappa_1)}{1+\text{\#\,horizontal steps at height $0$}} \bigg| \le C_0^{\,\text{\# horizontal steps at height $0$}},
\end{equation*}
for some sufficiently large $C_0>0$ that depends on $\kappa_1$ and $\gamma$, but not on $\ell$ or the specific $\Gamma$.

Second, note that the height of any $\Gamma\in\Luk(2\ell)$ never exceeds $2\ell-1$, therefore
\begin{multline*}
\bigg| \prod_{i\ge 1}(\kappa_1+i)^{\text{\# horizontal steps at height $i$\,}} \bigg|
= \prod_{i=1}^{2\ell}|\kappa_1+i|^{\text{\,\# horizontal steps at height $i$}}\\
\le\prod_{i=1}^{2\ell}(|\kappa_1|+2\ell)^{\text{\# horizontal steps at height $i$}} 
\le (2\ell C_1)^{\text{\# horizontal steps at heights $\ge 1$}},
\end{multline*}
for some constant $C_1>0$. Similarly,
\begin{equation*}
\bigg| \prod_{j \geq 1}(j+\gamma)^{\text{\# down steps from height $j$}} \bigg|
= \prod_{j=1}^{2\ell}|j+\gamma|^{\text{\,\# down steps from height $j$}}
\le (2\ell C_2)^{\text{\# down steps}},
\end{equation*}
for some constant $C_2>0$. Finally, by using the condition of the theorem, we obtain
\begin{multline*}
\bigg| \prod_{j \geq 1}(\kappa_j+\kappa_{j+1})^{\text{\# steps $(1,j)$}} \bigg|
\leq \prod_{j\ge 1}(|\kappa_j|+|\kappa_{j+1}|)^{\text{\# steps $(1,j)$}}
\leq \prod_{j\ge 1}(C^j+C^{j+1})^{\text{\# steps $(1,j)$}}\\
\leq C_3^{\,\sum_{j \geq 1}j\,\cdot\,\text{\# steps $(1,j)$}}
= C_3^{\,\text{\# down steps}},
\end{multline*}
for some constant $C_3>0$. To explain the last equality, observe that $(1,j)$ is an up step of size $j$, so the sum $\sum_{j \geq 1}j\cdot\text{\# steps$(1,j)$}$ is the total (upward) vertical displacement caused by the up steps in the Łukasiewicz path.
Since the path starts and ends at the same horizontal line $y=0$, then the total (downward) vertical displacement caused by down steps should equal this sum and, simultaneously, equal the number of down steps, since each down step goes down exactly by $1$; this justifies the last equality.

From the previous inequalities, the absolute value of the weight of $\Gamma$ in~\eqref{eq:Moments} is upper bounded by $(2\ell A)^{\text{\# horizontal steps}\,+\,\text{\# down steps}} \leq (2\ell A)^{2\ell}$, 
where $A := \max\{C_0,C_1,C_2C_3\}>0$, for all large $\ell$.
It is known that the number of Łukasiewicz paths of length $\ell$ is the $\ell$-th Catalan number:
\begin{equation*}
|\Luk(\ell)| = \frac{1}{\ell+1}\binom{2\ell}{\ell}
\sim\frac{4^\ell}{\ell^{3/2}\sqrt{\pi}},\quad \text{ as }\ell \to \infty,
\end{equation*}
therefore
\begin{equation*}
m_{2\ell}\le |\Luk(2\ell)|\cdot (2\ell A)^{2\ell}\le (8\ell A)^{2\ell},
\end{equation*}
for all large $\ell$.
Evidently, $m_{2\ell}\ge 0$ because $m_{2\ell}$ is an even moment, consequently, $m_{2\ell}^{-\frac{1}{2\ell}} \ge \frac{1}{8\ell A}$, for all large $\ell$.
Since $\sum_{\ell\ge 1}{\frac{1}{8\ell A}}$ diverges, so does $\sum_{\ell\geq 1}m_{2\ell}^{-\frac{1}{2\ell}}$.
Thus, Carleman's condition is satisfied and the proof is finished.
\end{proof}

\section{Applications}\label{sec:applications}

In our applications, we will consider probability measures on partitions instead of general signatures. This is partially motivated by the fact that one of our main examples, the Jack measures from \cite{BorodinOlshanski2005}, are defined in this generality.
In all our examples, we will see that the condition of HT-appropriateness from \cref{def:appropriate} is easily verified.

\subsection{Application 1: Jack Littlewood--Richardson coefficients and quantized $\gamma$-convolution}\label{sec:app_1}

\subsubsection{Jack-Littlewood--Richardson convolution}\label{sec:quantized_convolution}

Given partitions $\la,\mu,\nu\in\Y(N)$, we will consider the normalization of Jack Littlewood--Richardson coefficients, to be denoted by $c^\la_{\mu,\nu}(\theta)$, uniquely determined by the equality
\begin{equation}\label{eq:multiplication_jacks}
\frac{P_{\mu}(x_1,\dots,x_N;\theta)}{P_{\mu}(1^N;\theta)}
\frac{P_{\nu}(x_1,\dots,x_N;\theta)}{P_{\nu}(1^N;\theta)}
= \sum_{\la\in\Y(N)}{ c^\la_{\mu,\nu}(\theta)\,\frac{P_\la(x_1,\dots,x_N;\theta)}{P_\la(1^N;\theta)} }.
\end{equation}
From the definition, it is clear that $c^\la_{\mu,\nu}(\theta)\ne 0$ implies $|\la|=|\mu|+|\nu|$.
Also, for any fixed $\mu,\nu\in\Y(N)$, we have
\begin{equation}\label{eqn:sum_1}
\sum_{\la\in\Y(N)}{ c^\la_{\mu,\nu}(\theta) } = 1,
\end{equation}
as deduced by setting $x_1=\cdots=x_N=1$ in~\eqref{eq:multiplication_jacks}.

Let $\Y(N)^{\leq L}$ denote the finite subset of partitions in $\Y(N)$ with longest part at most $L$. If $\mu,\nu\in\Y(N)^{\leq L}$, then $c^\la_{\mu,\nu}(\theta)\ne 0$ implies that $\la \in \Y(N)^{\leq 2L}$.
This follows because the largest monomials (in the lexicographic order) in $P_\mu(x_1,\dots,x_N;\theta)$ and $P_\nu(x_1,\dots,x_N;\theta)$ are multiples of $x_1^{\mu_1}\cdots x_N^{\mu_N}$ and $x_1^{\nu_1}\cdots x_N^{\nu_N}$, respectively.
Then the largest monomial in the LHS of~\eqref{eq:multiplication_jacks} is a multiple of $x_1^{\mu_1+\nu_1}\cdots x_N^{\mu_N+\nu_N}$.
Since $\mu_1+\nu_1\le L+L=2L$, by assumption, it follows that any $\la\in\Y(N)$ with nonzero $c^\la_{\mu,\,\nu}(\theta)$ must have $\la_1\le 2L$.

\begin{definition}\label{def:signed_measure}
For any $N\in\Z_{\ge 1}$ and $\Prob_N^{(1)},\Prob_N^{(2)}$ probability measures on $\Y(N)$, define the signed measure $\Prob_N^{(1)}\boxplus_\theta\Prob_N^{(2)}$ on $\Y(N)$ by
\begin{equation*}
\big(\Prob_N^{(1)}\boxplus_\theta\Prob_N^{(2)}\big)(\la) = \sum_{\mu,\nu\in\Y(N)}c^\la_{\mu,\,\nu}(\theta)\Prob_N^{(1)}(\mu)\Prob_N^{(2)}(\nu),\quad\la\in\Y(N).
\end{equation*}
The measure $\Prob_N^{(1)}\boxplus_\theta\Prob_N^{(2)}$ can be called the \textbf{Jack-LR convolution} of probability measures.
\end{definition}

Due to~\eqref{eqn:sum_1}, the signed measure $\Prob_N^{(1)}\boxplus_\theta\Prob_N^{(2)}$ has total mass equal to $1$.
Moreover, due to the previous discussion, if $\PP_N^{(1)}$, $\PP_N^{(2)}$ are supported on $\Y(N)^{\le L}$, then the support of $\Prob_N^{(1)}\boxplus_\theta\Prob_N^{(2)}$ is contained in $\Y(N)^{\leq 2L}$, so that $\Prob_N^{(1)}\boxplus_\theta\Prob_N^{(2)}$ is also finitely supported.
For $\theta\in\{1/2,1,2\}$, the Jack-LR convolution was introduced and studied in \cite{Huang2021}.

\begin{remark}\label{rem:stanley}
It is a conjecture of Stanley~\cite{Stanley1989} that $c^\la_{\mu,\,\nu}(\theta)\ge 0$, for all partitions $\la,\mu,\nu\in\Y(N)$ and any $N\in\Z_{\ge 1}$, whenever $\theta>0$.\footnote{In fact, Stanley's conjecture is stronger and claims that the structure constants that appear in the $J$-normalization of Jack polynomials are polynomials in $\theta^{-1}$ with nonnegative integer coefficients. The $J$-normalization, which is not used here, was proven to be of combinatorial significance in the monomial expansion~\cite{KnopSahi1997} and the power-sum expansion~\cite{BenDaliDolega2023}.}
If the conjecture were true, then the Jack-LR convolution of any two probability measures would always be a probability measure.
\end{remark}

\begin{remark}\label{rem:rep_theory_interpretation}
When $\theta=\frac{1}{2},1,2$, it is known that the conjecture of Stanley, mentioned in \cref{rem:stanley}, is true.
This follows from a representation-theoretic meaning of the Jack Littlewood--Richardson coefficients: up to normalization, $c^\la_{\mu,\,\nu}(\theta)$ is equal to the multiplicity of the spherical function with highest weight $\la$ in the decomposition of the product of the spherical functions of highest weights $\mu,\nu$ for the Gelfand pairs $(\GL_N(\R),\Or_N(\R))$, $(\GL_N(\C),\U_N(\C))$, and $(\GL_N(\HH),\U_N(\HH))$, respectively. Here, $\HH$ denotes the skew-field of quaternions.
\end{remark}

\subsubsection{High temperature LLN for Jack-LR convolution and quantized $\gamma$-convolution}

Fix $L \in \Z_{\geq 1}$, let $\big\{\Prob_N^{(1)}\big\}_{N\ge 1}$ and $\big\{\Prob_N^{(2)}\big\}_{N\ge 1}$ be sequences of probability measures on $\Y(N)^{\leq L}$, and let $\mu^{(1)}_N, \mu^{(2)}_N$ be the associated empirical measures. Recall the terminology used in \cref{sec:main_results_intro}.
For a probability measure $\mu$ on $\R$ with finite moments $m_1,m_2,\dots$ of all orders (e.g.~if $\mu$ is compactly supported), its \emph{quantized $\gamma$-cumulants} $\kappa_1,\kappa_2,\dots$ are defined from the relation $(\kappa_1,\kappa_2,\dots) = \mathcal{J}_\gamma^{m\mapsto\kappa}\big( m_1,m_2,\dots \big)$.

\begin{theorem}\label{app:sum}
Let $\{\theta_N>0\}_{N\ge 1}$ be such that $N\theta_N\to\gamma$, as $N\to\infty$. Assume that we are in the setting above and that we have the weak limits
\begin{equation}\label{assumption:weak_lims}
\mu^{(1)}_N\to\mu^{(1)}, \qquad \mu^{(2)}_N\to\mu^{(2)}, \qquad \textrm{as $N\to\infty$},
\end{equation}
for certain probability measures $\mu^{(1)}$, $\mu^{(2)}$ on $\R$.
Let $\kappa_n\big[\mu^{(1)}\big]$, $\kappa_n\big[\mu^{(2)}\big]$ be the corresponding quantized $\gamma$-cumulants.
Then the sequence of signed measures $\big\{\Prob_N^{(1)}\boxplus_{\theta_N}\Prob_N^{(2)}\big\}_{N\ge 1}$ satisfies the LLN, in the sense of \cref{def:LLN}, with corresponding quantized $\gamma$-cumulants
\begin{equation}\label{kappas_addition}
\widetilde{\kappa}_n
:= \kappa_n\big[ \mu^{(1)} \big] + \kappa_n\big[ \mu^{(2)} \big],
\quad\textrm{for all }n\in\Z_{\ge 1}.
\end{equation}
\end{theorem}

If the conjecture from \cref{rem:stanley} were true, then $\Prob_N^{(1)}\boxplus_{\theta_N}\Prob_N^{(2)}$ would be probability measures and we would be able to define the random atomic measures
\begin{equation*}
\mu_N := \frac{1}{N}\sum_{i=1}^N{ \delta_{\la_i - \theta_N(i-1)} },\text{ where } (\la_1\ge\cdots\ge\la_N)\text{ is } \Prob_N^{(1)}\boxplus_{\theta_N}\Prob_N^{(2)}\text{--distributed}.
\end{equation*}
Each $\mu_N$ is supported on $[-\theta_N(N-1),\, 2L]$, and \cref{app:sum} would additionally imply the weak limit, in probability,
\begin{equation*}
\mu_{N} \to \mu^{(1)}\boxplus^{(\gamma)}\mu^{(2)},
\qquad \textrm{as }N\to\infty,
\end{equation*}
where $\mu^{(1)}\boxplus^{(\gamma)} \mu^{(2)}$ is the unique probability measure with quantized $\gamma$-cumulants $\widetilde{\kappa}_n$ from~\eqref{kappas_addition}; it is supported on $[-\gamma,2L]$.

More generally, if for given probability measures $\mu^{(1)}$, $\mu^{(2)}$ (compactly supported or not) there exists a probability measure with quantized $\gamma$-cumulants
\[
\kappa_n\big[\mu^{(1)}\big]+\kappa_n\big[\mu^{(2)}\big],
\qquad n\ge1,
\]
which is uniquely determined by them, then it will be called the \emph{quantized $\gamma$-convolution} of $\mu^{(1)}$ and $\mu^{(2)}$, and will be denoted by $\mu^{(1)}\boxplus^{(\gamma)}\mu^{(2)}$.

\begin{proof}[Proof of Theorem~\ref{app:sum}]
Since $N\theta_N \to \gamma$, as $N \to \infty$, there exists a constant $C>0$ such that $\sup_{N\ge 1}(\theta_N(N-1))\le C$, and all measures $\mu^{(1)}_N$, $\mu^{(2)}_N$ are supported on the same compact interval $[-C,L]$. Thus, the weak limits~\eqref{assumption:weak_lims} imply that the sequences $\{\Prob_N^{(1)}\}_{N\ge 1}$ and $\{\Prob_N^{(2)}\}_{N\ge 1}$ satisfy the LLN.
Therefore \cref{theo:main1} implies HT-appropriateness, meaning that the logarithms
\begin{equation*}
F_{\Prob_N^{(1)},\theta_N} := \ln\left( G_{\Prob_N^{(1)},\theta_N} \right),\quad
F_{\Prob_N^{(2)},\theta_N} := \ln\left( G_{\Prob_N^{(2)},\theta_N} \right),
\end{equation*}
have Taylor coefficients with certain limits (as in \cref{def:appropriate}), in particular,
\begin{equation}\label{eq:ht_1}
\lim_{N\to\infty} \frac{\partial_1^n F_{\Prob_N^{(1)},\theta_N}}{(n-1)!} \Bigg|_{(x_1,\dots,x_N) = (1^N)} = \kappa_n\big[\mu^{(1)}\big],\qquad
\lim_{N\to\infty} \frac{\partial_1^n F_{\Prob_N^{(2)},\theta_N}}{(n-1)!} \Bigg|_{(x_1,\dots,x_N) = (1^N)} = \kappa_n\big[\mu^{(2)}\big],
\end{equation}
for all $n\in\Z_{\ge 1}$.

Next, since the support of $\Prob_N^{(1)}\boxplus_{\theta_N}\Prob_N^{(2)}$ is the finite set $\Y(N)^{\le 2L}$, this measure is finitely supported and, therefore, it has the small tails property.
By \cref{def:signed_measure}, the Jack generating function of $\Prob_N^{(1)}\boxplus_{\theta_N}\Prob_N^{(2)}$, to be denoted by $G_{N,\theta_N}(x_1,\dots,x_N)$, is equal to
\begin{align*}
G_{N,\theta_N}(x_1,\dots,x_N) &= \sum_{\la\in\Y(N)}{ \big(\Prob_N^{(1)}\boxplus_{\theta_N}\Prob_N^{(2)}\big)(\la)\,\frac{P_\la(x_1,\dots,x_N;\theta_N)}{P_\la(1^N;\theta_N)} }\\
&=\sum_{\mu,\nu\in\Y(N)}\Prob_N^{(1)}(\mu)\Prob_N^{(2)}(\nu)\sum_{\la\in\Y(N)}c^\la_{\mu,\,\nu}(\theta_N)\,\frac{P_\la(x_1,\dots,x_N;\theta_N)}{P_\la(1^N;\theta_N)}\\
&= G_{\Prob_N^{(1)},\theta_N}(x_1,\dots,x_N)\cdot
G_{\Prob_N^{(2)},\theta_N}(x_1,\dots,x_N),
\end{align*}
and therefore $F_{N,\theta_N}(x_1,\dots,x_N) := \ln\left(G_{N,\theta_N}(x_1,\dots,x_N)\right)$ equals
\begin{equation}\label{eq:additive}
F_{N,\theta_N}(x_1,\dots,x_N) = F_{\Prob_N^{(1)},\theta_N}(x_1,\dots,x_N) + F_{\Prob_N^{(2)},\theta_N}(x_1,\dots,x_N).
\end{equation}
Since $\big\{F_{\Prob_N^{(1)},\theta_N}\big\}_{N\ge 1}$ and $\big\{F_{\Prob_N^{(2)},\theta_N}\big\}_{N\ge 1}$ satisfy the conditions of HT-appropriateness from \cref{def:appropriate}, it readily follows that the same conditions are true for $\big\{F_{N,\theta_N}\big\}_{N\ge 1}$, in particular, thanks to~\eqref{eq:ht_1} and~\eqref{eq:additive}, we have
\begin{equation*}
\lim_{N\to\infty} \frac{\partial_1^n F_{N,\theta_N}}{(n-1)!} \Bigg|_{(x_1,\dots,x_N) = (1^N)}
= \kappa_n\big[ \mu^{(1)} \big] + \kappa_n\big[ \mu^{(2)} \big]
= \widetilde{\kappa}_n,\quad\textrm{for all $n\in\Z_{\ge 1}$},
\end{equation*}
proving the desired~\eqref{kappas_addition}.
Finally, the fact that $\big\{\Prob_N^{(1)}\boxplus_{\theta_N}\Prob_N^{(2)}\big\}_{N\ge 1}$ satisfies the LLN follows from \cref{theo:main1}.
\end{proof}

\subsection{Application 2: Jack measures and nonintersecting random walks}\label{sec:jack_measures}

\subsubsection{Jack measures}

Let us discuss the general setting first and then specialize to three special (``pure'') families of examples.

\begin{definition}[\cite{BorodinOlshanski2005}]\label{def:general_jack_measure}
For any Jack-positive specializations $\rho_1,\rho_2\colon\Sym\to\R$ satisfying the finiteness assumption
\begin{equation}\label{eq:finiteness}
\sum_{\la\in\Y}{Q_\la(\rho_1;\theta)P_\la(\rho_2;\theta)} < \infty,
\end{equation}
the \textbf{Jack measure} $\,\PP^{(\theta)}_{\rho_1;\rho_2}(\cdot)$ is the probability measure on the set of all partitions, defined by
\begin{equation}
  \label{def:jack_measure}
\PP^{(\theta)}_{\rho_1;\rho_2}(\la) := \frac{1}{H_\theta(\rho_1;\rho_2)}\cdot Q_\la(\rho_1;\theta)P_\la(\rho_2;\theta),\quad\la\in\Y,
\end{equation}
where $H_\theta(\rho_1;\rho_2)$ and $Q_\la(;\theta)$ are defined by~\eqref{eqn:cauchy} and~\eqref{j_lambda}, respectively.
\end{definition}

We remark that our definition of Jack measures here is more restrictive than that of~\cite{BorodinOlshanski2005}, as we require both $\rho_1,\rho_2$ to be Jack-positive specializations; this will suffice for our purposes.
In fact, we are mostly interested here in the Jack measures when $\rho_1$ is Jack-positive and $\rho_2$ is the pure alpha specialization with exactly $N$ parameters $\alpha_i$ being equal to $1$, and $\delta=N$, and the rest of Thoma parameters are equal to $0$; this specialization was denoted $\Alpha(1^N)$ in \cref{exam:pure_alpha}.
In other words, $\rho_2$ is defined by
\begin{equation*}
p_k(\rho_2)=N,\quad\textrm{for all }k\ge 1.
\end{equation*}
Then the Jack measure $\PP^{(\theta)}_{\rho_1;\rho_2}$ depends only on $N\in\Z_{\ge 1}$ and one Jack-positive specialization $\rho_1$, to be denoted $\rho$; in this case, the Jack measure is supported on the subset $\Y(N)\subset\Y$, since $P_\la(1^N;\theta)\ne 0$, only if $\la\in\Y(N)$.
Moreover, observe that the finiteness condition~\eqref{eq:finiteness} is satisfied if the series
\begin{equation}\label{eq:stable_series}
\sum_{k\ge 1}\frac{|p_k(\rho)|}{k}z^k
\end{equation}
is absolutely convergent on some open annulus containing the unit circle.

\begin{definition}\label{def:stable}
We call a Jack-positive specialization $\rho$ \textbf{stable} if the series~\eqref{eq:stable_series} is absolutely convergent on $R^{-1} < z < R$, for some $R > 1$.
In this case, the Jack measure $\PP^{(\theta)}_{\rho;\Alpha(1^N)}$ (which is well-defined and supported on $\Y(N)$) will be denoted by $\PP_{\rho;N}^{(\theta)}$ and its Jack generating function will be denoted by $G_{\rho;N,\theta}(x_1,\dots,x_N)$.
\end{definition}

As an application of the Cauchy identity~\eqref{eqn:cauchy}, we deduce
\begin{equation}\label{eqn:Gt_product}
G_{\rho;N,\theta}(x_1,\dots,x_N) = \frac{H_\theta(\rho;\xx_N)}{H_\theta(\rho;1^N)} = \prod_{i=1}^{N} \exp\Bigg\{\theta\sum_{k\ge 1}\frac{p_k(\rho)}{k}\,(x_i^k-1)\Bigg\},
\end{equation}
where $\xx_N:=(x_1,\dots,x_N)$, so its natural logarithm will be
\begin{equation*}
F_{\rho;N,\theta}(x_1,\dots,x_N) := \ln{G_{\rho;N,\theta}(x_1,\dots,x_N)} = \theta\sum_{i=1}^{N}{\sum_{k\ge 1}\frac{p_k(\rho)}{k}\,(x_i^k-1)}.
\end{equation*}
As an immediate consequence, we have the following lemma.

\begin{lemma}\label{lem:check_ht_appropriateness}
Let $\rho$ be stable and let $F_{\rho;N,\theta}(x_1,\dots,x_N)$ be the logarithm of the Jack generating function of $\PP_{\rho;N}^{(\theta)}$. Then $\PP_{\rho;N}^{(\theta)}$ has the small tails property, and for all $r\ge 2$ and $i_1,\dots,i_r\in\Z_{\geq 1}$ with at least two distinct indices among $i_1,\dots,i_r$, we have
$\partial_{i_1}\cdots \partial_{i_r}F_{\rho;N,\theta} = 0$.
Moreover, for any $n\in\Z_{\ge 1}$,
\begin{equation}\label{eqn:first_condition_HT}
\frac{\partial_1^n F_{\rho;N,\theta}}{(n-1)!}\,\bigg|_{(x_1,\dots,x_N) = (1^N)} 
= \frac{\theta}{(n-1)!}\cdot\partial_x^n\sum_{k\ge 1}\frac{p_k(\rho)}{k}x^k\bigg|_{x=1}.
\end{equation}
\end{lemma}

This lemma shows that the second condition in the definition of HT-appropriateness (namely, \cref{def:appropriate}) always holds for $\big\{ \PP^{(\theta)}_{\rho;N} \big\}_{N\ge 1}$ with arbitrary stable $\rho$.
Moreover, the first condition is also satisfied if~\eqref{eqn:first_condition_HT} has a limit in the regime~\eqref{eq:HTRegime}, for all $n\in\Z_{\ge 1}$.

\subsubsection{Markov chains of nonintersecting random walks}

Recall from \cref{sec:jack_specs} the definitions of union $(\rho_1,\rho_2)$ and $n$-fold union $n\rho$, for specializations $\rho,\rho_1,\rho_2$.
These operations preserve the property of Jack-positiveness.

\begin{proposition}[\cite{GorinShkolnikov2015}]\label{prop:Transition}
For any Jack-positive specializations $\rho_1,\rho_2\colon\Sym\to\R$ satisfying the finiteness assumption~\eqref{eq:finiteness}, define
\begin{equation}\label{eq:transition_probs}
p_{\la\to\mu}(\rho_1,\rho_2;\theta) := \frac{1}{H_\theta(\rho_1;\rho_2)}\cdot \frac{P_\mu(\rho_2;\theta)}{P_\la(\rho_2;\theta)}\cdot Q_{\mu/\la}(\rho_1;\theta),\quad\la,\mu\in\Y(\rho_2),
\end{equation}
where $\Y(\rho_2)$ denotes the set of partitions $\nu$ such that $P_\nu(\rho_2; \theta)>0$, and where $Q_{\mu/\la}(;\theta)$ is defined by~\eqref{Q_skew}.
Then $\left(p_{\la\to\mu}(\rho_1,\rho_2;\theta)\right)_{\la,\mu\in\Y(\rho_2)}$ is a stochastic matrix, so that its entries define transition probabilities for a Markov chain on the set of partitions $\Y(\rho_2)$. Moreover, if $\rho_3$ is another Jack-positive specialization such that~\eqref{eq:finiteness} is satisfied for $\rho_2,\rho_3$, then
\begin{align*}
\sum_{\la\in\Y(\rho_2)}\PP^{(\theta)}_{\rho_1;\rho_2}(\la)p_{\la\to\mu}(\rho_3,\rho_2;\theta) &= \PP^{(\theta)}_{(\rho_1,\rho_3);\rho_2}(\mu),\\
\sum_{\mu\in\Y(\rho_2)}p_{\la\to\mu}(\rho_1,\rho_2;\theta)p_{\mu \to \nu}(\rho_3,\rho_2;\theta) &= p_{\la\to\nu}((\rho_1,\rho_3),\rho_2;\theta).
\end{align*}
\end{proposition}

\smallskip

We are interested in the Markov chain defined in~\cite{Huang2021} by means of the transition probabilities $p_{\la\to\mu}(\rho_1,\rho_2;\theta)$ when $\rho_2$ is a pure alpha specialization.

\begin{definition}[\cite{Huang2021}]\label{def:Markov_chain}
Let $p_{\la\to\mu}(\rho_1,\rho_2;\theta)$ be the transition probabilities defined in \cref{prop:Transition}.
When $\rho_2$ is the pure alpha specialization $\Alpha(1^N)$, defined by $p_k(\rho_2)=N$, for all $k\in\Z_{\ge 1}$, and $\rho_1$ is denoted simply by $\rho$, we denote $p_{\la\to\mu}(\rho;N;\theta):=p_{\la\to\mu}(\rho,\Alpha(1^N);\theta)$.
Note also that $\Y(\rho_2)=\Y(N)$ because $P_\la(1^N;\theta)\ne 0$ if and only if $\la\in\Y(N)$, due to the combinatorial formula for Jack polynomials~\cite[Ch.~VI.10]{Macdonald1995}.

Next, let $\PP_N$ be a fixed probability measure on $\Y(N)$ and $\rho$ a Jack-positive specialization.
Define the associated Markov chain $\la^{(0)},\la^{(1)},\la^{(2)},\dots \in \Y(N)$ by
\[
\PP\big(\la^{(0)} = \la\big) = \PP_N(\la),\quad \PP\big(\la^{(n+1)} = \mu \,\big|\, \la^{(n)} = \la\big) = p_{\la\to\mu}(\rho;N;\theta), \quad n \in \Z_{\geq 0}.
\]
Also, define the auxiliary Markov chain
\begin{equation*}
  \LL^{(0)},\LL^{(1)},\LL^{(2)},\dots,\quad \LL^{(n)} = \Big(\LL^{(n)}_1,\dots,\LL^{(n)}_N\Big),\quad n\in \Z_{\ge 0},
\end{equation*}
by $\LL^{(n)}_i := \la^{(n)}_i - \theta(i-1)$, for all $i=1,\dots,N$ and $n\in\Z_{\ge 0}$.
\end{definition}
The sequence $\LL^{(0)},\LL^{(1)},\LL^{(2)},\dots$ is a discrete Markov chain of $N$ nonintersecting particles.
The initial condition $\LL^{(0)}$ is random and determined by $\PP_N$.
In three special cases, the nonintersecting random walks have interesting interpretations:\footnote{Note that some interpretations are given in~\cite{Huang2021} (see immediately before his Theorem 5.8), but they are slightly inaccurate, as the Markov processes are not simply independent walks conditioned on not intersecting.}

\begin{itemize}
\setlength\itemsep{.1cm}

\item When $\theta = 1$ and $\rho$ is the pure alpha specialization (see~\cref{exam:pure_alpha}) $p_k(\rho) = p^k$, for some $0<p<1$, we have that $\big\{\LL^{(n)}\big\}_{n\ge 0}$ is a geometric-type random walk on $\Z$ in the sense that the $N$ particles can jump to the right by $k$ steps, where $k\in\Z_{\ge 0}$, but subject to certain conditions including the preservation of the nonintersecting property; moreover, when $N=1$, the process is just one particle doing independent jumps with geometric law of parameter $p$.
For general $\theta$ and $N$, the Markov chain is not a Doob $h$-transform, since~\eqref{eq:transition_probs} contains a ratio of Jack polynomials, not of Vandermonde determinants.

\item When $\theta = 1$ and $\rho$ is the pure beta specialization (see~\cref{exam:pure_beta}) $p_k(\rho) = (-1)^{k-1}p^k$, for some $0<p<1$, we have that $\big\{\LL^{(n)}\big\}_{n\ge 0}$ is a Bernoulli-type random walk on $\Z$ in the sense that the $N$ particles can jump to the right by $1$ or do not move, as long as the nonintersecting property is preserved; moreover, when $N=1$, the process is just one particle in $\Z$ moving at each time to the right by one with probability $\frac{p}{1+p}$.

\item When $\theta = 1$ and $\rho$ is the Plancherel specialization (see~\cref{exam:plancherel}) $p_k(\rho) = \chi\cdot\delta_{k,1}$, for some $\chi>0$, we have that $\big\{\LL^{(n)}\big\}_{n\ge 0}$ is a nonintersecting Poisson-type random walk on $\Z$, in the sense that when $N=1$, the process is just one particle doing independent Poisson jumps of intensity $\chi$.
\end{itemize}

For a general $\theta>0$, the defined Markov chain is a natural one-parameter deformation of log-gas type of the above examples.

\vspace{.1cm}

We are interested in the fixed-time distribution of the Markov chain in the high temperature regime.
In other words, we will study the empirical measure
\[
\mu^{(t)}_N := \frac{1}{N}\sum_{i=1}^N\delta_{\LL^{(t)}_i},
\]
where $t\in\Z_{\ge 0}$ possibly depends on $N$, in the limit as $N\to\infty$, $N\theta\to\gamma$. 
For this, let us compute the Jack generating function $G^{(t)}_{\rho;N,\theta}(x_1,\dots,x_N)$ of the law of $\la^{(t)}$.
From the second equation in \cref{prop:Transition}, we have that
\begin{equation}\label{special_jack_measure}
\PP\left( \la^{(t)} = \la\right) = \frac{1}{H_\theta(t \rho; 1^N)}\cdot \sum_{\la^{(0)} \in \Y(N)}\PP_N(\la^{(0)})\cdot\frac{P_\la(1^N;\theta)}{P_{\la^{(0)}}(1^N;\theta)}\cdot Q_{\la/\la^{(0)}}(t\rho;\theta).
\end{equation}

Note that when $\PP_N(\la) = \delta_{\la,\emptyset}$, the RHS of~\eqref{special_jack_measure} is the probability $\PP^{(\theta)}_{t\rho;N}(\la)$, according to the Jack measure $\PP^{(\theta)}_{t\rho;N}$ from \cref{def:stable}. Let us record this fact, as it will be useful later.

\begin{lemma}\label{lem:empty_initial}
Let $\{\la^{(t)}\}_{t=0,1,2,\dots}$ be the Markov chain from \cref{def:Markov_chain} with Jack-positive specialization $\rho$ and initial law of $\la^{(0)}$ being $\PP_N(\la) = \delta_{\la,\emptyset}$.
Then the law of $\la^{(t)}$ is the Jack measure $\PP^{(\theta)}_{t\rho;N}$, where $t\rho=(\rho,\dots,\rho)$ ($t$-fold union).
\end{lemma}

\begin{remark}\label{rem:GS}
Gorin and Shkolnikov defined in~\cite{GorinShkolnikov2015} the continuous-time Markov process $\{X_N^{disc}(s) \colon s\in\R_{\ge 0}\}$ on the discrete space $\Y(N)$, with the law given by \eqref{special_jack_measure} with $\rho$ being the pure Plancherel specialization with parameter $\delta = 1$, so that $t\rho$ is the pure Plancherel specialization with parameter $\delta = t$.
In our case, $t\in\Z_{\ge 0}$, so the process $X_N^{disc}(t)$ can be regarded as a continuous-time version of ours in the special case when $\rho$ is pure Plancherel.
\end{remark}

Let us return to~\eqref{special_jack_measure} again with a general measure $\PP_N$. The Jack generating function $G^{(t)}_{\rho;N,\theta}(x_1,\dots,x_N)$ of the law of $\la^{(t)}$ is then calculated as
\begin{align}
G^{(t)}_{\rho;N,\theta}(x_1,\dots,x_N) &= \frac{1}{H_\theta(t\rho;1^N)}\nonumber\\
&\quad\times\sum_{\la,\la^{(0)}\in\Y(N)} \PP_N(\la^{(0)})\cdot\frac{P_\la(1^N;\theta)}{P_{\la^{(0)}}(1^N;\theta)}\cdot Q_{\la/\la^{(0)}}(t\rho;\theta)\cdot \frac{P_\la(x_1,\dots,x_N;\theta)}{P_\la(1^N;\theta)}\nonumber\\
&= \frac{1}{H_\theta(t\rho;1^N)}\cdot
\sum_{\la,\la^{(0)}\in\Y(N)} \PP_N(\la^{(0)})\,\frac{Q_{\la/\la^{(0)}}(t \rho;\theta) P_\la(x_1,\dots,x_N;\theta)}{P_{\la^{(0)}}(1^N;\theta)}\nonumber\\
&= \frac{H_\theta(t\rho;\xx_N)}{H_\theta(t\rho;1^N)}\cdot\sum_{\la^{(0)}\in\Y(N)}\PP_N(\la^{(0)})\,\frac{P_{\la^{(0)}}(x_1,\dots,x_N;\theta)}{P_{\la^{(0)}}(1^N;\theta)}\nonumber\\
&= G_{t\rho;N,\theta}(x_1,\dots,x_N)\cdot G_{\PP_N,\theta}(x_1,\dots,x_N),\label{eqn:Gt}
\end{align}
where $G_{t\rho;N,\theta}(x_1,\dots,x_N)$ is the Jack generating function of the Jack measure $\PP_{t\rho;N}^{(\theta)}$.
We remark that the third equality is a consequence of the identity~\eqref{eqn:cauchy_skew}, while the last one follows from~\eqref{eqn:Gt_product} and the definition~\eqref{eq:JackGeneratFunction} of Jack generating functions.

\begin{theorem}\label{theo:LLN_Markov}
(1) Let $\{\theta_N>0\}_{N\ge 1}$ be such that $N\theta_N\to\gamma$, as $N\to\infty$.
Let $\rho=\{\rho_N\}_{N\ge 1}$ be a sequence of stable Jack-positive specializations and assume that $\{t_N\in\Z_{\ge 1}\}_{N\ge 1}$ are such that the limits
\begin{equation}\label{eqn:limit_trho}
\kappa_n^\rho := \lim_{\substack{N\to\infty\\N\theta\to\gamma}} 
{ \frac{t_N\theta_N}{(n-1)!}\cdot\partial_x^n\sum_{k\ge 1}\frac{p_k(\rho_N)}{k}x^k\bigg|_{x=1} }
\end{equation}
exist, for all $n\in\Z_{\ge 1}$.
Then the sequence of Jack measures $\PP_{t_N\rho_N;N}^{(\theta_N)}$ satisfies the LLN with limiting quantized $\gamma$-cumulants $\kappa_n^\rho$, $n\ge1$.

If, in addition, there exists a constant $C_\rho>0$ such that
\begin{equation}\label{eq:rho_growth}
|\kappa_n^\rho|\le (C_\rho)^n,\quad\text{for all }n\ge1,
\end{equation}
then the limiting moment sequence determines a unique probability measure, which we denote by $\mu_\rho$, and the empirical measures converge weakly, in probability, to $\mu_\rho$.

\smallskip

(2) Additionally, let $\,\PP_N$ be probability measures on $\Y(N)$ with the small tails property, and let $\{\la^{(0)}(N)\in\Y(N)\}_{N\ge 1}$ be a sequence of $\PP_N$-distributed partitions. Assume that their empirical measures are all supported on the same compact interval $[-L,L]$, and their weak limit, in probability,
\begin{equation}\label{eq:assumption_LLN}
\lim_{\substack{N\to\infty\\N\theta\to\gamma}} \mu^{(0)}_N = \mu^{(0)}
\end{equation}
exists. Then the law of $\la^{(t_N)}(N)$ from the Markov chain $\la^{(0)}(N),\la^{(1)}(N),\la^{(2)}(N),\cdots$, given by \cref{def:Markov_chain}, satisfies the LLN with limiting quantized $\gamma$-cumulants
\begin{equation}\label{sum_cumulants}
\kappa_n^{(t)} = \kappa_n\big[\mu^{(0)}\big]+\kappa_n^\rho,\qquad n\ge1.
\end{equation}
If the additional growth condition~\eqref{eq:rho_growth} from part~(1) is satisfied, then there exists a unique probability measure with the moments determined by~\eqref{sum_cumulants}. We denote it by
$\mu^{(0)}\boxplus^{(\gamma)}\mu_\rho$, and the empirical measures of $\la^{(t_N)}(N)$ converge weakly, in probability, to this measure.
\end{theorem}

\begin{proof}
The LLN statement in part~(1) follows directly from our main \cref{theo:main1}, \cref{lem:check_ht_appropriateness}, and our assumption~\eqref{eqn:limit_trho}.
Under the additional condition~\eqref{eq:rho_growth}, the uniqueness of $\mu_\rho$ and weak convergence follow from the last part of \cref{theo:main1}.

For part~(2), due to \cref{theo:main1}, it will be enough first to show that the sequence of laws of $\la^{(t_N)}(N)$ is HT-appropriate.
Equations~\eqref{eqn:Gt}, \eqref{eqn:Gt_product}, together with the assumptions that $\rho_N$ is stable and $\PP_N$ has the small tails property imply the small tails property for the law of $\la^{(t_N)}(N)$.
Moreover, \eqref{eqn:Gt} implies that
\begin{equation}\label{eq:sum_logs}
\ln G^{(t_N)}_{\rho_N;\,N,\theta_N}(x_1,\dots,x_N)
=
\ln{G_{\PP_N,\theta_N}(x_1,\dots,x_N)}
+
F_{t_N\rho_N;\,N,\theta_N}(x_1,\dots,x_N).
\end{equation}
Equation~\eqref{eq:assumption_LLN} and the fact that all $\mu^{(0)}_N,\mu^{(0)}$ are supported on the same compact interval imply also the convergence $\mu_N^{(0)}\to\mu^{(0)}$ in the sense of moments, in probability.
Therefore, \cref{theo:main1} shows that the sequence $\{\PP_N\}_{N\ge1}$ is HT-appropriate.
Combined with \cref{lem:check_ht_appropriateness} and the assumption~\eqref{eqn:limit_trho}, we deduce from~\eqref{eq:sum_logs} that the sequence of laws of $\la^{(t_N)}(N)$ is HT-appropriate with limiting quantized $\gamma$-cumulants~\eqref{sum_cumulants}. Hence it satisfies the LLN.

It remains to prove the last assertion of part (2).
Since all the initial empirical measures are supported on $[-L,L]$ and the $\la^{(0)}(N)$ are partitions, the one-variable Jack generating functions
\[
g_N(x):=G_{\PP_N,\theta_N}(x,1,\dots,1)
\]
are polynomials of uniformly bounded degree, with nonnegative coefficients and $g_N(1)=1$.
It follows that $\log g_N$ are uniformly analytic in a fixed neighborhood of $x=1$.
Hence, by Cauchy estimates and HT-appropriateness of $\{\PP_N\}_{N\ge1}$, there exists a constant $C_0>0$ such that
\[
|\kappa_n[\mu^{(0)}]|\le (C_0)^n,\quad\text{for all }n\ge1.
\]
Together with~\eqref{eq:rho_growth}, this gives an exponential bound for the cumulants $\kappa_n^{(t)}$ (see~\eqref{sum_cumulants}).
The last part of \cref{theo:main1} therefore proves the uniqueness of the corresponding probability measure, as well as the weak convergence in probability.
\end{proof}

A few remarks are in order.
Without the additional growth condition~\eqref{eq:rho_growth}, \cref{theo:LLN_Markov} is a statement about convergence in the sense of moments.
When~\eqref{eq:rho_growth} holds, part~(2) shows in addition that the quantized $\gamma$-convolution $\mu^{(0)}\boxplus^{(\gamma)}\mu_\rho$ exists and that the empirical measures converge to it weakly, in probability.
Below, we discuss the three families of \emph{pure specializations}; in each of these cases the relevant limiting cumulants satisfy an exponential bound.

\smallskip

Also, note that the sequence $\{\PP_N\}_{N\ge 1}$ of probability measures $\PP_N(\la)=\delta_{\la,\emptyset}$ have Jack generating functions $G_{\PP_N,\theta}(x_1,\dots,x_N)=1$ and the limit of their empirical measures is
\begin{equation}\label{eq:EmptyPartitionLimit}
\lim_{\substack{N\to\infty\\N\theta\to\gamma}}\frac{1}{N}\sum_{i =0}^{N-1}\delta_{-i\theta} = \gamma^{-1}\cdot\mathbf{1}_{[-\gamma,0]}(x)\text{d}x.
\end{equation}
By \cref{theo:main1}, the corresponding quantized $\gamma$-cumulants are $\kappa_n[\gamma^{-1}\cdot\mathbf{1}_{[-\gamma,0]}(x)\text{d}x]=0$, for all $n\in\Z_{\ge 1}$; see item (1) of \cref{def:appropriate}.

As an interesting consequence, for any compactly supported probability measure $\mu$ arising in the setting of part~(2) of \cref{theo:LLN_Markov}, we have the quantized $\gamma$-convolution identity
\begin{equation}\label{eq:Fixpoint_for_Convolution}
\mu\boxplus^{(\gamma)}\big(\gamma^{-1}\cdot\mathbf{1}_{[-\gamma,0]}(x)\text{d}x\big) = \mu.
\end{equation}
Thus the measure $\gamma^{-1}\cdot\mathbf{1}_{[-\gamma,0]}(x)\text{d}x$ plays the role of the neutral element for the quantized $\gamma$-convolution in this setting.

\subsubsection{Pure-alpha Jack measure}

Let $c\in(0,1)$, $\eta>0$ be arbitrary constants, and let $N\in\Z_{\ge 1}$.
Let us consider the pure alpha specialization $\Alpha(c^{\lfloor N\eta\rfloor})$ (see~\cref{exam:pure_alpha}) defined by $f\big(\Alpha(c^{\lfloor N\eta\rfloor})\big) := f(c^{\lfloor N\eta\rfloor})$, for all $f\in\Sym$, or equivalently,
\begin{equation*}
p_k\big(\Alpha(c^{\lfloor N\eta\rfloor})\big) = \lfloor N\eta\rfloor\cdot c^k,\ \text{ for all $k\ge 1$}.
\end{equation*}
Then the normalization constant for the Jack measure $\PP^{(\theta)}_{\Alpha(c^{\lfloor N\eta\rfloor});\,N}$, given by~\eqref{def:jack_measure}, is
\begin{equation*}
\exp\Bigg\{ N\theta\sum_{k\ge 1}{\frac{p_k\big(\Alpha(c^{\lfloor N\eta\rfloor})\big)}{k}} \Bigg\}
= \exp\big( -N\theta\lfloor N\eta\rfloor\cdot\ln(1-c)\big) = (1-c)^{-N\theta\lfloor N\eta\rfloor}.
\end{equation*}
It follows that
\begin{equation*}
\PP^{(\theta)}_{\Alpha(c^{\lfloor N\eta\rfloor});\,N}(\la) = (1-c)^{N\theta\lfloor N\eta\rfloor}\, c^{|\la|}
\,Q_\la(1^{\lfloor N\eta\rfloor};\theta) P_\la(1^N;\theta),
\end{equation*}
and this is nonzero whenever $\ell(\la)\le\min\{N, \lfloor N\eta\rfloor\}$.
The previous equation can be made more explicit with the help of~\eqref{j_lambda} and the following evaluation identity (see \cite[Ch.~VI.10, Eqn.~(10.20)]{Macdonald1995} or \cite[Thm.~5.4]{Stanley1989}):
\begin{equation}\label{eqn:jack_evaluation}
P_\la(1^N;\theta) = \prod_{(i,j)\in\la}\frac{ N\theta + (j-1) - \theta(i-1) }{(\la_i-j)+\theta(\la_j'-i)+\theta},
\end{equation}
valid for all $N\in\Z_{\ge 1}$ and $\la\in\Y$. We summarize this discussion in the following.

\begin{definition}\label{def:example1}
Let $c\in(0,1)$, $\eta,\theta>0$, and $N\in\Z_{\ge 1}$.
The \textbf{pure-alpha Jack measure} associated to these parameters is the probability measure on partitions $\la$ with $\ell(\la)\le\min\{N, \lfloor N\eta\rfloor\}$, defined by
\begin{multline*}
\PP^{(\theta)}_{\Alpha(c^{\lfloor N\eta\rfloor});\,N}(\la) := (1-c)^{N\theta\lfloor N\eta\rfloor} c^{|\la|}\\
\times\prod_{(i,j)\in\la}{ \frac{ (\lfloor N\eta\rfloor\theta + (j-1) - \theta(i-1)) (N\theta + (j-1) - \theta(i-1)) }
{ ((\la_i-j)+\theta(\la_j'-i)+\theta)((\la_i-j)+\theta(\la_j'-i)+1) } },
\end{multline*}
for all $\la\in\Y\big(\min\{N,\lfloor N\eta\rfloor\}\big)$.\footnote{We include the formulas in \Cref{def:example1,def:example2,def:example3} only to show explicit expressions for the pure-alpha, pure-beta and pure-Plancherel Jack measures, but they are not used in our proofs.}
\end{definition}

\begin{remark}
One can show that
\[
\PP^{(\theta)}_{\Alpha(c^{\lfloor N\eta\rfloor});\,N}(\la)\propto\prod_{1\le i<j\le N}\frac{\Gamma(\LL_i-\LL_j+1)\Gamma(\LL_i-\LL_j+\theta)}{\Gamma(\LL_i-\LL_j)\Gamma(\LL_i-\LL_j+1-\theta)}\prod_{k=1}^N{\frac{\Gamma(\LL_k+\lfloor N\eta \rfloor\theta)}{\Gamma(\LL_k+(N-1)\theta+1)}\,c^{\LL_k}},
\]
where $\LL_i=\la_i-\theta(i-1)$ are the shifted coordinates.
This shows that the pure-alpha Jack measure is a discrete beta ensemble, in the sense of \cite{BorodinGorinGuionnet2017}, with weight function $w(x;N):=\Gamma(x+\lfloor N\eta\rfloor\theta)\, c^x/\Gamma(x+(N-1)\theta+1)$.
Similarly, the pure-beta and pure-Plancherel Jack measures in \cref{def:example2,def:example3} below are also discrete beta ensembles.
\end{remark}

We claim that $\left\{\PP^{(\theta)}_{\Alpha(c^{\lfloor N\eta\rfloor});\,N}\right\}_{N\ge 1}$ satisfies the two conditions of HT-appropriateness from \cref{def:appropriate}, when $c\in(0,1)$, $\eta>0$ are fixed.
By \cref{lem:check_ht_appropriateness}, the second condition is automatic, while the first follows from Equation~\eqref{eqn:first_condition_HT} and the following calculation:
\begin{multline}\label{eqn:first_condition_HT_example1}
\lim_{\substack{N\to\infty\\N\theta\to\gamma}} 
{ \frac{\theta}{(n-1)!}\partial_x^n\sum_{k\ge 1}\frac{p_k\big( \Alpha(c^{\lfloor N\eta\rfloor}) \big)}{k}x^k\bigg|_{x=1} }
= \frac{1}{(n-1)!}\lim_{\substack{N\to\infty\\N\theta\to\gamma}} { \theta \lfloor N\eta\rfloor \, \partial_x^n\sum_{k\ge 1}\frac{(cx)^k}{k}\bigg|_{x=1} }\\
= -\frac{\gamma\eta}{(n-1)!}{ \partial_x^n \, \ln(1-cx) \Big|_{x=1} }
= -\frac{\gamma\eta}{(n-1)!} { \frac{-(n-1)!\cdot c^n}{(1-cx)^n} \bigg|_{x=1} }
= \frac{\gamma\eta c^n}{(1-c)^n}.
\end{multline}

\begin{theorem}\label{thm:application1}
Let $c\in(0,1)$ and $\gamma,\eta>0$ be fixed.
Let $\{\theta_N>0\}_{N\ge 1}$ be such that $N\theta_N\to\gamma$, as $N\to\infty$.
Let $\big\{\la^{(t)}(N)\big\}_{t=0,1,2,\dots}$ be the Markov chain from \cref{def:Markov_chain} with $\rho_N = \Alpha(c)$ (independent of $N$), and initial laws of $\la^{(0)}(N)$ given by $\PP_N$.

\smallskip

(1) If $\PP_N(\la)=\delta_{\la,\emptyset}$, then at time $t=\lfloor N\eta\rfloor$, the law of $\la^{(t)}(N)=\la^{(\lfloor N\eta\rfloor)}(N)$ is the pure-alpha Jack measure $\PP^{(\theta_N)}_{\Alpha(c^{\lfloor N\eta\rfloor});\,N}$.
Moreover, this sequence of Jack measures satisfies the LLN and the corresponding empirical measures converge weakly, in probability, to the unique probability measure $\mu_{\Alpha}^{\gamma;c;\eta}$ with moments
\begin{multline*}
\int_\R{ x^\ell \mu_{\Alpha}^{\gamma;c;\eta}(\dd\!x) } = 
\sum_{\Gamma\in\Luk(\ell)}
(a+ab)^{\text{\#\,up steps}} \, b^{\text{\#\,down steps}} \ \frac{\Delta_\gamma\left(x^{1+\text{\#\,horizontal steps at height $0$}}\right)(ab)}{1+\text{\#\,horizontal steps at height $0$}}\\
\cdot\prod_{j\ge 1} (j+ab)^{\text{\#\,horizontal steps at height $j$}}
(j+\gamma)^{\text{\#\,down steps from height $j$}},\quad\text{for all $\ell\in\Z_{\ge 1}$},
\end{multline*}
where we set $a:=\gamma\eta>0$ and $b:=c/(1-c)>0$, for convenience.

\smallskip

(2) In the general case, when $\PP_N$ are such that the empirical measures $\mu_N^{(0)}$ of $\la^{(0)}(N)$ converge to $\mu^{(0)}$, as in~\eqref{eq:assumption_LLN}, then the empirical measures of $\la^{(\lfloor N\eta\rfloor)}(N)$ converge weakly, in probability, to
\[
\lim_{\substack{N\to\infty\\N\theta\to\gamma}}
\mu_N^{(\lfloor N\eta\rfloor)} 
= \mu^{(0)} \boxplus^{(\gamma)}\mu_{\Alpha}^{\gamma;c;\eta},
\]
where the RHS is understood as the measure with quantized $\gamma$-cumulants $\kappa_n\big[\mu^{(0)}\big] + \kappa_n\big[\mu_{\Alpha}^{\gamma;c;\eta}\big]$. The existence and uniqueness of this probability measure is part of the statement of the theorem.
\end{theorem}
\begin{proof}
The fact that the law of $\la^{(\lfloor N\eta\rfloor)}(N)$ is $\PP^{(\theta_N)}_{\Alpha(c^{\lfloor N\eta\rfloor});\,N}$ if $\PP_N(\la)=\delta_{\la,\emptyset}$ follows from \cref{lem:empty_initial}.
Next, thanks to Equations \eqref{eqn:first_condition_HT} and \eqref{eqn:first_condition_HT_example1}, as well as \cref{lem:check_ht_appropriateness}, it follows that the sequence $\Big\{ \PP^{(\theta_N)}_{\Alpha(c^{\lfloor N\eta\rfloor});\,N} \Big\}_{N\ge 1}$ is HT-appropriate with $\kappa_n = ab^n$.
Note that $|\kappa_n|=ab^n\le (\max\{a,1\}\cdot b)^n$, for all $n\ge 1$, therefore we can apply \cref{theo:main1}, which completes the proof of part (1). Finally, part (2) follows from \cref{theo:LLN_Markov}.
\end{proof}

\begin{corollary}
Assume we are in the context of \cref{thm:application1}.
If $\eta=(1-c)/c$, then the moments of $\mu_{\Alpha}^{\gamma;c;\eta}=\mu_{\Alpha}^{\gamma;c;\frac{1-c}{c}}$ can be expressed as
\begin{multline*}
\int_\R{ x^\ell \mu_{\Alpha}^{\gamma;c;\frac{1-c}{c}}(\dd\!x) } = 
\sum_{\Gamma\in\Luk(\ell)}
\frac{(1-c)^{-\,\text{\#\,down steps}} \, c^{\text{\#\,down steps} \,-\, \text{\#\,up steps}}}{1+\text{\#\,horizontal steps at height $0$}}\\
\cdot\gamma^{\text{\#\,up steps} \ +\ \text{\#\,horizontal steps at height $0$}}
\cdot\prod_{j\ge 1} (j+\gamma)^{\text{\#\,horizontal steps at height $j$}\,\,+\,\,\text{\#\,down steps from height $j$}}.
\end{multline*}
\end{corollary}
\begin{proof}
By the proof of \cref{thm:application1}, $\kappa_1=\gamma$ iff $\eta = (1-c)/c$, therefore \cref{lem:special_case} gives the desired formula.
\end{proof}

\subsubsection{Pure-beta Jack measure}\label{sec:pure_beta}

Let $c>0$ and $N, M\in\Z_{\ge 1}$.
Let us consider the pure beta specialization $\Beta((c/\theta)^M)$ (see~\cref{exam:pure_beta}), defined by
\begin{equation*}
p_k\big(\Beta((c/\theta)^M)\big) := (-1)^{k-1}\theta^{-1}M\,c^k,\ \text{ for all $k\ge 1$}.
\end{equation*}

Recall that the pure beta specialization $\Beta((c/\theta)^M)$ is equal to the composition of the $\Sym$-automorphism $\omega_{\theta^{-1}}$ defined by $\omega_{\theta^{-1}}(p_r)=(-1)^{r-1}\theta^{-1}p_r$, for all $r\ge 1$, and the pure alpha specialization $\Alpha(c^M)$.
In particular, since $\omega_{\theta^{-1}}(Q_\la(;\theta)) = P_{\la'}(;\theta^{-1})$ (see \cite[Ch.~VI.10, Eqn.~(10.17)]{Macdonald1995}), then
\begin{equation}\label{eqn:jack_eval_beta}
Q_\la\big(\Beta\big((c/\theta)^M\big);\theta\big) = c^{|\la|}\cdot P_{\la'}(1^M;\theta^{-1}).
\end{equation}
Note that this expression is nonzero only if $\ell(\la')\le M$.
Moreover, $P_\la(1^N;\theta)\ne 0$ only if $\ell(\la)\le N$, therefore the product $P_\la(1^N;\theta)Q_\la\big(\Beta\big((c/\theta)^M\big);\theta\big)$ is nonzero only for partitions $\la\subseteq(M^N)$ that fit into an $N\times M$ rectangle.
In particular, the corresponding Jack measure is finitely supported for every $c>0$.

The dual Cauchy identity~\eqref{eqn:cauchy_dual} gives
\begin{equation}\label{eqn:norm_beta}
\sum_{\la\subseteq(M^N)}
c^{|\la|}P_\la(1^N;\theta)P_{\la'}(1^M;\theta^{-1})
=(1+c)^{NM}.
\end{equation}
The identities \eqref{eqn:jack_eval_beta}-\eqref{eqn:norm_beta} then show that the Jack measure~\eqref{def:jack_measure} admits the following formula
\begin{equation*}
\PP^{(\theta)}_{\Beta((c/\theta)^M);\,N}(\la) = (1+c)^{-NM}\, c^{|\la|}
P_\la(1^N;\theta) P_{\la'}(1^M;\theta^{-1}).
\end{equation*}
By virtue of~\eqref{eqn:jack_evaluation}, we obtain the explicit formula in the following definition.
\begin{definition}\label{def:example2}
Let $c,\theta>0$, and $N,M\in\Z_{\ge 1}$.
Then the \textbf{pure-beta Jack measure} associated to these parameters is the probability measure on partitions $\la\subseteq (M^N)=(\underbrace{M,M,\dots,M}_{N \text{ times}})$, defined by
\begin{equation*}
\PP^{(\theta)}_{\Beta((c/\theta)^M);\,N}(\la) :=  (1+c)^{-NM}\, c^{|\la|} \prod_{(i,j)\in\la} \frac{ (N\theta+(j-1)-\theta(i-1))(M+\theta(i-1)-(j-1)) }{ ((\la_i-j)+\theta(\la_j'-i)+\theta)((\la_i-j)+\theta(\la_j'-i)+1) },
\end{equation*}
for all $\la\subseteq(M^N)$.
\end{definition}

The same dual Cauchy identity gives the Jack generating function
\begin{equation}\label{eqn:G_beta_product}
G_{\Beta((c/\theta)^M);N,\theta}(x_1,\dots,x_N)
=
\prod_{i=1}^N\left(\frac{1+cx_i}{1+c}\right)^M.
\end{equation}
Since this is a polynomial, the corresponding Jack measure has the small tails property for all $c>0$.
Taking the logarithm of~\eqref{eqn:G_beta_product} in a neighborhood of $(1^N)$, we obtain
\[
F_{\Beta((c/\theta)^M);N,\theta}(x_1,\dots,x_N)
=
M\sum_{i=1}^N\ln\left(\frac{1+cx_i}{1+c}\right).
\]
In particular, all mixed derivatives vanish, and for every $n\in\Z_{\ge1}$,
\begin{equation}\label{eqn:first_condition_HT_example2}
\frac{\partial_1^n F_{\Beta((c/\theta)^M);N,\theta}}{(n-1)!}\bigg|_{(x_1,\dots,x_N)=(1^N)}
=
\frac{(-1)^{n-1}Mc^n}{(1+c)^n}.
\end{equation}
Thus this sequence of pure-beta Jack measures is HT-appropriate for every $c>0$.

\begin{remark}
Notice that when $c\ge1$, the specialization $\Beta((c/\theta)^M)$ is not stable in the sense of \cref{def:stable}. However, stability is not needed here: stability was introduced simply as a convenient condition that implies HT-appropriateness (see~\cref{lem:check_ht_appropriateness}), while in this case, we have checked directly that the sequence of pure-beta Jack measures is HT-appropriate.
\end{remark}

\begin{theorem}\label{thm:application2}
Let $\gamma,c>0$ and $M\in\Z_{\ge 1}$ be fixed.
Let $\{\theta_N>0\}_{N\ge 1}$ be such that $N\theta_N\to\gamma$, as $N\to\infty$.
Let $\big\{\la^{(t)}(N)\big\}_{t=0,1,2,\dots}$ be the Markov chain from \cref{def:Markov_chain} with $\rho_N = \Beta(c/\theta_N)$ and initial laws of $\la^{(0)}(N)$ given by $\PP_N$.

\smallskip

(1) If $\PP_N(\la)=\delta_{\la,\emptyset}$, then the law of $\la^{(M)}(N)$ is the pure-beta Jack measure $\PP^{(\theta_N)}_{\Beta((c/\theta_N)^M);\,N}$.
Moreover, these Jack measures satisfy the LLN and the corresponding empirical measures converge weakly, in probability, to the unique probability measure $\mu_{\Beta}^{\gamma;c;M}$ with moments
\begin{multline*}
\int_\R{ x^\ell \mu_{\Beta}^{\gamma;c;M}(\dd\!x) } = 
\sum_{\Gamma\in\Luk(\ell)}
(Mb-M)^{\text{\#\,up steps}}\, (-b)^{\text{\#\,down steps}}\ \frac{\Delta_\gamma\left(x^{1+\text{\#\,horizontal steps at height $0$}}\right)(Mb)}{1+\text{\#\,horizontal steps at height $0$}}\\
\cdot \prod_{j\ge 1}(j+Mb)^{\text{\#\,horizontal steps at height $j$}}
(j+\gamma)^{\text{\#\,down steps from height $j$}},\quad\text{ for all }\ell\in\Z_{\ge 1},
\end{multline*}
where we set $b:=c/(1+c)\in(0,1)$, for convenience.

\smallskip

(2) In the general case, when $\PP_N$ are such that the empirical measures $\mu_N^{(0)}$ of $\la^{(0)}(N)$ converge to $\mu^{(0)}$, as in~\eqref{eq:assumption_LLN}, then the empirical measures of $\la^{(M)}(N)$ converge weakly, in probability, to
\[
\lim_{\substack{N\to\infty\\N\theta\to\gamma}}{\mu_N^{(M)}} = \mu^{(0)} \boxplus^{(\gamma)} \mu_{\Beta}^{\gamma;c;M}.
\]
\end{theorem}
\begin{proof}
The fact that the law of $\la^{(M)}(N)$ is $\PP^{(\theta_N)}_{\Beta((c/\theta_N)^M);\,N}$ if $\PP_N(\la)=\delta_{\la,\emptyset}$ follows from \cref{lem:empty_initial}. By \eqref{eqn:G_beta_product}-\eqref{eqn:first_condition_HT_example2}, this sequence of Jack measures is HT-appropriate with
\[
\kappa_n=(-1)^{n-1}Mb^n,\qquad b=\frac{c}{1+c}.
\]
Since $|\kappa_n|=Mb^n\le (Mb)^n$, for all $n\ge1$, \cref{theo:main1} proves the LLN, uniqueness of the limiting measure and weak convergence in part~(1).

For part~(2), note that~\eqref{eqn:Gt} gives
\[
G^{(M)}_{\Beta(c/\theta_N);N,\theta_N}
=
G_{\Beta((c/\theta_N)^M);N,\theta_N}\,G_{\PP_N,\theta_N}.
\]
By \eqref{eq:assumption_LLN} and the common compact-support assumption, the initial laws satisfy the LLN and hence are HT-appropriate by \cref{theo:main1}.
Taking logarithms in the preceding product and using~\eqref{eqn:first_condition_HT_example2}, we conclude that the laws at time $M$ are HT-appropriate with quantized $\gamma$-cumulants
\[
\kappa_n[\mu^{(0)}]+(-1)^{n-1}Mb^n.
\]
The proof of part~(2) of \cref{theo:LLN_Markov} shows that the compact-support assumption on $\mu^{(0)}$ gives an exponential bound on $\kappa_n[\mu^{(0)}]$.
Therefore the cumulants above satisfy the growth condition from \cref{theo:main1}, which gives uniqueness of the limiting measure and weak convergence.
This proves part~(2).
\end{proof}

\begin{corollary}\label{cor:beta}
In the setting of \cref{thm:application2}, if $\gamma=Mc/(1+c)$, or equivalently, if $\gamma<M$ and $c=\gamma/(M-\gamma)$, then the moments of the measure $\mu_{\Beta}^{\gamma;c;M}=\mu_{\Beta}^{\gamma;\frac{\gamma}{M-\gamma};M}$ can be expressed as
\begin{multline}\label{eqn:beta}
\int_\R{ x^\ell \mu_{\Beta}^{\gamma;\frac{\gamma}{M-\gamma};M}(\dd\!x) } = 
\sum_{\Gamma\in\Luk(\ell)}
(-1)^{\text{\#\,up steps} \ +\ \text{\#\,down steps}}\,
\frac{(M-\gamma)^{\text{\#\,up steps}} \, M^{-\text{\#\,down steps}}}{1+\text{\#\,horizontal steps at height $0$}}\\
\cdot\gamma^{\text{\#\,down steps\ +\ \#\,horizontal steps at height $0$}} \, \prod_{j\ge 1} (j+\gamma)^{\text{\#\,horizontal steps at height $j$}\ +\ \text{\#\,down steps from height $j$}}.
\end{multline}
\end{corollary}
\begin{proof}
According to the proof of \cref{thm:application2}, we have $\kappa_1=\gamma$ iff $\gamma<M$ and $c=\gamma/(M-\gamma)$. As a result, \cref{lem:special_case} gives the desired formula.
\end{proof}

\subsubsection{Pure-Plancherel Jack measure}\label{sec:pure_plancherel}

Consider the Plancherel specialization $\Planch(t/\theta)$, given by $p_k\big( \Planch(t/\theta) \big) = \frac{t}{\theta}\cdot\delta_{k, 1}$, for all $k\ge 1$.
As the coefficient of $p_1^{|\la|}$ in $P_\la(;\theta)$ is equal to $\prod_{(i,j)\in\la}\frac{\theta}{(\la_i-j)+\theta(\la_j'-i)+\theta}$ (see~\cite[Ch.~VI, (10.27),(10.29)]{Macdonald1995}), it follows that
\begin{equation}\label{eq:planch_1}
P_\la\big( \Planch(t/\theta); \theta \big) = \prod_{(i,j)\in\la}\frac{t}{(\la_i-j)+\theta(\la_j'-i)+\theta},\quad\text{for all }\la\in\Y.
\end{equation}

The normalization constant corresponding to the Jack measure $\PP^{(\theta)}_{\Planch(t/\theta);\,N}$ in~\eqref{special_jack_measure} is
\begin{equation}\label{eq:planch_2}
\exp\Bigg\{ N\theta\sum_{k\ge 1}{\frac{p_k\big( \Planch(t/\theta) \big)}{k}} \Bigg\} = \exp(Nt).
\end{equation}

The identities \eqref{eq:planch_1}--\eqref{eq:planch_2} turn the Jack measure~\eqref{special_jack_measure} into the following explicit formula.

\begin{definition}\label{def:example3}
Let $t,\theta>0$ and $N\in\Z_{\ge 1}$.
Then the \textbf{pure-Plancherel Jack measure} associated to these parameters is the probability measure on partitions of length $\le N$, given by
\begin{equation}\label{eq:formal_pure_plancherel}
\PP^{(\theta)}_{\Planch(t/\theta);\,N}(\la) = e^{-Nt} t^{|\la|}
\prod_{(i,j)\in\la}{ \frac{N\theta + (j-1) - \theta(i-1)}{((\la_i-j) + \theta(\la_j'-i) + \theta)((\la_i-j) + \theta(\la_j'-i) + 1)} },
\end{equation}
for all $\la\in\Y(N)$.
\end{definition}

We show that $\big\{\PP^{(\theta)}_{\Planch(t/\theta);\,N}\big\}_{N\ge 1}$ is HT-appropriate in the sense of \cref{def:appropriate}.
The second condition is automatic; as for the first, we need~\eqref{eqn:first_condition_HT} and the following calculation:
\begin{equation}\label{eqn:first_condition_HT_example3}
\lim_{\substack{N\to\infty\\N\theta\to\gamma}} { \frac{\partial_1^n\, F_{\Planch(t/\theta);\,N,\theta}}{(n-1)!}\bigg|_{(x_1,\dots,x_N) = (1^N)} }
= \frac{1}{(n-1)!}\cdot \lim_{\substack{N\to\infty\\N\theta\to\gamma}} { \theta\cdot\partial_x^n \left( \frac{t}{\theta}\, x \right) \Big|_{x=1} }
= t\cdot\delta_{n,1}.
\end{equation}

\begin{theorem}\label{thm:application3}
Let $\gamma>0$ be fixed.
Let $\{\theta_N>0\}_{N\ge 1}$ be such that $N\theta_N\to\gamma$, as $N\to\infty$.
Let $\big\{\la^{(t)}(N)\big\}_{t=0,1,2,\dots}$ be the Markov chain from \cref{def:Markov_chain} with $\rho_N = \Planch(1/\theta_N)$ and initial laws of $\la^{(0)}(N)$ given by $\PP_N$.
For any $t\in\Z_{\ge1}$, the following statements hold.

\smallskip

(1) If $\PP_N(\la)=\delta_{\la,\emptyset}$, then the law of $\la^{(t)}(N)$ is the pure-Plancherel Jack measure $\PP^{(\theta_N)}_{\Planch(t/\theta_N);\,N}$.
Moreover, these Jack measures satisfy the LLN and the corresponding empirical measures converge weakly, in probability, to the unique probability measure $\mu_{\Planch}^{\gamma;t}$ with moments
\begin{multline}\label{eq:Moments_example3}
\int_\R{ x^\ell \mu_{\Planch}^{\gamma;t}(\dd\!x) } =
\sum_{\Gamma\in\mathbf{M}(\ell)}
t^{\text{\#\,up steps}}\ \frac{\Delta_\gamma\left(x^{1+\text{\#\,horizontal steps at height $0$}}\right)(t)}{1+\text{\#\,horizontal steps at height $0$}}\\
\cdot\prod_{j\ge 1}(j+t)^{\text{\#\,horizontal steps at height $j$}}\,
(j+\gamma)^{\text{\#\,down steps from height $j$}},\quad\text{for all }\ell\in\Z_{\ge 1}.
\end{multline}
In this formula, $\mathbf{M}(\ell)$ is the set of all \textbf{Motzkin paths of length $\ell$}; by definition, they are the Łukasiewicz paths of length $\ell$ whose up steps are all $(1,1)$.%
\footnote{In other words, Motzkin paths of length $\ell$, by definition, are the lattice paths with steps $(1,1), (1,0), (1,-1)$, that stay above the $x$-axis, begin at $(0,0)$ and end at $(\ell,0)$.}

\smallskip

(2) In the general case, when $\PP_N$ are such that the empirical measures $\mu_N^{(0)}$ of the laws of $\la^{(0)}(N)$ converge to $\mu^{(0)}$, as in~\eqref{eq:assumption_LLN}, then the empirical measures of $\la^{(t)}(N)$ converge weakly, in probability, to
\[
\lim_{\substack{N\to\infty\\N\theta\to\gamma}}{\mu_N^{(t)}} = \mu^{(0)} \boxplus^{(\gamma)} \mu_{\Planch}^{\gamma;t}.
\]
\end{theorem}
\begin{proof}
The law of $\la^{(t)}(N)$ is $\PP^{(\theta_N)}_{\Planch(t/\theta_N);\,N}$ if $\PP_N(\la)=\delta_{\la,\emptyset}$, because of \cref{lem:empty_initial}.
By \eqref{eqn:first_condition_HT}, \eqref{eqn:first_condition_HT_example3} and \cref{lem:check_ht_appropriateness}, the sequence $\Big\{ \PP^{(\theta_N)}_{\Planch(t/\theta_N);\,N} \Big\}_{N\ge 1}$ is HT-appropriate with $\kappa_n = t\cdot\delta_{n,1}$.
Therefore the first statement follows from \cref{theo:LLN_Markov}.
Since $|\kappa_n|=t\cdot\delta_{n,1}\le t^n$, for all $n\ge 1$, \cref{theo:main1} proves part (1).
Part (2) follows from \cref{theo:LLN_Markov}.
\end{proof}

\begin{remark}
Note that $t$ in $\PP^{(\theta_N)}_{\Planch(t/\theta_N);\,N}$ does not necessarily have to be an integer, but it can also be a positive real number.
In this case, we obtain the LLN in the high temperature regime for the fixed-time distribution of the continuous-time Markov chain of Gorin--Shkolnikov mentioned in \cref{rem:GS}.
\end{remark}

\begin{corollary}
In the setting of \cref{thm:application3}, if $t=\gamma$, the moments of $\mu_{\Planch}^{\gamma;\gamma}$ are equal to
\begin{multline*}
\int_\R{ x^\ell \mu_{\Planch}^{\gamma;\gamma}(\dd\!x) } = 
\sum_{\Gamma\in\mathbf{M}(\ell)}
\frac{\gamma^{\text{\#\,up steps} \,\,+\,\, \text{\#\,horizontal steps at height $0$}}}{1+\text{\#\,horizontal steps at height $0$}}\\
\cdot\prod_{j\ge 1}(j+\gamma)^{\text{\#\,horizontal steps at height $j$} \,\,+\,\, \text{\#\,down steps from height $j$}}.
\end{multline*}
\end{corollary}
\begin{proof}
By the proof of \cref{thm:application3}, $\kappa_1=\gamma$ iff $t=\gamma$; then apply \cref{lem:special_case}.
\end{proof}

\section{Comparison with previous related results}\label{sec:comparison}

Here, we discuss the relations between our results and related ones from the literature.

Even though the main topic of this paper is the asymptotic behavior of measures in the high temperature regime $N\to\infty, N\theta\to\gamma$, one can prove the analogous results in the fixed temperature regime, i.e. when $N\to\infty$ and $\theta$ is fixed, by the same arguments.
As the proofs are the same, we will only sketch the arguments and leave the details to the interested reader.
In the case of a fixed $\theta>0$, the correct scaling of the empirical measure $\mu_N$ of
$\Prob_N$ is as follows:
\begin{equation}\label{eq:FixedEmpirical}
\tilde{\mu}_N := \frac{1}{N}\sum_{i=1}^N\delta_{\frac{\LL_i}{N\theta}},
\end{equation}
where $\LL_i := \la_i - (i-1)\theta$, for $i=1,2,\dots,N$.

\begin{definition}\label{def:J_infty}
Given a sequence $\vec{\kappa}=(\kappa_n)_{n\ge 1}$, define $\vec{m}=(m_n)_{n\ge 1}$ by
\begin{multline}\label{eq:MomentsInf}
m_\ell :=  \sum_{\Gamma\in\Luk(\ell)}
\frac{\kappa_1^{1+\text{\#\,horizontal steps at height $0$ in $\Gamma$}}- (\kappa_1-1)^{1+\text{\#\,horizontal steps at height $0$ in $\Gamma$}} }{1+\text{\#\,horizontal steps at height $0$ in $\Gamma$}}\\
\cdot\kappa_1^{\text{\#\,horizontal steps at height larger than $0$  in $\Gamma$}}\cdot\prod_{j\ge 1}(\kappa_j+\kappa_{j+1})^{\text{\#\,steps $(1,j)$ in $\Gamma$}},\quad\text{for all $\ell\in\Z_{\ge 1}$}.
\end{multline}
We denote the map $\vec{\kappa}\mapsto\vec{m}$ by $\mathcal{J}_\infty^{\kappa\mapsto m}\colon\R^\infty\to\R^\infty$ and the compositional inverse $\vec{m}\mapsto\vec{\kappa}$ by $\mathcal{J}_\infty^{m\mapsto\kappa} := (\mathcal{J}_\infty^{\kappa\mapsto m})^{(-1)}$.
\end{definition}

Then we have the following result.

\begin{theorem}\label{thm:Fixed_1}
Let $\theta>0$ be fixed.
Let $\{\PP_N\}_{N\ge 1}$ be a sequence of probability measures with total mass $1$ on
$N$-signatures and with small tails, and let $\{\tilde{\mu}_N\}_{N\ge 1}$ be their
empirical measures, given by~\eqref{eq:FixedEmpirical}.
Also, recall that $G_{\PP_N,\theta}(x_1,\dots,x_N)$ denotes the Jack generating function of $\PP_N$.
Suppose that the following two conditions are satisfied, for some constants $\kappa_1^\infty,\kappa_2^\infty,\cdots$:
\vspace{.1cm}

\begin{enumerate}

\item $\displaystyle\lim_{N\to\infty} \frac{1}{N\theta}\frac{\partial_1^n}{(n-1)!} \ln\big( G_{\Prob_N,\theta}\big) \big|_{(x_1,\dots,x_N) = (1^N)} = \kappa_n^\infty$ exists and is finite, for all $n \in \Z_{\geq 1}$,

\item $\displaystyle\lim_{N\to\infty} \frac{1}{N\theta}\partial_{i_1}\cdots \partial_{i_r} \ln\big( G_{\Prob_N,\theta}\big) \big|_{(x_1,\dots,x_N) = (1^N)} = 0$, for all $r\ge 2$ and $i_1,\dots,i_r\in\Z_{\geq 1}$ with at least two distinct indices among $i_1,\dots,i_r$.
\end{enumerate}
Then $\tilde{\mu}_N$ converges, as $N\to\infty$, in the sense of moments, in probability, to a probability measure $\mu_\infty$ with finite moments $m_n^\infty$.
Moreover, the sequences $\vec{\kappa}=(\kappa^\infty_n)_{n\ge 1}$ and $\vec{m}=(m^\infty_n)_{n\ge 1}$ are related to each other by $\vec{m}=\mathcal{J}_\infty^{\kappa\mapsto m}(\vec{\kappa})$.
Finally, if $\sup_{n\ge 1}\big|\kappa_n^\infty\big|^{1/n}<\infty$, then $\mu_\infty$ is uniquely determined by its moments.
\end{theorem}

This theorem can also be generalized to the case where $\PP_N$ are finite signed measures, as in \cref{theo:main1}.

\begin{example}\label{exam:formulas_infty}
The first few moments $m_n^\infty$ in terms of $\kappa_n^\infty$ are
\begin{align*}
m_1^\infty &= \kappa_1^\infty - \frac{1}{2},\\
m_2^\infty &= \kappa_2^\infty + \big(\kappa_1^\infty\big)^2+\frac{1}{3},\\
m_3^\infty &= \kappa_3^\infty + 3\kappa_2^\infty\kappa_1^\infty + \big(\kappa_1^\infty\big)^3 + \frac{3}{2}\big(\kappa_1^\infty\big)^2 - \frac{1}{4}.
\end{align*}
\end{example}

The fixed temperature regime was studied previously by Bufetov--Gorin~\cite{BufetovGorin2015b} for $\theta=1$, and by Huang~\cite{Huang2021} for arbitrary $\theta >0$.
Bufetov--Gorin employed certain differential operators on the space of symmetric polynomials that were designed from the explicit determinantal formula for Schur polynomials.
As a result, there is no obvious extension of their ideas to our setting, since Jack polynomials do not have determinantal formulas; moreover, the Cherednik operators that we use are differential-reflection operators.
On the other hand, the Jack generating function used in the main part of Huang's work is a formal power series of symmetric functions, rather than symmetric polynomials.
In \cite[Appendix B]{Huang2021}, Huang also introduces an $N$-finite Jack generating function in $N$ variables, which is closer to the one used here.
The assumptions in the main formulation of~\cite{Huang2021} are more complicated and technical than ours and seem more difficult to verify in practice; this is a consequence of the fact that the techniques in that paper are different and rely on the analysis of the Nazarov--Sklyanin operators~\cite{NazarovSklyanin2013} that do not directly reflect the behavior of the model.%
\footnote{
Huang's technical assumptions involve infinite sums of differential operators with respect to the power-sum symmetric functions. They were derived from the explicit form of Nazarov--Sklyanin operators, which are differential operators in these power-sum variables. These operators are very useful in studying other models of discrete random $\beta$-ensembles, as shown in~\cite{Moll2023,CuencaDolegaMoll2026}, but the model studied here is intrinsically related to the monomial or Taylor expansions of the logarithms of Jack generating functions, and therefore Cherednik operators are much better suited to its analysis.
}
Appendix B of~\cite{Huang2021} shows that these assumptions can be reformulated as conditions on the derivatives of the logarithm of the $N$-finite Jack generating function at $(1^N)$.
After taking into account our normalization by $N\theta$, with $\theta$ fixed, these conditions agree with the ones in \cref{thm:Fixed_1}; in particular, this last theorem recovers the LLN proved in~\cite{Huang2021}.

As a sanity check, let us next deduce the formula obtained in \cite[Theorem~2.4]{BufetovGorin2018}, \cite[Eqn.~(2.7) in Theorem~2.6]{BufetovGorin2019} and \cite[Theorem~2.4]{Huang2021} (expressed by extracting coefficients of certain generating functions) from our combinatorial formula~\eqref{eq:MomentsInf}.

\begin{proposition}
Let $\vec{\kappa}=(\kappa_n)_{n\ge 1}$ be arbitrary, let $R(z) := \sum_{n\ge 1}{\kappa_n z^{n-1}}$ and $(m_n)_{n\ge 1}:=\mathcal{J}_\infty^{\kappa\mapsto m}(\vec{\kappa})$.
Then, for any $n\in\Z_{\ge 1}$, we have
\begin{equation}\label{eq:MomentHuang}
m_n = \frac{1}{n+1}[z^{-1}]\frac{1}{1+z}\left(z^{-1}+(1+z)R(z)\right)^{n+1}.
\end{equation}
\end{proposition}

\begin{proof}
Let $F(z) = 1+\sum_{i \geq 1}f_i z^i$ be a formal power series. It is a classical result in enumerative combinatorics (see e.g.~\cite[Sections 5.3--5.4]{Stanley:EC2}) that the generating function for Łukasiewicz paths
\[ L(t) := 1 + \sum_{n\ge 1}\sum_{\Gamma\in\Luk(n)}
t^n \prod_{j\ge 0}(f_{j+1})^{\text{\#\,steps $(1,j)$ in $\Gamma$}}\]
satisfies the following functional equation:
\[ t\cdot L(t) = t\cdot F(t\cdot L(t)).\]
In particular, it follows from the Lagrange inversion formula (see~\cite[Theorem 5.4.2]{Stanley:EC2}) that
\begin{multline}\label{eqn:from_lagrange} 
\frac{1}{n+1}[z^{-1}]\frac{1}{1+z}\left(z^{-1}\cdot F(z)\right)^{n+1} = \sum_{i= 0}^{n}\frac{(-1)^{i}}{i+1}\cdot\frac{i+1}{n+1}[z^{n+1-(i+1)}]\left(F(z)\right)^{n+1} \\
= [t^{n+1}]\sum_{i= 0}^{n}\frac{(-1)^{i}}{i+1}\left(t \cdot L(t)\right)^{i+1}.
\end{multline}
Note that $\left(t \cdot L(t)\right)^{i+1}$ can be interpreted as the generating series for the sequences of $i+1$ Łukasiewicz paths, where each of them starts with a horizontal step at height $0$. By concatenating these paths and then removing the first horizontal step at height $0$, we obtain a Łukasiewicz path of length $n$ (if looking at the coefficient of $t^{n+1}$) with $i$ marked horizontal steps at height $0$. Therefore, we can rewrite the RHS of~\eqref{eqn:from_lagrange} as a sum over all Łukasiewicz paths of length $n$ with some number of marked horizontal steps at height $0$ (we can mark $i$ steps in $\binom{h_0(\Gamma)}{i}$ ways, where $h_0(\Gamma)$ is the number of horizontal steps at height $0$ in $\Gamma$), and then assigning the corresponding weight:
\[ \sum_{\Gamma\in\Luk(n)}\sum_{i=0}^{h_0(\Gamma)}\frac{(-1)^{i}}{i+1}\binom{h_0(\Gamma)}{i}f_1^{(\text{\#\,steps $(1,0)$ in $\Gamma$})\,-\,i}\prod_{j\ge 1}(f_{j+1})^{\text{\#\,steps $(1,j)$ in $\Gamma$}}.\]
By choosing $F(z)=1+z(1+z)R(z)=1+\kappa_1z+\sum_{j\ge 1}{(\kappa_j+\kappa_{j+1})z^{j+1}}$, and applying the identity
$\frac{1}{i+1}\binom{h_0(\Gamma)}{i} = \frac{1}{1+h_0(\Gamma)}\binom{1+h_0(\Gamma)}{i+1}$ in the above expression, we can express the RHS of~\eqref{eq:MomentHuang} as
\[
\sum_{\Gamma\in\Luk(n)}\frac{1}{1+h_0(\Gamma)}\sum_{i=0}^{h_0(\Gamma)} (-1)^i\binom{1+h_0(\Gamma)}{i+1}\kappa_1^{\text{(\#\,steps $(1,0)$ in $\Gamma$}) \,-\, i}\prod_{j\ge 1}(\kappa_j+\kappa_{j+1})^{\text{\#\,steps $(1,j)$ in $\Gamma$}}.
\]
This latter formula is clearly equal to the expression from \cref{def:J_infty}, due to the binomial formula, thus the proof is completed.
\end{proof}

\begin{remark}
The formula in \cite[Theorem~2.4]{BufetovGorin2018} and \cite[Eqn.~(2.7)]{BufetovGorin2019} is actually:
\begin{equation}\label{eq:bg_formula}
m_n = \frac{1}{n+1}[z^{-1}]\frac{1}{1+z}\left(z^{-1}+1+(1+z)R(z)\right)^{n+1}.
\end{equation}
Note that \cref{eq:bg_formula} differs slightly from \cref{eq:MomentHuang}, but this difference is a result of the simple fact that (when $\theta=1$) our empirical measures are $\frac{1}{N}\sum_{i=1}^N\delta\big(\frac{\la_i+1-i}{N}\big)$, while the empirical measures used in the previous references are $\frac{1}{N}\sum_{i=1}^N\delta\big(\frac{\la_i+N-i}{N}\big)$.
\end{remark}

\begin{remark}
As discussed in \cref{sec:app_1}, the quantized $\gamma$-cumulants linearize the quantized $\gamma$-convolution.
The limits of the quantized $\gamma$-cumulants when $\gamma\to\infty$, denoted $\kappa_n^\infty$, should therefore linearize the quantized convolution from \cite{BufetovGorin2015b}. However, the preceding remark shows that, when $\theta=1$, our convention for empirical measures differs from the one in~\cite{BufetovGorin2015b} by a translation by $-1$.
Therefore, the limits of the quantized $\gamma$-cumulants when $\gamma\to\infty$, denoted $\kappa_n^\infty$, should be compared with the quantized $R$-transform from~\cite{BufetovGorin2015b} after the corresponding translation.

More precisely, let $\mathfrak{m}$ be a probability measure with finite moments $m_n^\infty$ and let
\[
(\kappa_n^\infty)_{n\ge 1}
:=
\mathcal{J}_\infty^{m\mapsto\kappa}\big((m_n^\infty)_{n\ge 1}\big).
\]
For $a\in\R$, let $\tau_a\mathfrak{m}$ denote the pushforward of $\mathfrak{m}$ under $x\mapsto x+a$.
If $R_{\mathfrak m}^{quant,\mathrm{BG}}$ denotes the quantized $R$-transform from~\cite{BufetovGorin2015b}, define the transform adapted to our convention by
\[
\widetilde{R}_{\mathfrak m}^{quant}(z)
:=
R_{\tau_1\mathfrak m}^{quant,\mathrm{BG}}(z).
\]
By~\cite[(2.4)--(2.5)]{BufetovGorin2015b} and the identity
$R_{\tau_1\mathfrak m}(z)=R_{\mathfrak m}(z)+1$, we have
\[
\widetilde{R}_{\mathfrak m}^{quant}(z)
=
R_{\mathfrak m}(z)+1+z^{-1}-(1-e^{-z})^{-1}
=
R_{\mathfrak m}(z)+z^{-1}-(e^z-1)^{-1}.
\]
Following~\cite[Theorem~2.9]{BufetovGorin2015b}, a quantized $R$-transform $\widetilde{R}_{\mathfrak m}^{quant}(z)$ adapted to our convention linearizes the $\gamma=\infty$ quantized convolution. It turns out that the coefficients of that transform are not exactly $\kappa_n^\infty$, but linear combinations of them. The above formula and \cref{exam:formulas_infty} give
\[
\widetilde{R}_{\mathfrak{m}}^{quant}(z)
=
\kappa_1^\infty
+
(\kappa_2^\infty+\kappa_1^\infty)z
+
\left(
\kappa_3^\infty+\frac{3}{2}\kappa_2^\infty+\frac{1}{2}\kappa_1^\infty
\right)z^2
+O(z^3),
\]
so even the first few coefficients of the quantized $R$-transform adapted to our convention are already nontrivial linear combinations of $\kappa_n^\infty$.
\end{remark}

\smallskip

We claim that the map $\mathcal{J}_\infty^{\kappa\mapsto m}$ from \cref{def:J_infty} can be regarded as the limit of $\mathcal{J}_\gamma^{\kappa\mapsto m}$, as $\gamma\to\infty$.
Indeed, in the high temperature regime $N\to\infty$, $N\theta\to\gamma$, the moments $m^N_\ell$ of the empirical measure $\mu_N$ have limits $m_\ell$ if and only if the moments $\tilde{m}^N_\ell$ of the empirical measure $\tilde{\mu}_N$ have limits $\tilde{m}_\ell$.
In that case, these limits are related to each other by $m_\ell = \gamma^\ell \cdot \tilde{m}_\ell$.
Similarly, the relation between the $\kappa_\ell$'s that appear in \cref{def:appropriate} and the $\tilde{\kappa}_\ell$'s from \cref{thm:Fixed_1} is given by $\kappa_\ell = \gamma\cdot\tilde{\kappa}_\ell$.
Therefore, the formula \eqref{eq:Moments}, rewritten in terms of $\tilde{\kappa}_\ell$ and $\tilde{m}_\ell$, reads
\begin{multline}\label{eq:MomentsInfty}
\tilde{m}_\ell =  \sum_{\Gamma\in\Luk(\ell)}
\frac{\tilde{\kappa}_1^{1+h_0(\Gamma)} - (\tilde{\kappa}_1-1)^{1+h_0(\Gamma)}}{1+h_0(\Gamma)}\cdot\prod_{i\ge 1}(\tilde{\kappa}_1+i/\gamma)^{\text{\#\,horizontal steps at height $i$ in $\Gamma$}}\\
\cdot\prod_{j\ge 1}(\tilde{\kappa}_j+\tilde{\kappa}_{j+1})^{\text{\#\,steps $(1,j)$ in $\Gamma$}}(j/\gamma+1)^{\text{\#\,down steps from height $j$ in $\Gamma$}},\quad\text{for all $\ell\in\Z_{\ge 1}$},
\end{multline}
where $h_0(\Gamma) := \text{\#\,horizontal steps at height $0$ in $\Gamma$}$.
Evidently, the formula~\eqref{eq:MomentsInfty} in the limit $\gamma\to\infty$ is exactly the formula~\eqref{eq:MomentsInf} for the moments $m^\infty_\ell$, as claimed.
This limit can also be understood by introducing a certain gradation on the polynomial algebra $\C[\gamma,\kappa_1,\kappa_2,\dots]$.
If we define
\[
\deg_1(\gamma) = \deg_1(\kappa_1) = \deg_1(\kappa_2) = \dotsm = 1,
\]
then the combinatorial formula~\eqref{eq:Moments} for the moments $m_\ell$ implies that $m_\ell \in \C[\gamma,\kappa_1,\kappa_2,\dots]$ has $\deg_1(m_\ell)=\ell$, and $m_\ell^\infty$ is equal to the top degree term of $m_\ell$ followed by the substitutions $\gamma\mapsto 1$ and $\kappa_n\mapsto\kappa^\infty_n$.

\smallskip

It turns out that there is another gradation $\deg_2$ on the polynomial algebra $\C[\gamma,\kappa_1,\kappa_2,\dots]$ that allows us to deduce a recent result of~\cite{Benaych-GeorgesCuencaGorin2022} from our formula~\eqref{eq:Moments}, namely:
\[
\deg_2(\gamma) = 0, \quad \deg_2(\kappa_n) = n,\quad\forall\,n\ge 1.
\]
In order to motivate this gradation, recall that in the proof of \cref{theo:main1}, the moments $m_\ell$ were calculated as the limits
\[
\lim_{\substack{N\to\infty\\N\theta\to\gamma}}\frac{1}{N}\sum_{i=1}^N\big(\xi_i + \tilde{\xi}_i\big)^\ell \exp\left(\hat{F}_N\right)\bigg|_{(x_1,\dots,x_N)=(0^N)},
\]
where $\xi_i$ are the Cherednik operators, $\tilde{\xi}_i$ are the Dunkl operators, and $\hat{F}_N=\hat{F}_N(x_1,\dots,x_N)$ is a formal power series such that
\begin{enumerate}
\item $\displaystyle\lim_{\substack{N\to\infty\\N\theta\to\gamma}} \frac{\partial_1^n}{(n-1)!}\hat{F}_N \big|_{(x_1,\dots,x_N) = (0^N)} = \kappa_n$ exists and is finite, for all $n \in \Z_{\geq 1}$,

\item $\displaystyle\lim_{\substack{N\to\infty\\N\theta\to\gamma}} \partial_{i_1}\cdots \partial_{i_r}\hat{F}_N \big|_{(x_1,\dots,x_N) = (0^N)} = 0$, for all $r\ge 2$ and $i_1,\dots,i_r\in\Z_{\geq 1}$ with at least two distinct indices among $i_1,\dots,i_r$;
\end{enumerate}
see \cref{cherednik_dunkl}.
The Cherednik operators preserve the degree of power series, while the Dunkl operators decrease the degree by $1$, therefore by introducing a formal variable $t$, the limit
\[\lim_{\substack{N\to\infty\\N\theta\to\gamma}}\frac{1}{N}\sum_{i=1}^N \big(\xi_i+t\cdot\tilde{\xi}_i\big) ^{\ell} \exp\left(\hat{F}_N\right)\bigg|_{(x_1,\dots,x_N)=(0^N)},\]
is equal to the following $t$-deformation of the formula for $m_\ell$:
\begin{multline*}
\sum_{\Gamma\in\Luk(\ell)}
\frac{\Delta_\gamma\left(x^{1+\text{\#\,horizontal steps at height $0$ in $\Gamma$}}\right)(t\cdot\kappa_1)}{1+\text{\#\,horizontal steps at height $0$ in $\Gamma$}}
\cdot\prod_{i\ge 1}(t\cdot\kappa_1+i)^{\text{\#\,horizontal steps at height $i$ in $\Gamma$}}\\
\cdot\prod_{j\ge 1}(t^j\kappa_j+t^{j+1}\kappa_{j+1})^{\text{\#\,steps $(1,j)$ in $\Gamma$}}(j+\gamma)^{\text{\#\,down steps from height $j$ in $\Gamma$}}.
\end{multline*}
In this formula, the coefficient of $t^d$ has degree $\deg_2$ equal to the exponent $d$.
Define $m_\ell^{\topp}$ as the top homogeneous degree part of the RHS of \cref{eq:Moments} with respect to the gradation given by $\deg_2$.
From the formula above, it follows that
\begin{equation}\label{eq:MomentsTop}
m^{\topp}_\ell =  \sum_{\Gamma\in\Luk(\ell)}\prod_{j\ge 0}\kappa_{j+1}^{\text{\#\,steps $(1,j)$ in $\Gamma$}}(j+\gamma)^{\text{\#\,down steps from height $j$ in $\Gamma$}}.
\end{equation}
This formula corresponds exactly to the limit
\[
\lim_{\substack{N\to\infty\\N\theta\to\gamma}} \frac{1}{N} \sum_{i=1}^N \big(\tilde{\xi}_i\big)^{\ell} \exp\left(\hat{F}_N\right)\bigg|_{(x_1,\dots,x_N)=(0^N)},
\]
with the above assumptions on $\hat{F}_N$.
The calculation of such a limit was one of the main calculations in~\cite{Benaych-GeorgesCuencaGorin2022}, where the authors showed that it is the $\ell$-th moment of the limiting measure of a sequence of empirical measures of certain
continuous $\beta$-ensembles defined by the Bessel generating functions.
Their combinatorial formula was formulated in terms of set partitions, but is, in fact, equivalent to \cref{eq:MomentsTop}; indeed, \cite[Proposition 5.11]{Xu2026} formulates it in terms of noncrossing set partitions and it can be translated to Łukasiewicz paths by using the well-known bijection in \cite[Proposition 9.8]{NicaSpeicherBook}.
Our discrete model is a quantization of the continuous model from~\cite{Benaych-GeorgesCuencaGorin2022} (see the introduction of \cite{BufetovGorin2015b}), yet it is still quite interesting that we can recover their exact results as corollaries of our proofs.

\bibliographystyle{amsalpha}

\bibliography{biblio2015}

@preamble{"\def\cprime{$'$}"}

@article{Xu2026,
	author = {Xu, Jiaming},
	doi = {10.3842/SIGMA.2026.053},
	fjournal = {SIGMA. Symmetry, Integrability and Geometry. Methods and Applications},
	issn = {1815-0659},
	journal = {SIGMA Symmetry Integrability Geom. Methods Appl.},
	mrclass = {60B20 (05E05 60B10)},
	mrnumber = {5078446},
	pages = {Paper No. 053, 51},
	title = {Rectangular matrix additions in low and high temperatures},
	url = {https://doi.org/10.3842/SIGMA.2026.053},
	volume = {22},
	year = {2026}}

@article{SergeevVeselov2015,
	author = {Sergeev, A. N. and Veselov, A. P.},
	doi = {10.1112/plms/pdv023},
	fjournal = {Proceedings of the London Mathematical Society. Third Series},
	issn = {0024-6115,1460-244X},
	journal = {Proc. Lond. Math. Soc. (3)},
	mrclass = {05E05 (33C52 81R12)},
	mrnumber = {3404776},
	mrreviewer = {Ali\ Shojaei-Fard},
	number = {1},
	pages = {63--92},
	title = {Jack-{L}aurent symmetric functions},
	url = {https://doi.org/10.1112/plms/pdv023},
	volume = {111},
	year = {2015}}

@article{DunklGorin2025,
	author = {Dunkl, Charles and Gorin, Vadim},
	doi = {10.5802/alco.451},
	fjournal = {Algebraic Combinatorics},
	issn = {2589-5486},
	journal = {Algebr. Comb.},
	mrclass = {05E05},
	mrnumber = {5011473},
	number = {6},
	pages = {1547--1566},
	title = {Eigenvalues of {H}eckman-{P}olychronakos operators},
	url = {https://doi.org/10.5802/alco.451},
	volume = {8},
	year = {2025}}

@incollection{Biane1997FreeBrownianMotion,
	author = {Biane, Philippe},
	booktitle = {Free probability theory ({W}aterloo, {ON}, 1995)},
	doi = {10.1090/fic/012/01},
	isbn = {0-8218-0675-0},
	mrclass = {46L50 (60J65 81S25)},
	mrnumber = {1426833},
	mrreviewer = {David\ Applebaum},
	pages = {1--19},
	publisher = {Amer. Math. Soc., Providence, RI},
	series = {Fields Inst. Commun.},
	title = {Free {B}rownian motion, free stochastic calculus and random matrices},
	url = {https://doi.org/10.1090/fic/012/01},
	volume = {12},
	year = {1997}}

@article{Biane1997FreeConvolutionSemicircle,
	author = {Biane, Philippe},
	doi = {10.1512/iumj.1997.46.1467},
	fjournal = {Indiana University Mathematics Journal},
	issn = {0022-2518,1943-5258},
	journal = {Indiana Univ. Math. J.},
	mrclass = {46L50 (60E99)},
	mrnumber = {1488333},
	mrreviewer = {Andrzej\ \L uczak},
	number = {3},
	pages = {705--718},
	title = {On the free convolution with a semi-circular distribution},
	url = {https://doi.org/10.1512/iumj.1997.46.1467},
	volume = {46},
	year = {1997}}

@article{GuionnetMemin2022,
	author = {Guionnet, Alice and Memin, Ronan},
	doi = {10.1214/22-ejp771},
	fjournal = {Electronic Journal of Probability},
	issn = {1083-6489},
	journal = {Electron. J. Probab.},
	mrclass = {60B20 (60F10 60K35)},
	mrnumber = {4408125},
	pages = {Paper No. 46, 29},
	title = {Large deviations for {G}ibbs ensembles of the classical {T}oda chain},
	url = {https://doi.org/10.1214/22-ejp771},
	volume = {27},
	year = {2022}}

@article{Spohn2020,
	author = {Spohn, Herbert},
	doi = {10.1007/s10955-019-02320-5},
	fjournal = {Journal of Statistical Physics},
	issn = {0022-4715,1572-9613},
	journal = {J. Stat. Phys.},
	mrclass = {82C23 (81T32)},
	mrnumber = {4130980},
	mrreviewer = {V.\ I.\ Yukalov},
	number = {1-6},
	pages = {4--22},
	title = {Generalized {G}ibbs ensembles of the classical {T}oda chain},
	url = {https://doi.org/10.1007/s10955-019-02320-5},
	volume = {180},
	year = {2020}}

@article{CuencaDolegaMoll2026,
	author = {Cuenca, Cesar and Do{\l}{\k e}ga, Maciej and Moll, Alexander},
	doi = {10.1214/25-AOP1769},
	fjournal = {The Annals of Probability},
	issn = {0091-1798,2168-894X},
	journal = {Ann. Probab.},
	mrclass = {Prelim},
	mrnumber = {5019016},
	number = {1},
	pages = {421--488},
	title = {Universality of global asymptotics of {J}ack-deformed random {Y}oung diagrams at varying temperatures},
	url = {https://doi.org/10.1214/25-AOP1769},
	volume = {54},
	year = {2026}}

@article{Yao2021,
	author = {Yao, Andrew},
	note = {Preprint; arXiv:2109.14052},
	title = {Limits of Probability Measures With General Coefficients},
	year = {2021}}

@article{GorinShkolnikov2015,
	author = {Gorin, Vadim and Shkolnikov, Mykhaylo},
	doi = {10.1007/s00440-014-0596-2},
	fjournal = {Probability Theory and Related Fields},
	issn = {0178-8051},
	journal = {Probab. Theory Related Fields},
	mrclass = {60B20 (05E05 60H10 82C22)},
	mrnumber = {3418747},
	number = {3-4},
	pages = {413--463},
	title = {Multilevel {D}yson {B}rownian motions via {J}ack polynomials},
	url = {https://doi.org/10.1007/s00440-014-0596-2},
	volume = {163},
	year = {2015}}

@article{BufetovGorin2015b,
	author = {Bufetov, Alexey and Gorin, Vadim},
	doi = {10.1007/s00039-015-0323-x},
	fjournal = {Geometric and Functional Analysis},
	issn = {1016-443X},
	journal = {Geom. Funct. Anal.},
	mrclass = {17B20 (17B10 60B15)},
	mrnumber = {3361772},
	mrreviewer = {Martin V. Hildebrand},
	number = {3},
	pages = {763--814},
	title = {Representations of classical {L}ie groups and quantized free convolution},
	url = {https://doi.org/10.1007/s00039-015-0323-x},
	volume = {25},
	year = {2015}}

@article{GorinSun2022,
	author = {Gorin, Vadim and Sun, Yi},
	doi = {10.1353/ajm.2022.0006},
	fjournal = {American Journal of Mathematics},
	issn = {0002-9327},
	journal = {Amer. J. Math.},
	mrclass = {60B20 (15B52 33E20)},
	mrnumber = {4401507},
	number = {2},
	pages = {287--393},
	title = {Gaussian fluctuations for products of random matrices},
	url = {https://doi.org/10.1353/ajm.2022.0006},
	volume = {144},
	year = {2022}}

@article{BufetovKnizel2018,
	author = {Bufetov, Alexey and Knizel, Alisa},
	doi = {10.1214/17-AIHP838},
	fjournal = {Annales de l'Institut Henri Poincar\'{e} Probabilit\'{e}s et Statistiques},
	issn = {0246-0203},
	journal = {Ann. Inst. Henri Poincar\'{e} Probab. Stat.},
	mrclass = {60K35 (05B45 60C05 60F05 60F15)},
	mrnumber = {3825881},
	number = {3},
	pages = {1250--1290},
	title = {Asymptotics of random domino tilings of rectangular {A}ztec diamonds},
	url = {https://doi.org/10.1214/17-AIHP838},
	volume = {54},
	year = {2018}}

@article{BufetovGorin2019,
	author = {Bufetov, Alexey and Gorin, Vadim},
	doi = {10.1215/00127094-2019-0023},
	fjournal = {Duke Mathematical Journal},
	issn = {0012-7094},
	journal = {Duke Math. J.},
	mrclass = {60B15 (22E65 60K35)},
	mrnumber = {4007600},
	number = {13},
	pages = {2559--2649},
	title = {Fourier transform on high-dimensional unitary groups with applications to random tilings},
	url = {https://doi.org/10.1215/00127094-2019-0023},
	volume = {168},
	year = {2019}}

@article{LiQinWang2003,
	author = {Li, Wei-Ping and Qin, Zhenbo and Wang, Weiqiang},
	doi = {10.1515/crll.2003.007},
	fjournal = {Journal f\"{u}r die Reine und Angewandte Mathematik. [Crelle's Journal]},
	issn = {0075-4102},
	journal = {J. Reine Angew. Math.},
	mrclass = {14C05 (14F25)},
	mrnumber = {1952174},
	mrreviewer = {I. Dolgachev},
	pages = {217--234},
	title = {Stability of the cohomology rings of {H}ilbert schemes of points on surfaces},
	url = {https://doi.org/10.1515/crll.2003.007},
	volume = {554},
	year = {2003}}

@article{CollinsSniady2009,
	author = {Collins, Beno{\^\i}t and {\'S}niady, Piotr},
	note = {Preprint arXiv:0911.5546},
	title = {Asymptotic fluctuations of representations of the unitary groups},
	year = {2009}}

@article{TrinhTrinh2021,
	author = {Trinh, Hoang Dung and Trinh, Khanh Duy},
	doi = {10.1007/s10955-021-02832-z},
	fjournal = {Journal of Statistical Physics},
	issn = {0022-4715},
	journal = {J. Stat. Phys.},
	mrclass = {60F05 (42C05 60B20 60K35)},
	mrnumber = {4321572},
	mrreviewer = {Guangyu Yang},
	number = {1},
	pages = {Paper No. 4, 15},
	title = {Beta {J}acobi ensembles and associated {J}acobi polynomials},
	url = {https://doi.org/10.1007/s10955-021-02832-z},
	volume = {185},
	year = {2021}}

@article{PohlYoung2025,
	author = {Pohl, Kyla and Young, Ben},
	journal = {S\'{e}m. Lothar. Combin.},
	note = {Preprint; arXiv:2410.03912},
	pages = {Art. 7, 12},
	title = {{Jack combinatorics of the equivariant edge measure}},
	volume = {93B},
	year = {2025}}

@article{AllezBouchaudMajumdarVivo2013,
	author = {Allez, Romain and Bouchaud, Jean-Philippe and Majumdar, Satya N. and Vivo, Pierpaolo},
	doi = {10.1088/1751-8113/46/1/015001},
	fjournal = {Journal of Physics. A. Mathematical and Theoretical},
	issn = {1751-8113},
	journal = {J. Phys. A},
	mrclass = {60B20},
	mrnumber = {3001575},
	number = {1},
	pages = {015001, 22},
	title = {Invariant {$\beta$}-{W}ishart ensembles, crossover densities and asymptotic corrections to the {M}ar\v{c}enko-{P}astur law},
	url = {https://doi.org/10.1088/1751-8113/46/1/015001},
	volume = {46},
	year = {2013}}

@article{Wigner1958,
	author = {Wigner, Eugene P.},
	doi = {10.2307/1970008},
	fjournal = {Annals of Mathematics. Second Series},
	issn = {0003-486X},
	journal = {Ann. of Math. (2)},
	mrclass = {60.00 (15.00)},
	mrnumber = {95527},
	mrreviewer = {J. G. Wendel},
	pages = {325--327},
	title = {On the distribution of the roots of certain symmetric matrices},
	url = {https://doi.org/10.2307/1970008},
	volume = {67},
	year = {1958}}

@article{BufetovGorin2018,
	author = {Bufetov, Alexey and Gorin, Vadim},
	doi = {10.1016/j.aim.2018.07.009},
	fjournal = {Advances in Mathematics},
	issn = {0001-8708},
	journal = {Adv. Math.},
	mrclass = {60K35},
	mrnumber = {3861715},
	mrreviewer = {Anatoly Yambartsev},
	pages = {702--781},
	title = {Fluctuations of particle systems determined by {S}chur generating functions},
	url = {https://doi.org/10.1016/j.aim.2018.07.009},
	volume = {338},
	year = {2018}}

@article{GorinXuZhang2024,
	author = {Gorin, Vadim and Xu, Jiaming and Zhang, Lingfu},
	journal = {Preprint arXiv:2411.10829},
	title = {{Airy$_\beta$ line ensemble and its Laplace transform}},
	year = {2024}}

@article{KeatingXu2024,
	author = {Keating, David and Xu, Jiaming},
	journal = {Preprint arXiv:2411.12149},
	title = {{Airy limit for $\beta$-additions through Dunkl operators}},
	year = {2024}}

@article{Cherednik1991,
	author = {Cherednik, Ivan},
	doi = {10.1007/BF01243918},
	fjournal = {Inventiones Mathematicae},
	issn = {0020-9910},
	journal = {Invent. Math.},
	mrclass = {17B37 (16W30 17B20 20F36)},
	mrnumber = {1128220},
	mrreviewer = {V. A. Golubeva},
	number = {2},
	pages = {411--431},
	title = {A unification of {K}nizhnik-{Z}amolodchikov and {D}unkl operators via affine {H}ecke algebras},
	url = {https://doi.org/10.1007/BF01243918},
	volume = {106},
	year = {1991}}

@article{Opdam1995,
	author = {Opdam, Eric M.},
	doi = {10.1007/BF02392487},
	fjournal = {Acta Mathematica},
	issn = {0001-5962},
	journal = {Acta Math.},
	mrclass = {33C80 (17B99 33D80 43A30)},
	mrnumber = {1353018},
	mrreviewer = {V. A. Golubeva},
	number = {1},
	pages = {75--121},
	title = {Harmonic analysis for certain representations of graded {H}ecke algebras},
	url = {https://doi.org/10.1007/BF02392487},
	volume = {175},
	year = {1995}}

@unpublished{BenDaliDolega2023,
	archiveprefix = {arXiv},
	author = {Ben Dali, Houcine and Do{\l}\k{e}ga, Maciej},
	eprint = {2305.07966},
	journal = {arXiv:2305.07966},
	note = {Preprint arXiv:2305.07966},
	title = {{Positive formula for Jack polynomials, Jack characters and proof of Lassalle's conjecture}},
	year = {2023}}

@article{DolegaFeray2016,
	author = {Do{\l}{\k{e}}ga, Maciej and F\'eray, Valentin},
	doi = {10.1215/00127094-3449566},
	fjournal = {Duke Mathematical Journal},
	issn = {0012-7094},
	journal = {Duke Math. J.},
	mrclass = {60C05 (05E05 33C52 60B20)},
	mrnumber = {3498866},
	mrreviewer = {Sho Matsumoto},
	number = {7},
	pages = {1193--1282},
	title = {Gaussian fluctuations of {Y}oung diagrams and structure constants of {J}ack characters},
	url = {http://dx.doi.org/10.1215/00127094-3449566},
	volume = {165},
	year = {2016}}

@article{DolegaSniady2019,
	author = {Do{\l}{\k e}ga, Maciej and {\'S}niady, Piotr},
	doi = {10.1007/s00440-018-0854-9},
	fjournal = {Probability Theory and Related Fields},
	issn = {0178-8051},
	journal = {Probab. Theory Related Fields},
	key = {D{\'S}},
	mrclass = {05E05 (05E10 20C30 60B20 60K35)},
	mrnumber = {3947322},
	number = {1-2},
	pages = {133--176},
	title = {Gaussian fluctuations of {J}ack-deformed random {Y}oung diagrams},
	url = {https://mathscinet.ams.org/mathscinet-getitem?mr=3947322},
	volume = {174},
	year = {2019}}

@article{Huang2021,
	author = {Huang, Jiaoyang},
	doi = {10.1016/j.aim.2020.107545},
	fjournal = {Advances in Mathematics},
	issn = {0001-8708},
	journal = {Adv. Math.},
	mrclass = {05E05 (60F05 60K35 82C22)},
	mrnumber = {4200465},
	mrreviewer = {Sho Matsumoto},
	pages = {Paper No. 107545, 91},
	title = {Law of large numbers and central limit theorems through {J}ack generating functions},
	url = {https://doi.org/10.1016/j.aim.2020.107545},
	volume = {380},
	year = {2021}}

@article{HardyLambert2021,
	author = {Hardy, Adrien and Lambert, Gaultier},
	doi = {10.1016/j.jfa.2020.108869},
	fjournal = {Journal of Functional Analysis},
	issn = {0022-1236},
	journal = {J. Funct. Anal.},
	mrclass = {60F05 (47A10 60B20 82B05)},
	mrnumber = {4211032},
	mrreviewer = {Anamaria Savu},
	number = {7},
	pages = {Paper No. 108869, 40},
	title = {C{LT} for circular beta-ensembles at high temperature},
	url = {https://doi.org/10.1016/j.jfa.2020.108869},
	volume = {280},
	year = {2021}}

@article{DuyShirai2015,
	author = {Duy, Trinh Khanh and Shirai, Tomoyuki},
	doi = {10.1214/ECP.v20-4252},
	fjournal = {Electronic Communications in Probability},
	journal = {Electron. Commun. Probab.},
	mrclass = {60B20},
	mrnumber = {3407212},
	pages = {no. 68, 13},
	title = {The mean spectral measures of random {J}acobi matrices related to {G}aussian beta ensembles},
	url = {https://doi.org/10.1214/ECP.v20-4252},
	volume = {20},
	year = {2015}}

@article{AllezBouchaudGuionnet2012,
	author = {Allez, Romain and Bouchaud, Jean-Philippe and Guionnet, Alice},
	doi = {10.1103/PhysRevLett.109.094102},
	issue = {9},
	journal = {Phys. Rev. Lett.},
	month = {Aug},
	numpages = {5},
	pages = {094102},
	publisher = {American Physical Society},
	title = {{Invariant Beta Ensembles and the Gauss-Wigner Crossover}},
	url = {https://link.aps.org/doi/10.1103/PhysRevLett.109.094102},
	volume = {109},
	year = {2012}}

@article{Benaych-GeorgesCuencaGorin2022,
	author = {Benaych-Georges, Florent and Cuenca, Cesar and Gorin, Vadim},
	doi = {10.1007/s00220-022-04411-z},
	fjournal = {Communications in Mathematical Physics},
	issn = {0010-3616},
	journal = {Comm. Math. Phys.},
	mrclass = {60B20 (15B52 33C10 60B10 60E10 60F05 60G57)},
	mrnumber = {4469406},
	number = {2},
	pages = {735--795},
	title = {Matrix addition and the {D}unkl transform at high temperature},
	url = {https://doi.org/10.1007/s00220-022-04411-z},
	volume = {394},
	year = {2022}}

@incollection{NekrasovOkounkov2006,
	author = {Nekrasov, Nikita A. and Okounkov, Andrei},
	booktitle = {The unity of mathematics},
	doi = {10.1007/0-8176-4467-9\_15},
	mrclass = {81T60 (05E10 11Z05 14D21 60C05 81T45)},
	mrnumber = {2181816},
	mrreviewer = {Johan A. Martens},
	pages = {525--596},
	publisher = {Birkh\"{a}user Boston, Boston, MA},
	series = {Progr. Math.},
	title = {Seiberg-{W}itten theory and random partitions},
	url = {https://doi.org/10.1007/0-8176-4467-9_15},
	volume = {244},
	year = {2006}}

@article{BorodinOlshanski2005,
	author = {Borodin, Alexei and Olshanski, Grigori},
	doi = {10.1016/j.ejc.2004.06.003},
	fjournal = {European Journal of Combinatorics},
	issn = {0195-6698},
	journal = {European J. Combin.},
	mrclass = {60C05 (05E35 33C80 33C99 82B20 82B44)},
	mrnumber = {2143199},
	mrreviewer = {Bernhard Gittenberger},
	number = {6},
	pages = {795--834},
	title = {{$Z$}-measures on partitions and their scaling limits},
	url = {https://doi.org/10.1016/j.ejc.2004.06.003},
	volume = {26},
	year = {2005}}

@article{NazarovSklyanin2013,
	author = {Nazarov, Maxim and Sklyanin, Evgeny},
	doi = {10.3842/SIGMA.2013.078},
	fjournal = {SIGMA. Symmetry, Integrability and Geometry. Methods and Applications},
	journal = {SIGMA Symmetry Integrability Geom. Methods Appl.},
	mrclass = {81Q80 (05E05 33D52 37K10)},
	mrnumber = {3141546},
	mrreviewer = {Boris I. Zhilinskii},
	pages = {Paper 078, 14},
	title = {Integrable hierarchy of the quantum {B}enjamin-{O}no equation},
	url = {https://doi.org/10.3842/SIGMA.2013.078},
	volume = {9},
	year = {2013}}

@article{Dunkl1989,
	author = {Dunkl, Charles F.},
	doi = {10.2307/2001022},
	fjournal = {Transactions of the American Mathematical Society},
	issn = {0002-9947},
	journal = {Trans. Amer. Math. Soc.},
	mrclass = {33A45 (20H15 33A65 42C10 51F15)},
	mrnumber = {951883},
	mrreviewer = {W. A. Al-Salam},
	number = {1},
	pages = {167--183},
	title = {Differential-difference operators associated to reflection groups},
	url = {https://doi.org/10.2307/2001022},
	volume = {311},
	year = {1989}}

@article{Moll2023,
	author = {Moll, Alexander},
	doi = {10.1093/imrn/rnab300},
	journal = {Int. Math. Res. Not. IMRN},
	number = {3},
	pages = {1801--1881},
	title = {{Gaussian Asymptotics of Jack Measures on Partitions from Weighted Enumeration of Ribbon Paths}},
	year = {2023}}

@incollection{Heckman_1991,
	address = {Boston},
	author = {Heckman, Gerrit J.},
	booktitle = {Harmonic {A}nalysis on {R}eductive {G}roups},
	doi = {10.1007/978-1-4612-0455-8_8},
	pages = {181--191},
	publisher = {Birkh{\"a}user},
	title = {A remark on the {D}unkl differential--difference operators},
	volume = {101},
	year = {1991}}

@incollection{Viennot1985,
	author = {Viennot, G\'{e}rard},
	booktitle = {Orthogonal polynomials and applications ({B}ar-le-{D}uc, 1984)},
	doi = {10.1007/BFb0076539},
	mrclass = {42C05 (05A15 33A65)},
	mrnumber = {838979},
	mrreviewer = {R. A. Askey},
	pages = {139--157},
	publisher = {Springer, Berlin},
	series = {Lecture Notes in Math.},
	title = {A combinatorial theory for general orthogonal polynomials with extensions and applications},
	url = {https://doi.org/10.1007/BFb0076539},
	volume = {1171},
	year = {1985}}

@article{GuionnetHuang2019,
	author = {Guionnet, A. and Huang, J.},
	doi = {10.1002/cpa.21818},
	eprint = {https://onlinelibrary.wiley.com/doi/pdf/10.1002/cpa.21818},
	journal = {Communications on Pure and Applied Mathematics},
	number = {9},
	pages = {1875-1982},
	title = {{Rigidity and Edge Universality of Discrete $\beta$-Ensembles}},
	url = {https://onlinelibrary.wiley.com/doi/abs/10.1002/cpa.21818},
	volume = {72},
	year = {2019}}

@article{BaikDeiftJohansson1999,
	author = {Baik, J. and Deift, P. and Johansson, K.},
	fjournal = {Journal of the American Mathematical Society},
	issn = {0894-0347},
	journal = {J. Amer. Math. Soc.},
	mrclass = {05A05 (33D45 45E05 60C05)},
	mrnumber = {2000e:05006},
	mrreviewer = {David J. Aldous},
	number = {4},
	pages = {1119--1178},
	title = {{On the distribution of the length of the longest increasing subsequence of random permutations}},
	volume = {12},
	year = {1999}}

@article{Biane1995,
	author = {Biane, P.},
	coden = {KRMPBV},
	fjournal = {Kyoto University. Research Institute for Mathematical Sciences. Publications},
	issn = {0034-5318},
	journal = {Publ. Res. Inst. Math. Sci.},
	mrclass = {22E46},
	mrnumber = {MR1317523 (96c:22021)},
	mrreviewer = {Ken Dykema},
	number = {1},
	pages = {63--79},
	title = {{Representations of unitary groups and free convolution}},
	volume = {31},
	year = {1995}}

@article{BorodinGorinGuionnet2017,
	author = {Borodin, A. and Gorin, V. and Guionnet, A.},
	doi = {10.1007/s10240-016-0085-5},
	fjournal = {Publications Math\'ematiques. Institut de Hautes \'Etudes Scientifiques},
	issn = {0073-8301},
	journal = {Publ. Math. Inst. Hautes \'Etudes Sci.},
	mrclass = {60B10 (60B20 81R12)},
	mrnumber = {3668648},
	pages = {1--78},
	title = {{Gaussian asymptotics of discrete $\beta$-ensembles}},
	url = {http://dx.doi.org/10.1007/s10240-016-0085-5},
	volume = {125},
	year = {2017}}

@article{BorodinOkounkovOlshanski2000,
	author = {Borodin, Alexei and Okounkov, Andrei and Olshanski, Grigori},
	journal = {Journal of the American Mathematical Society},
	number = {3},
	pages = {481--515},
	title = {{Asymptotics of Plancherel measures for symmetric groups}},
	volume = {13},
	year = {2000}}

@article{Johansson2001,
	author = {Johansson, Kurt},
	coden = {ANMAAH},
	fjournal = {Annals of Mathematics. Second Series},
	issn = {0003-486X},
	journal = {Ann. of Math. (2)},
	mrclass = {05E10 (05A17 15A52 33C90 42C05 60C05 82B44)},
	mrnumber = {MR1826414 (2002g:05188)},
	mrreviewer = {Nicholas S. Witte},
	number = {1},
	pages = {259--296},
	title = {{Discrete orthogonal polynomial ensembles and the {P}lancherel measure}},
	volume = {153},
	year = {2001}}

@article{Kerov1993,
	author = {Kerov, Sergei},
	fjournal = {Rossi\u\i skaya Akademiya Nauk. Sankt-Peterburgskoe Otdelenie. Matematicheski\u\i\ Institut im. V. A. Steklova. Zapiski Nauchnykh Seminarov (POMI)},
	issn = {0373-2703},
	journal = {Zap. Nauchn. Sem. S.-Peterburg. Otdel. Mat. Inst. Steklov. (POMI)},
	mrclass = {05E10 (60B99 60C05)},
	mrnumber = {MR1255301},
	number = {Differentsialnaya Geom. Gruppy Li i Mekh. 13},
	pages = {21--29, 179},
	title = {{The asymptotics of interlacing sequences and the growth of continual {Y}oung diagrams}},
	volume = {205},
	year = {1993}}

@article{Kerov1993transition,
	author = {Kerov, Sergei},
	journal = {Funct. Anal. Appl.},
	number = {3},
	pages = {104--117},
	title = {{Transition probabilities of continual {Y}oung diagrams and the {M}arkov moment problem}},
	volume = {27},
	year = {1993}}

@article{Kerov2000,
	author = {Kerov, Sergei},
	journal = {Funct. Anal. Appl.},
	pages = {41--51},
	title = {{Anisotropic Young diagrams and Jack symmetric functions}},
	volume = {34},
	year = {2000}}

@article{KerovOkounkovOlshanski1998,
	author = {Kerov, Sergei and Okounkov, Andrei and Olshanski, Grigori},
	doi = {10.1155/S1073792898000154},
	fjournal = {International Mathematics Research Notices},
	issn = {1073-7928},
	journal = {Internat. Math. Res. Notices},
	mrclass = {05E10 (20C30 31C35 33C80)},
	mrnumber = {1609628},
	mrreviewer = {Tim H. Baker},
	number = {4},
	pages = {173--199},
	title = {The boundary of the {Y}oung graph with {J}ack edge multiplicities},
	url = {https://doi.org/10.1155/S1073792898000154},
	year = {1998}}

@article{KnopSahi1997,
	author = {Knop, Friedrich and Sahi, Siddhartha},
	doi = {10.1007/s002220050134},
	fjournal = {Inventiones Mathematicae},
	issn = {0020-9910},
	journal = {Invent. Math.},
	mrclass = {33D80 (05E05 39A10)},
	mrnumber = {1437493},
	mrreviewer = {Kenji Taniguchi},
	number = {1},
	pages = {9--22},
	title = {A recursion and a combinatorial formula for {J}ack polynomials},
	url = {http://dx.doi.org/10.1007/s002220050134},
	volume = {128},
	year = {1997}}

@article{LoganShepp1977,
	author = {Logan, B. F. and Shepp, L. A.},
	fjournal = {Advances in Mathematics},
	journal = {Advances in Math.},
	mrclass = {05E10 (49Q10 60C05)},
	mrnumber = {MR1417317 (98e:05108)},
	mrreviewer = {Graham Brightwell},
	number = {2},
	pages = {206--222},
	title = {{A variational problem for random {Y}oung tableaux}},
	volume = {26},
	year = {1977}}

@article{MergnyPotters2022,
	author = {Mergny, Pierre and Potters, Marc},
	journal = {SciPost Physics},
	number = {1},
	pages = {022},
	title = {Rank one {HCIZ} at high temperature: interpolating between classical and free convolutions},
	volume = {12},
	year = {2022}}

@article{Okounkov2000,
	author = {Okounkov, Andrei},
	doi = {10.1155/S1073792800000532},
	fjournal = {International Mathematics Research Notices},
	issn = {1073-7928},
	journal = {Internat. Math. Res. Notices},
	mrclass = {15A52 (05C30 60C05 82B31)},
	mrnumber = {1802530},
	mrreviewer = {Oleksiy Khorunzhiy},
	number = {20},
	pages = {1043--1095},
	title = {Random matrices and random permutations},
	url = {https://doi.org/10.1155/S1073792800000532},
	year = {2000}}

@article{Okounkov2001,
	author = {Okounkov, Andrei},
	coden = {SMATF6},
	fjournal = {Selecta Mathematica. New Series},
	issn = {1022-1824},
	journal = {Selecta Math. (N.S.)},
	mrclass = {60C05},
	mrnumber = {MR1856553 (2002f:60019)},
	mrreviewer = {Hajime Yamato},
	number = {1},
	pages = {57--81},
	title = {{Infinite wedge and random partitions}},
	volume = {7},
	year = {2001}}

@article{Speicher1994,
	author = {Speicher, R.},
	issn = {0025-5831},
	journal = {Mathematische Annalen},
	number = {1},
	pages = {611--628},
	publisher = {Springer},
	title = {{Multiplicative functions on the lattice of non-crossing partitions and free convolution}},
	volume = {298},
	year = {1994}}

@article{Stanley1989,
	author = {Stanley, Richard P.},
	coden = {ADMTA4},
	doi = {10.1016/0001-8708(89)90015-7},
	fjournal = {Advances in Mathematics},
	issn = {0001-8708},
	journal = {Adv. Math.},
	mrclass = {05A15 (20C30)},
	mrnumber = {1014073 (90g:05020)},
	mrreviewer = {Dennis White},
	number = {1},
	pages = {76--115},
	title = {{Some combinatorial properties of {J}ack symmetric functions}},
	url = {http://dx.doi.org/10.1016/0001-8708(89)90015-7},
	volume = {77},
	year = {1989}}

@article{VershikKerov1977,
	author = {Vershik, A. M. and Kerov, S. V.},
	fjournal = {Doklady Akademii Nauk SSSR},
	issn = {0002-3264},
	journal = {Dokl. Akad. Nauk SSSR},
	mrclass = {10J20 (20C30)},
	mrnumber = {0480398 (58 \#562)},
	mrreviewer = {V. M. Maksimov},
	number = {6},
	pages = {1024--1027},
	title = {{Asymptotic behavior of the {P}lancherel measure of the symmetric group and the limit form of {Y}oung tableaux}},
	volume = {233},
	year = {1977}}

@article{Voiculescu1986,
	author = {Voiculescu, Dan},
	coden = {JFUAAW},
	fjournal = {Journal of Functional Analysis},
	issn = {0022-1236},
	journal = {J. Funct. Anal.},
	mrclass = {46L50 (22E30)},
	mrnumber = {MR839105 (87j:46122)},
	mrreviewer = {D. Petz},
	number = {3},
	pages = {323--346},
	title = {{Addition of certain noncommuting random variables}},
	volume = {66},
	year = {1986}}

@article{Voiculescu1991,
	author = {Voiculescu, Dan},
	coden = {INVMBH},
	fjournal = {Inventiones Mathematicae},
	issn = {0020-9910},
	journal = {Invent. Math.},
	mrclass = {46L50 (60F99)},
	mrnumber = {MR1094052 (92d:46163)},
	mrreviewer = {Sh. A. Ayupov},
	number = {1},
	pages = {201--220},
	title = {{Limit laws for random matrices and free products}},
	volume = {104},
	year = {1991}}

@book{Macdonald1995,
	address = {New York},
	author = {Macdonald, I. G.},
	edition = {Second},
	isbn = {0-19-853489-2},
	mrclass = {05E05 (05-02 20C30 20C33 20K01 33C80 33D80)},
	mrnumber = {1354144},
	mrreviewer = {John R. Stembridge},
	note = {With contributions by A. Zelevinsky, Oxford Science Publications},
	pages = {x+475},
	publisher = {The Clarendon Press Oxford University Press},
	series = {{Oxford Mathematical Monographs}},
	title = {{Symmetric functions and {H}all polynomials}},
	year = {1995}}

@book{NicaSpeicherBook,
	address = {Cambridge},
	author = {Nica, Alexandru and Speicher, Roland},
	doi = {10.1017/CBO9780511735127},
	isbn = {978-0-521-85852-6; 0-521-85852-6},
	mrclass = {46L54 (46L53 60C05 81S25)},
	mrnumber = {2266879 (2008k:46198)},
	mrreviewer = {Todd Kemp},
	pages = {xvi+417},
	publisher = {Cambridge University Press},
	series = {{London Mathematical Society Lecture Note Series}},
	title = {{Lectures on the combinatorics of free probability}},
	url = {http://dx.doi.org/10.1017/CBO9780511735127},
	volume = {335},
	year = {2006}}

@inproceedings{Okounkov2003,
	author = {Okounkov, Andrei},
	booktitle = {{Fourteenth International Congress on Mathematical Physics}},
	pages = {379--403},
	publisher = {Word Scientists},
	title = {{The uses of random partitions}},
	year = {2003}}

@book{Stanley:EC2,
	address = {Cambridge},
	author = {Stanley, R. P.},
	doi = {10.1017/CBO9780511609589},
	isbn = {0-521-56069-1; 0-521-78987-7},
	mrclass = {05A15 (05-02 05E05 05E10 68R05)},
	mrnumber = {1676282 (2000k:05026)},
	mrreviewer = {Ira Gessel},
	note = {With a foreword by Gian-Carlo Rota and appendix 1 by Sergey Fomin},
	pages = {xii+581},
	publisher = {Cambridge University Press},
	series = {Cambridge Studies in Advanced Mathematics},
	title = {Enumerative combinatorics. {V}ol. 2},
	url = {http://dx.doi.org/10.1017/CBO9780511609589},
	volume = {62},
	year = {1999}}

@article{MazzucaMemin2025,
	author = {Mazzuca, G. and Memin, R.},
	doi = {10.1007/s00023-024-01435-0},
	journal = {Ann. Henri Poincar\'e},
	number = {1},
	pages = {245--316},
	title = {{CLT} for {$\beta$}-Ensembles at High Temperature and for Integrable Systems: A Transfer Operator Approach},
	volume = {26},
	year = {2025}}

\end{document}